\documentclass[11pt]{article}

\usepackage[margin=1.3in]{geometry}

\usepackage{graphicx}
\usepackage{amsmath}
\usepackage{amssymb}
\usepackage{amsthm}
\usepackage{mathrsfs}
\usepackage{textcomp}
\usepackage{esint}
\usepackage{xcolor}
\usepackage[colorlinks=true]{hyperref}
\hypersetup{
    colorlinks=blue,%
    citecolor=red,%
    filecolor=black,%
    linkcolor=blue,%
    urlcolor=blue
}
\usepackage[colorinlistoftodos]{todonotes}

\newtheorem{remark}{Remark}[section]
\newtheorem{theorem}{Theorem}[section]
\newtheorem{lemma}[theorem]{Lemma}
\newtheorem{cor}[theorem]{Corollary}
\newtheorem{prop}[theorem]{Proposition}
\newtheorem*{claim}{Claim}
\newtheorem*{convention}{\textbf{Convention}}

\def \no{\nonumber}
\newcommand{\R}{\mathbb{R}}
\newcommand{\ud}{\mathrm{d}}
\newcommand{\Sn}{\mathbb{S}^n}
\newcommand{\al}{\alpha}                
\newcommand{\lda}{\lambda}
\newcommand{\om}{\Omega}                
\newcommand{\pa}{\partial}
\newcommand{\va}{\varepsilon}

\numberwithin{equation}{section}

\makeatletter

\newdimen\bibspace
\renewenvironment{thebibliography}[1]{%
 \section*{\refname 
       \@mkboth{\MakeUppercase\refname}{\MakeUppercase\refname}}%
     \list{\@biblabel{\@arabic\c@enumiv}}%
          {\settowidth\labelwidth{\@biblabel{#1}}%
           \leftmargin\labelwidth
           \advance\leftmargin\labelsep
           \itemsep\bibspace
           \parsep\z@skip     %
           \@openbib@code
           \usecounter{enumiv}%
           \let\p@enumiv\@empty
           \renewcommand\theenumiv{\@arabic\c@enumiv}}%
     \sloppy\clubpenalty4000\widowpenalty4000%
     \sfcode`\.\@m}
    {\def\@noitemerr
      {\@latex@warning{Empty `thebibliography' environment}}%
     \endlist}

\makeatother

\makeatletter
 \allowdisplaybreaks

\begin{document}

\title{A fractional conformal curvature flow on the unit sphere}

\author{\medskip Xuezhang Chen~ and~ Pak Tung Ho}

\date{}

\maketitle

\begin{abstract}

We study a fractional conformal curvature flow on the standard unit sphere and prove a  perturbation result of the fractional Nirenberg problem with fractional exponent $\sigma \in (1/2,1)$.  This extends the result of Chen-Xu (Invent. Math. 187, no. 2, 395-506, 2012) for the scalar curvature flow on the standard unit sphere. 

\medskip 

{{\bf $\mathbf{2020}$ MSC:} 35R09,35B44,35K55 (53C21,58E05,58K55)}

\medskip 


\end{abstract}

\section{Introduction}\label{Sect:Intro}

Let $(\Sn,g_{\Sn})$ be the unit sphere equipped with the standard round metric  $g_{\Sn}$ with $n\ge 2$. Denote by $[g_{\Sn}]=\{\rho g_{\Sn}; 0<\rho\in C^\infty(\Sn) \}$ the conformal class of $g_{\Sn}$.  By viewing $(\Sn,[g_{\Sn}])$ as conformal infinity of the Poincar\'e ball, Graham-Zworski \cite{GZ} established that  there is a family of conformally covariant (pseudo-)differential  operators $P^g_\sigma$ for $g\in [g_{\Sn}]$, where $\sigma\in (0,n/2)$ and the principle symbol of $P_\sigma^g$ is $|\xi|^{2\sigma}$.  These operators satisfy the following conformal invariance: 
\begin{equation}\label{conformal_invariance}
P_\sigma^{\rho^{\frac{4}{n-2\sigma}}g}(\phi)=\rho^{-\frac{n+2\sigma}{n-2\sigma}}P^g_\sigma(\rho \phi),\quad \forall ~\rho, \phi \in C^\infty(\Sn) \quad \mathrm{with~~}\rho>0.
\end{equation}
We call
\begin{equation}\label{1.6}
R_\sigma^g=P_\sigma^g(1)
\end{equation}
  $Q$-curvature of order $2 \sigma$. Up to positive constants, $R_1^g$ is the scalar curvature and $R_2^g$ is the fourth order  $Q$-curvature of Paneitz \cite{Paneitz} and Branson \cite{Bran85}, respectively. We refer to Fefferman-Graham \cite{FG85, FG}, Graham-Jenne-Mason-Sparling \cite{GJMS},   Gover-Peterson \cite{GP}, Juhl \cite{Juhl} and references therein for the construction of conformally covariant operators on manifolds and $Q$-curvatures related to them. 

The operator $P_\sigma:=P^{g_{\Sn}}_\sigma$ is referred to the intertwining operator,  see Beckner \cite{Beckner}, Branson \cite{Bran} and Morpurgo \cite{Morpurgo}.  It can be written as \begin{equation}\label{1.2}
P_\sigma=\frac{\Gamma(B+\frac{1}{2}+\sigma)}{\Gamma(B+\frac{1}{2}-\sigma)}\quad\mathrm{ and }\quad B=\sqrt{-\Delta_{g_{\Sn}}+\left(\frac{n-1}{2}\right)^2},
\end{equation}
where $\Gamma$ is the Gamma function and $\Delta_{g_{\Sn}}$ is the Laplace-Beltrami operator of $g_{\Sn}$.  More precisely, $B$ and $P_\sigma$ are determined by the formulas 
\[
B(Y^{(k)})= \left(k+\frac{n-1}{2}\right) Y^{(k)} \quad \mathrm{and}\quad P_\sigma(Y^{(k)})=\frac{\Gamma(k+\frac{n}{2}+\sigma)}{\Gamma(k+\frac{n}{2}-\sigma)} Y^{(k)}
\] for every spherical harmonic polynomial $Y^{(k)}$ of degree $k\ge 0$.  Furthermore, $P_\sigma$ is the pull back of the fractional Laplacian $(-\Delta)^\sigma$ on $\R^n$ via the stereographic projection through
\begin{equation} \label{eq:r1}
(P_\sigma(\phi))\circ \Psi=  (\det \ud\Psi)^{-\frac{n+2\sigma}{2n}}(-\Delta)^\sigma( (\det \ud\Psi)^{\frac{n-2\sigma}{2n}}\phi\circ \Psi)\quad \mathrm{for~~}\phi\in C^2(\Sn),
\end{equation}
where $\Psi: \mathbb{R}^n \to \Sn$ is the inverse of the stereographic projection from the south pole $S$ and $\Psi^\ast(g_{\Sn})=(\det \ud\Psi)^{2/n}g_{\Sn}$.  When $\sigma\in (0,1)$, Pavlov and Samko \cite{Pavlov&Samko} proved that
\begin{equation}\label{1.2.4}
P_\sigma(v)(x)
=c_{n,-\sigma}\int_{\Sn}\frac{v(x)-v(y)}{|x-y|^{n+2\sigma}}\ud V_{g_{\Sn}}(y)+R_\sigma v(x) \quad \mathrm{for~~}v\in C^2(\Sn),
\end{equation}
 where $|x-y|$ is the Euclidean distance in $\R^{n+1}$ between $x$ and $y$,  $c_{n,-\sigma}=\frac{2^{2\sigma}\sigma\Gamma(\frac{n+2\sigma}{2})}{\pi^{\frac{n}{2}}\Gamma(1-\sigma)}$ and $ R_\sigma=P_\sigma(1)=\frac{\Gamma(\frac{n}{2}+\sigma)}{\Gamma(\frac{n}{2}-\sigma)}$ is the  $Q$-curvature of  order $2\sigma$ with respect to $g_{\Sn}$. 
 
In a series of papers \cite{Jin&Li&Xiong1,Jin&Li&Xiong2,Jin&Li&Xiong3},  Jin-Li-Xiong studied the prescribing fractional $Q$-curvature problem on $\Sn$ for $\sigma\in (0,n/2)$, equivalently, the fractional Nirenberg problem, generalizing the classical Nirenberg problem  ($\sigma=1$). This  problem is equivalent to solving
\begin{equation}\label{1.2.5}
P_\sigma(v)=fv^{\frac{n+2\sigma}{n-2\sigma}} ,\quad  v>0 \quad\mathrm{on~~}\Sn ,
\end{equation}
where  $f$ is a given continuous function on $\Sn$.  When $\sigma=1/2$ and $2$, it recovers the prescribing mean curvature problem and Paneitz-Branson $Q$-curvature problem, respectively; see  \cite{Jin&Li&Xiong1, Jin&Li&Xiong3} for brief reviews of the classical Nirenberg problem and its generalizations, as well as references in this field. Other studies on the fractional Nirenberg problem include Escobar-Garcia \cite{EG}, Chen-Zheng \cite{CZ},  Abdelhedi-Chtioui-Hajaiej \cite{ACH}, Chen-Liu-Zheng \cite{CLZ}, Guo-Nie-Niu-Tang \cite{G+},  Liu-Ren \cite{LR}, Niu-Tang-Wang \cite{NTW}, etc.   The limiting case $\sigma=n/2$ is of particular interest; see Moser \cite{Moser}, Chang-Yang \cite{CY87}, Wei-Xu \cite{WX98, WX09}, Brendle \cite{Brendle1, Brendle2}, Da Lio-Martinazzi-Rivi\'ere \cite{DMR}, etc. When $\sigma\in (0,1)$, \eqref{1.2.5} is related to the fractional Yamabe problem, which has been studied by Gonz\'alez-Qing \cite{GQ}, Gonz\'alez-Wang \cite{GW}, Kim-Musso-Wei \cite{KMW}, Ndiaye-Sire-Sun \cite{NSS}, Mayer-Ndiaye \cite{MN}, etc.

In this paper, we are interested in a flow approach to  the fractional Nirenberg problem. The flow approach has been studied in the limiting case ($\sigma=n/2$)  by  Brendle \cite{Brendle2,Brendle4},  Struwe  \cite{S}, Malchiodi-Struwe \cite{MS} and Chen-Xu \cite{Chen&Xu1}, Ho \cite{Ho,Ho2,Ho3}, etc. The scalar curvature flow was firstly introduced by Chen-Xu \cite{Chen&Xu}, to which this paper is close.  Since the heat kernel of $\pa_t +(-\Delta)^{\sigma} $ changes signs when $\sigma>1$, we confine our present investigation to the range  $\sigma\in (0,1)$. If $\sigma>1$,  one might consider some nonlocal  flows as Baird-Fardoun-Regbaoui \cite{BFR} and Gursky-Malchiodi \cite{GM} did for the fourth order $Q$-curvature problem. 

For any positive smooth function  $f $ on $\Sn$, we study the Cauchy problem
\begin{align}\label{1.3} 
\begin{cases}
~~~\frac{\partial g}{\partial t}=-(R_\sigma^g-\alpha f)g \quad \mathrm{on~~}\Sn \times (0,\infty),  \\ 
 g(0)=g_0\in [g_{\Sn}],
\end{cases}
\end{align}
where
\begin{equation}\label{1.4}
\alpha(t)=\frac{\int_{\Sn}R_\sigma^g\ud V_g}{\int_{\Sn}f\ud V_g}.
\end{equation}  
The above evolution equation is a negative gradient flow of the normalized total fractional $Q$-curvature functional
\[
\mathcal{S}(g)= \frac{\int_{\Sn}R_\sigma^g \ud V_{g}}{(\int_{\Sn}f\ud V_{g})^{\frac{n-2\sigma}{n}}} \quad \mathrm{for~~} g\in [g_{\Sn}].
\]
If $\sigma=1$, it is the scalar curvature flow initially  studied by Chen-Xu \cite{Chen&Xu}; see also Mayer \cite{Mayer}. 
 If $f=1$ and $\sigma \in (0,1)$, it is the  fractional Yamabe flow  studied by Jin-Xiong \cite{Jin&Xiong}; see also Daskalopoulos-Sire-V\'azquez \cite{DSV} and Chan-Sire-Sun \cite{CSS} on manifolds. Using the localization formula of Caffarelli-Silvestre \cite{CaS} or Chang-Gonz\'alez \cite{Chang-G}, this flow with $\sigma=1/2$  coincides with the one of prescribing mean curvature in the Euclidean unit  ball by Xu-Zhang \cite{XZ}.  
 
 If  we write $g(t)= u(t)^{4/(n-2\sigma)} g_{\Sn}$ with $u(0)=u_0\in C^\infty(\Sn)$ being a positive smooth function, then \eqref{1.3} becomes
\begin{align}\label{1.7} 
\begin{cases}
~~~\frac{\partial u}{\partial t}=-\frac{n-2\sigma}{4}(R_\sigma^g-\alpha f)u \quad \mathrm{on~~}\Sn \times (0,T),\\  
u(0)=u_0,
\end{cases}
\end{align}
where $T \leq \infty$.
By \eqref{1.6}, the above first equation becomes
\begin{equation}\label{1.8}
\frac{4}{n+2\sigma}\frac{\partial }{\partial t} u^{\frac{n+2\sigma}{n-2\sigma}}=
-P_\sigma(u)+\alpha f u^{\frac{n+2\sigma}{n-2\sigma}}\quad \mathrm{on~~}\Sn\times (0,T). 
\end{equation}

Our first theorem asserts that the Cauchy problem has a unique global solution. 

\begin{theorem}\label{thm:1} Suppose  that $\sigma\in (0,1)$ and $u_0, f \in C^{\infty}(\Sn)$ are positive functions. Then for any $0<T<\infty$, the Cauchy problem \eqref{1.7} has a unique smooth positive solution in $\Sn\times (0,T]$. 
 \end{theorem} 

Due to the Kazdan-Warner obstruction \cite{Jin&Li&Xiong1} to the existence of  positive smooth solutions of \eqref{1.2.5},  solutions of  \eqref{1.7} are not necessarily uniformly bounded. In particular, it is the case  when $f=2+x_{n+1}$. Through establishing a Stroock-Varopoulos type inequality, we apply a similar argument in  Schwetlick-Struwe \cite{Schwetlick&Struwe}, Brendle \cite{Brendle3} and Chen-Xu \cite{Chen&Xu} to derive evolution  equations of fractional $Q$-curvature and its $L^p(\Sn,g)$-convergence for all $p\ge 1$, as well as its $H^\sigma(\Sn,g)$- convergence.  

The next theorem is a natural extension of Chen-Xu \cite{Chen&Xu} to the fractional Nirenberg problem. 

\begin{theorem}\label{main_thm} 
Suppose  that $\sigma \in (1/2,1)$ and $f \in C^{\infty}(\Sn)$ is a positive Morse function.  Assume that:
\begin{itemize} 
\item[(i)] 
\begin{equation}\label{sbc}
\max_{\Sn}f\big/\min_{\Sn}f<2^{\frac{2\sigma}{n-2\sigma}};
\end{equation} 
\item[(ii)] 
$$|\nabla f|_{g_{\Sn}}^2+|\Delta_{g_{\Sn}}f|^2 \neq 0 \quad \mathrm{on~~} \Sn;$$ 
\item[(iii)] For any integer $0\le i\le n$,  denote 
\begin{equation}\label{gamma_j}
\gamma_i=\# \{\theta\in \Sn; \nabla f(\theta)=0, \Delta_{\Sn} f(\theta)<0, \mathrm{ind} (f,\theta)=n-i\},
\end{equation}
where $\mathrm{ind} (f,\theta)$ stands for the Morse index of $f$ at the critical point $\theta$. 
There are no nonnegative constants  $k_0, k_1,\cdots, k_n$ satisfying
\begin{equation}\label{assump:Morse_Sys_Cond}
\gamma_0=1+k_0, \quad \gamma_i=k_{i-1}+k_i\quad\mathrm{ for~~}1\leq i\leq n, \quad k_n=0.
\end{equation}
\end{itemize} 
Then there exists at least one positive smooth solution of  \eqref{1.2.5}.
\end{theorem}
We refer to \cite{Chen&Xu} for more comments on the condition \textit{(iii)} in Theorem \ref{main_thm}. Condition \textit{(i)} guarantees that only single bubble occurs in the following blow-up analysis. Due to the same reason, we call Theorem \ref{main_thm}  a perturbation result of the fractional Nirenberg problem.   The perturbation theorems for the Nirenberg problem was proved by Chang-Yang \cite{CY}, Li \cite{L}, Malchiodi \cite{M},  Ji \cite{J1,J2} and references therein. Jin-Li-Xiong \cite{Jin&Li&Xiong2}  proved a perturbation theorem for the fractional Nirenberg problem  when $\sigma\in (1/2,1)$. Jin-Li-Xiong \cite{Jin&Li&Xiong3}  established uniform a priori estimates with respect to $\sigma\in (0,n/2)$ and a degree argument deforming  $\sigma$ to $\sigma=1$ allows them to avoid proving perturbation theorems for the fractional Nirenberg problem.  Moreover, conditions \textit{(i)} and \textit{(iii)} together avoid concentration of any sequence of flow solutions.

In comparison to \cite{Chen&Xu}, we need to overcome several difficulties stemming from the nonlocality. We will handle some useful  nonlocal ingredients in the consequent sections.  During the long journey to the proof of our main theorem, we need several new insights,  see, e.g.,  the proofs of Lemma \ref{lemma5.4} and  the characterization of homotopy types of critical points of $f$ in the Morse theory part in Proposition \ref{proposition7.1} \textit{(i)} etc. Especially, in order to show that a certain sub-level set of energy $\mathcal{S}(g)$ is contractible in Proposition \ref{proposition7.1} \textit{(i)}, our new strategy is to construct a relatively simpler but new homotopy  and develop new estimates to quantify a gap between a normalized flow $v(t)$ and its asymptotic limit $1$ in the $C^1(\Sn)$-norm, precisely, there exists a positive constant $\va_0$, depending only on $n,\sigma$, such that $\|v(T_1)-1\|_{C^1(\Sn)}<\va_0$ for some finite time $T_1=T_1(u_0)>0$, where $u_0$ is the initial datum of the flow. This enables us to make the proof more transparent.

The paper is organized as follows. In Section \ref{Sect:Prelim}, we establish an extension formula of  $P_\sigma$ in the unit ball $\overline{B_1}$ of $\R^{n+1}$ and a Stroock-Varopoulos type inequality, which is crucial in the proof of the asymptotic behaviors of $R_\sigma^{g(t)}$, as well as other elementary estimates. In Section \ref{Sect:Property_flow}, we collect some basic properties of the fractional $Q$-curvature flow \eqref{1.3}. In Section \ref{Sect:Integ_curv_est}, we prove $L^p$ and $H^\sigma$ estimates of $\alpha f-R_\sigma^{g(t)}$ with respect to the flow metric $g(t)$.  The purpose of Section \ref{Sect:Blow-up_anal} is devoted to the blow-up analysis for any time sequence of solutions to the flow \eqref{1.3} and establish a compactness-concentration lemma, which shows that if concentration phenomenon occurs, then there exists only single bubble under condition \eqref{sbc}. Starting from Section \ref{Sect:Finite_dim_dynamics}, we restrict ourselves to $\sigma \in (1/2,1)$. Part of the reason is that when $0<\sigma< 1/2$,  the compact embedding of  $H^{2\sigma,p}(\Sn)\hookrightarrow C^1(\Sn)$ is missed, especially for a normalized flow $v(t)$; see \eqref{def:fractional_Sobolev_space} below for the definition of the space $H^{2\sigma,p}(\Sn)$. However, this is crucial to the rest part of this scheme.  Subsequently, we adopt a contradiction argument, suppose that the flow diverges for any initial data, then we derive that  $v(t)\to 1$ in $C^{1,\lambda}(\Sn)$ for suitable $\lambda\in (0,1)$. In Section \ref{Sect:Existence}, through some elaborate efforts to characterize the homotopy type of an appropriate sub-level set of $\mathcal{S}(g)$, we reach a contradiction with condition \textit{(iii)} in Theorem \ref{main_thm}.

\vspace{0.2cm}
\noindent{\bf Acknowledgments.} 
Chen's research was supported by NSFC (No.11771204), A Foundation for the Author of National Excellent Doctoral Dissertation of China (No.201417) and start-up grant of 2016 Deng Feng program B at Nanjing University. Ho's research was supported by Basic Science Research Program through the National Research Foundation of Korea (NRF) funded by the Ministry of Education, Science and Technology (2019041021), and by Korea Institute for Advanced Study (KIAS) grant funded by the Korea government (MSIP). 

\bigskip

\section{Preliminaries}\label{Sect:Prelim}

For  $\sigma \in (0,1)$, we define $H^\sigma(\Sn)$ as the fractional Sobolev space with the norm
\[
\|u\|_{H^\sigma(\Sn)}= \left(\int_{\Sn} u P_\sigma u\,\ud V_{g_{\Sn}}\right)^{1/2}.
\]
The Sobolev inequality asserts that
\begin{equation}\label{ineq:Sobolev}
Y_\sigma(\Sn)\left(\int_{\Sn} |u|^{\frac{2n}{n-2\sigma}}\,\ud V_{g_{\Sn}} \right)^{\frac{n-2\sigma}{n}} \le  \int_{\Sn} u P_\sigma u\,\ud V_{g_{\Sn}} \quad \mathrm{for~~}u\in H^\sigma (\Sn),
\end{equation}
with equality if and only if $$u=(\det \ud \phi)^{\frac{n-2\sigma}{2n}}$$ 
for some conformal transformation $\phi$ on $\Sn$, up to a nonzero constant multiple. See Beckner \cite[Theorem 6]{Beckner} together with \eqref{eq:r1}. By \eqref{ineq:Sobolev} and  \eqref{conformal_invariance} , for $g\in [g_{\Sn}]$ there holds
\begin{equation}\label{eq:sobolev}
Y_\sigma(\Sn)\left(\int_{\Sn} |u|^{\frac{2n}{n-2\sigma}}\,\ud V_g \right)^{\frac{n-2\sigma}{n}} \le  \int_{\Sn} u P_\sigma^g u\,\ud V_g \quad \mathrm{for~~}u\in H^\sigma (\Sn,g).
\end{equation}

 We are going to derive an extension formula for $P_\sigma^g$ in the unit ball $\overline {B_1}\subset \R^{n+1}$, which is of independent interest. Let $g=u^{4/(n-2\sigma)} g_{\Sn}$ and $u\in C^\infty(\Sn)$ be positive, and $v\in C^\infty (\Sn)$. 
\medskip

\emph{Step 1.} An extension formula in  $\R^{n+1}_+:=\{x=(x', x_{n+1}) \in \R^{n+1}; x_{n+1}>0\}$ equipped with an appropriate Riemannian metric.

\medskip

Denote $e_{n+1}=(0,\cdots, 0, 1)\in \R^{n+1}$. Let $\bar \Psi: \R^{n+1}_+\to B_1$ be an inversion with respect to the sphere $\pa B_{\sqrt 2}(-e_{n+1})$, which is a conformal map from $\mathbb{R}^{n+1}_+$ to $B_1$, explicitly,
\[
\bar \Psi (x)=\frac{2(x+e_{n+1})}{|x+e_{n+1}|^2}-e_{n+1}:=y \in B_1.
\]
In particular, $\Psi:=\bar \Psi|_{\pa \R^{n+1}_+}: \pa \mathbb{R}^{n+1}_+ \to \Sn=\pa B_1$ is the inverse of the stereographic projection from the south pole $S=-e_{n+1}$, and its inverse is given by
\[
x=\bar \Psi^{-1}(y)= \frac{2(y+e_{n+1})}{|y+e_{n+1}|^2}-e_{n+1}.
\]
Thus, we obtain
 \begin{equation}\label{eq:extension1}
 x_{n+1}= \frac{1-|y|^2} {|y+e_{n+1}|^2} \mathrm{~~and~~} |J_{ \bar \Psi}(x)|:=\left(\frac{\sqrt{2}}{|x+e_{n+1}|}\right)^{2(n+1)}=\left(\frac{\sqrt{2}}{|y+e_{n+1}|}\right)^{-2(n+1)}.
\end{equation}

We now set $w(x')=\det \ud \Psi(x')^{\frac{n-2\sigma}{2n}} u\circ \Psi(x')$, where
\begin{equation}\label{eq:extension2}
\det \ud \Psi(x')=\left(\frac{\sqrt{2}}{|x'+e_{n+1}|}\right)^{2n}=\left(\frac{\sqrt{2}}{|y+e_{n+1}|}\right)^{-2n}
\end{equation}
for $x' \in \pa \R^{n+1}_+$ and $y=\Psi(x') \in \Sn=\pa B_1$.
In other words, 
$$w^{\frac{4}{n-2\sigma}}|\ud x'|^2=\Psi^\ast(u^{\frac{4}{n-2\sigma}}g_{\Sn}).$$
For $x\in \mathbb{R}_+^{n+1}$, we define
\begin{equation} \label{eq:poisson}
W(x)=\mathcal{P}_\sigma[w](x):= \beta(n,\sigma)\int_{\R^n} \frac{x_{n+1}^{2\sigma}}{(|x'-y'|^2+x_{n+1}^2)^{\frac{n+2\sigma}{2}}} w(y')\,\mathrm{d} y'
\end{equation}
and
\begin{equation*}
\xi:=\mathcal{P}_\sigma[w v\circ \Psi],
 \end{equation*}
 where $\beta(n,\sigma)$ is a positive constant defined in Theorem \ref{thm:lf} below.  
 Then $W(x',0)=w(x')$, $\xi(x',0)=w(x')v(\Psi(x'))$ and $W(x)>0$ in $\overline{\mathbb{R}_+^{n+1}}$. Since $u$ and $v$ are smooth on $\Sn$, it follows from the regularity theory (see \cite{Jin&Li&Xiong1}) that
\[
W_{0,1},~ x_{n+1}^{1-2\sigma}\partial_{x_{n+1}} W_{0,1},~ \xi_{0,1}, ~x_{n+1}^{1-2\sigma}\partial_{x_{n+1}} \xi_{0,1} \in C^0(B_1\times [0,1]),
\]
where
\[
W_{0,1}(x)=|x|^{2\sigma-n}W(\frac{x}{|x|^2})
\]
and $\xi_{0,1}$ is defined in the same way. Hence in the following computations, we can always identify $\R^{n+1}_+$ with the unit ball of $\R^{n+1}$.

By Caffarelli-Silvestre \cite{CaS}, $W $ satisfies
\begin{align*}
\mathrm{div}(x_{n+1}^{1-2\sigma} \nabla_{x} W)=0 \quad \mathrm{in~~}\mathbb{R}^{n+1}_+ \mathrm{~~and~~} -\lim_{x_{n+1}\to 0} x_{n+1}^{1-2\sigma} \partial_{x_{n+1}} W=N(\sigma)(-\Delta)^\sigma W(x',0),
\end{align*}
where $N(\sigma)=2^{1-2\sigma}\Gamma(1-\sigma)/\Gamma(\sigma)$, as well as does $\xi$.
Let us introduce a Riemannian metric
\[
\bar g= W(x)^{\frac{4}{n-2\sigma}} \, |\ud x|^2 \quad \mathrm{~~in~~}\R^{n+1}_+.
\]
It follows from the proof of \cite[Proposition 3.2]{JX} or a direct computation that for $a=\frac{2-4\sigma}{n-2\sigma}$, there hold
\begin{align} \label{eq:localformula_extension-1}
\mathrm{div}_{\bar g}\left(x_{n+1}^{1-2\sigma} W^{a} \nabla_{\bar g} \frac{\xi}{W}\right)= 0 \quad \mathrm{~~in~~}\R^{n+1}_+
\end{align}
and
\begin{align}\label{eq:product_fractional_Laplace}
&\lim_{x_{n+1}\to 0}x_{n+1}^{1-2\sigma} W^{a}\frac{\partial}{\partial \nu_{\bar g}} \frac{\xi}{W}=-\lim_{x_{n+1}\to 0}x_{n+1}^{1-2\sigma} W^{a-\frac{2}{n-2\sigma}}\frac{\partial}{\partial x_{n+1}} \frac{\xi}{W}\no\\
=&N(\sigma)w^{-\frac{n+2\sigma}{n-2\sigma}}(-\Delta)^\sigma(wv\circ \Psi) -N(\sigma) v\circ \Psi w^{-\frac{n+2\sigma}{n-2\sigma}}(-\Delta)^\sigma w.
\end{align}
Under stereographic projection coordinates, it follows from \eqref{conformal_invariance} and \eqref{eq:r1} that
\[
P^g_\sigma(\phi)\circ \Psi=w^{-\frac{n+2\sigma}{n-2\sigma}}(-\Delta)^\sigma(w \phi \circ \Psi),  \quad \forall ~\phi\in C^\infty(\Sn).
\]
This together with \eqref{eq:product_fractional_Laplace} and \eqref{eq:r1} yields
\begin{equation} \label{eq:localformula_extension-2}
P^g_\sigma(v)\circ \Psi= \frac{1}{N(\sigma)}\lim_{x_{n+1}\to 0}x_{n+1}^{1-2\sigma} W^{a}\frac{\partial}{\partial \nu_{\bar g}} \frac{\xi}{W} + (R_\sigma^g v)\circ \Psi.
\end{equation}

\medskip

\emph{Step 2.} Express formulae \eqref{eq:localformula_extension-1}, \eqref{eq:localformula_extension-2} in $\overline {B_1}$.

\medskip
Notice that the push-forward metric of $(\overline{B_1}, \bar g)$ by the map $\bar \Psi$ is given by
\begin{align*}
\tilde g:=&\bar \Psi_\ast (\bar g) = (W\circ \bar \Psi^{-1})^{\frac{4}{n-2\sigma}} |J_{\bar \Psi^{-1}}|^{\frac{2}{n+1}}\, |\ud y|^2 \\
=&\left(W\circ \bar \Psi^{-1} |J_{ \bar \Psi^{-1}}|^{\frac{n-2\sigma}{2(n+1)}}\right)^{\frac{4}{n-2\sigma}} \,|\ud y|^2 :=\tilde W^{\frac{4}{n-2\sigma}} \,|\ud y|^2.
\end{align*}
This together with \eqref{eq:extension1} implies that $|J_{\bar \Psi^{-1}}|=(\frac{\sqrt{2}}{|y+e_{n+1}|})^{2(n+1)}$ and
\begin{align*}
(x_{n+1}^{1-2\sigma} W^{a})\circ \bar \Psi^{-1}(y)&= (1-|y|^2)^{1-2\sigma} |y+e_{n+1}|^{-2(1-2\sigma)} (W\circ \bar \Psi^{-1}(y))^{\frac{2-4\sigma}{n-2\sigma}} \\&
=2^{2\sigma-1}(1-|y|^2)^{1-2\sigma} \tilde W(y)^a.
\end{align*}
In terms of the variable $y$, \eqref{eq:localformula_extension-1} and \eqref{eq:localformula_extension-2} become
\begin{align*}
&\mathrm{div}_{\tilde g}\left((1-|y|^2)^{1-2\sigma} \tilde W^{a} \nabla_{\tilde g} \frac{\tilde \xi}{\tilde W}\right)= 0 \quad \mathrm{in~~}B_1,\\
&P^g_\sigma(v)(x)= \frac{2^{2\sigma-1}}{N(\sigma)}\lim_{y\to x}(1-|y|^2)^{1-2\sigma} \tilde W^{a}  \frac{\partial}{\partial \nu_{\tilde g}} \frac{\tilde \xi}{\tilde W} +  R_\sigma^g v \quad \mathrm{for~~any~~}x\in \Sn,
\end{align*}
where $\lim_{y\to x}$ is understood as $y=\tau x$ with $\tau \nearrow 1$, and  $\tilde \xi= \xi\circ \bar \Psi^{-1} |J_{\bar \Psi^{-1}}|^{\frac{n-2\sigma}{2(n+1)}} $. For $x\in B_1$ and $y\in \pa B_1$, by  \eqref{eq:poisson} we have
\begin{align*}
&W(\bar \Psi^{-1}(x))= \beta(n,\sigma) \int_{\Sn} \frac{(1-|x|^2)^{2\sigma}|x+e_{n+1}|^{-4\sigma}}{| \bar \Psi^{-1} (x)- \bar \Psi^{-1}(y)|^{n+2\sigma}} w(\bar \Psi^{-1}(y)) |\left. J_{\bar \Psi^{-1}}\right|_{\Sn}(y)| \,\ud V_{g_{\Sn}}(y)\\
=&\beta(n,\sigma)  \int_{\Sn} \frac{(1-|x|^2)^{2\sigma} |x+e_{n+1}|^{n-2\sigma}|y+e_{n+1}|^{n+2\sigma} }{  2^{n+2\sigma} | x- y|^{n+2\sigma}} u(y)\left(\frac{\sqrt{2}}{|y+e_{n+1}|}\right)^{n+2\sigma}\,\ud V_{g_{\Sn}}(y)\\
=&\beta(n,\sigma) \left(\frac{|x+e_{n+1}|}{\sqrt{2}}\right)^{n-2\sigma} 2^{-2\sigma} \int_{\Sn} \frac{(1-|x|^2)^{2\sigma} }{  | x- y|^{n+2\sigma}} u(y) \,\ud V_{g_{\Sn}}(y),
\end{align*}
where the second identity follows from definition of $w$, \eqref{eq:extension2} and the following elementary identity
\[
\frac{|x+e_{n+1}|}{\sqrt{2}} \frac{|y+e_{n+1}|}{\sqrt{2}} |\bar \Psi^{-1}(x)-\bar \Psi^{-1}(y)|=|x-y|.
\]
By definition of $\tilde W$ and \eqref{eq:extension1} we conclude that
\begin{equation*}
\tilde W(x)=2^{-2\sigma}  \beta(n,\sigma)  \int_{\Sn} \frac{(1-|x|^2)^{2\sigma} }{  | x- y|^{n+2\sigma}} u(y) \,\ud V_{g_{\Sn}}(y)=:\mathcal{Q}_\sigma [u](x)
\end{equation*}
for $x \in B_1$. Similarly, we have $\bar \xi(x)= \mathcal{Q}_\sigma [u v](x)$. Furthermore, it is easy to check that $\lim_{x\to y} \mathcal{Q}_\sigma[w](x)=w(y)$ for any fixed $y\in \Sn$ and $w\in C(\Sn)$.

In conclusion, we are now in a position to state the following extension formula in the unit ball.

\begin{theorem}[Extension formula of $P_\sigma$] \label{thm:lf}
Given $\sigma \in (0,1)$, and let $g=u^{4/(n-2\sigma)} g_{\Sn}$ with $0<u\in C^\infty(\Sn)$, then for any  $v\in C^\infty(\Sn)$, there hold
\begin{align}
\label{eq:lf-1}
&\mathrm{div}_{\tilde g}\left((1-|y|^2)^{1-2\sigma} U^{a} \nabla_{\tilde g} \frac{ V}{ U}\right)= 0 \quad \mathrm{in~~}B_1,\\
&P^g_\sigma(v)= \frac{2^{2\sigma-1}}{N(\sigma)} \mathcal{N}_{\tilde g}\frac{V(x)}{U(x)} +  R_\sigma^g v(x) \quad \qquad\mathrm{on~~} \Sn,
\label{eq:lf-2}
\end{align}
where $a=\frac{2-4\sigma}{n-2\sigma}, N(\sigma)=2^{1-2\sigma}\Gamma(1-\sigma)/\Gamma(\sigma), \tilde g= U(y)^{4/(n-2\sigma)}|\ud y|^2$ is a Riemannian metric on $\overline{B_1}$, and
\begin{equation}\label{eq:lf_ball}
U(y)= \mathcal{Q}_\sigma [u](y) = 2^{-2\sigma}  \beta(n,\sigma)  \int_{\Sn} \frac{(1-|y|^2)^{2\sigma} }{  | x- y|^{n+2\sigma}} u(x) \,\ud V_{g_{\Sn}}(x) \mathrm{~~for~~} y \in B_1,
\end{equation}
similarly, $V(y)= \mathcal{Q}_\sigma [u v](y)$, and
$\mathcal{N}_{\tilde g}=\lim_{y\to x}(1-|y|^2)^{1-2\sigma}  U(y)^{a}  \frac{\partial}{\partial \nu_{\tilde g}} $ is the conormal derivative respect to the divergence equation \eqref{eq:lf-1},  $\lim_{y\to x}$ is understood as $y=\tau x$ with $\tau \nearrow 1$, and $\beta(n,\sigma)$ is a positive constant such that 
$$\beta(n,\sigma)\int_{\mathbb{R}^n}\frac{x_{n+1}^{2\sigma}}{(|x'|^2+x_{n+1}^2)^{\frac{n+2\sigma}{2}}}\ud x'=1.$$
\end{theorem}
\begin{remark}
When $\sigma=1/2$, the formula \eqref{eq:lf_ball} coincides with the classical Poisson's formula in the unit ball $(\overline{B_1},g_{\R^{n+1}})$.
\end{remark}

Next, we will establish a Stroock-Varopoulos type inequality, such an inequality was once proved in \cite{DQRV} for the Euclidean space.

\begin{prop}\label{lemma1.11} Let $g\in [g_{\Sn}]$. Then for any $p\ge 2$ and $v\in C^\infty(\Sn)$, there hold
\begin{equation} \label{eq:pre-1}
\int_{\Sn}|v|^{p-2}v P^g_\sigma(v)\,\mathrm{d} V_{g} \ge \frac{4(p-1)}{ p^2} \int_{\Sn} |v|^{\frac{p}{2}} P^g_\sigma(|v|^{\frac{p}{2}})\,\ud V_g +\frac{(p-2)^2}{p^2} \int_{\Sn} R_\sigma^g |v|^p \,\mathrm{d} V_{g}
\end{equation}
and
\begin{equation} \label{eq:coercive-1}
\int_{\Sn}|v|^{p-2}v (P^g_\sigma-R_\sigma^g)(v)\,\mathrm{d} V_{g} \ge 0.
\end{equation}
\end{prop}

\begin{proof} Suppose $g=u^{4/(n-2\sigma)} g_{\Sn}$ for some positive function  $u\in C^\infty(\Sn)$. Let $U=\mathcal{Q}_\sigma[u]$, $V= \mathcal{Q}_\sigma[uv]$ and $\tilde g =U(y)^{4/(n-2\sigma)} |\ud y|^2$. By Theorem \ref{thm:lf}, we have
\begin{align*}
&\int_{\Sn}|v|^{p-2}v P^g_\sigma(v)\,\mathrm{d} V_{g}\\
=&\int_{\Sn}\left|\tfrac{V}{U}\right|^{p-2} \tfrac{V}{U} \left(  \frac{2^{2\sigma-1}}{N(\sigma)}\mathcal{N}_{\tilde g} \tfrac{V}{U} +  R_\sigma^g v  \right) \,\mathrm{d} V_{ g}\\
=&\frac{2^{2\sigma-1}}{N(\sigma)}\int_{B_1} (1-|y|^2)^{1-2\sigma} U^{a} \langle\nabla\tfrac{V}{U}, \nabla\left(\left|\tfrac{V}{U}\right|^{p-2} \tfrac{V}{U}\right)\rangle_{
\tilde g} \,\ud V_{\tilde g} +\int_{\Sn} R_\sigma^g |v|^p \,\mathrm{d} V_{g}\\
=&\frac{2^{2\sigma-1} 4(p-1)}{N(\sigma) p^2} \int_{B_1} (1-|y|^2)^{1-2\sigma} U^{a} |\nabla |\tfrac{V}{U}|^{\frac{p}{2}}|_{
\tilde g}^2\,\ud V_{\tilde g} +\int_{\Sn} R_\sigma^g |v|^p \,\mathrm{d} V_g.
\end{align*}
Let $\eta= \mathcal{Q}_\sigma[|v|^{p/2} u]$. By \eqref{eq:lf-1}, we have
$\mathrm{div}_{\tilde g}((1-|y|^2)^{1-2\sigma} U^{a} \nabla_{\tilde g} \frac{\eta}{U})=0$ in $B_1$.
Notice that
$\frac{\eta}{U}=|v|^{p/2}$ on $\Sn=\pa B_1$. Denote $\zeta=U|\frac{V}{U}|^{p/2}-\eta$. Then we have
\begin{align}
&\frac{2^{2\sigma-1}}{N(\sigma)}\int_{B_1} (1-|y|^2)^{1-2\sigma} U^{a} |\nabla |\tfrac{V}{U}|^{\frac{p}{2}}|_{
\tilde g}^2\,\ud V_{\tilde g}  \nonumber \\
=& \frac{2^{2\sigma-1}}{N(\sigma)}\int_{B_1} (1-|y|^2)^{1-2\sigma} U^{a} \left( |\nabla \tfrac{\eta}{U}|_{\tilde g}^2+|\nabla \tfrac{\zeta}{U}|_{\tilde g}^2+2 \langle\nabla \tfrac{\eta}{U},\nabla \tfrac{\zeta}{U}\rangle_{\tilde g}\right)\,\ud V_{\tilde g} \nonumber \\
=&  \frac{2^{2\sigma-1}}{N(\sigma)}\int_{B_1} (1-|y|^2)^{1-2\sigma} U^{a}  ( |\nabla \tfrac{\eta}{U}|_{\tilde g}^2+|\nabla \tfrac{\zeta}{U}|_{\tilde g}^2)\,\ud V_{\tilde g} \nonumber \\
\ge& \frac{2^{2\sigma-1}}{N(\sigma)}\int_{B_1} (1-|y|^2)^{1-2\sigma} U^{a}   |\nabla \tfrac{\eta}{U}|_{\tilde g}^2\,\ud V_{\tilde g}\nonumber\\
=&\int_{\Sn} |v|^{\frac{p}{2}} P^g_\sigma(|v|^{\frac{p}{2}})\,\ud V_g-\int_{\Sn} R_\sigma^g |v|^p \,\mathrm{d} V_g,
\label{eq:coercive}
\end{align}
where the second identity follows from an integration by parts together with $\zeta=0$ on $\partial B_1$ and the equation of $\eta$ in $B_1$, and the last equality again follows from \eqref{eq:lf-2}.
Therefore, putting these facts together, we obtain
\begin{equation*}
\int_{\Sn}|v|^{p-2}v P^g_\sigma(v)\,\mathrm{d} V_{g} \ge \frac{4(p-1)}{ p^2} \int_{\Sn} |v|^{\frac{p}{2}} P^g_\sigma(|v|^{\frac{p}{2}})\,\ud V_g +\frac{(p-2)^2}{p^2} \int_{\Sn} R_\sigma^g |v|^p \,\mathrm{d} V_{g},
\end{equation*}
which is exactly \eqref{eq:pre-1}.

By \eqref{eq:pre-1} and the last identity in \eqref{eq:coercive}, we immediately obtain
\begin{align*}
&\int_{\Sn}|v|^{p-2}v (P^g_\sigma-R_\sigma^g)(v)\,\mathrm{d} V_{g}  \\
\ge& \frac{4(p-1)}{ p^2} \int_{\Sn} |v|^{\frac{p}{2}} P^g_\sigma(|v|^{\frac{p}{2}})\,\ud V_g -\frac{4(p-1)}{p^2} \int_{\Sn} R_\sigma^g |v|^p \,\mathrm{d} V_{g} \\
\ge& \frac{4(p-1)}{ p^2} \int_{\Sn} |v|^{\frac{p}{2}} (P^g_\sigma-R_\sigma^g)(|v|^{\frac{p}{2}})\,\ud V_g
\ge 0.
\end{align*}
This completes the proof.
\end{proof}



For $0<\beta<\min\{1,2\sigma\}$ and $T>0$,
we say that a function $v\in C^{\beta,\frac{\beta}{2\sigma}}(\Sn\times (0,T])$ if
\begin{equation*}
\begin{split}
\|v\|_{\beta,\frac{\beta}{2\sigma};\Sn\times (0,T]}
&=\|v\|_{0;\Sn\times (0,T]}+[v]_{\beta,\frac{\beta}{2\sigma};\Sn\times (0,T]}\\
&:=\sup_{Y\in \Sn\times (0,T]}|v(Y)|
+\sup_{Y_1\neq Y_2,Y_1,Y_2\in \Sn\times (0,T]}\frac{|v(Y_1)-v(Y_2)|}{\rho(Y_1,Y_2)^\beta}<\infty,
\end{split}
\end{equation*}
where $Y_1=(x_1,t_1), Y_2=(x_2,t_2)$, $\rho(Y_1,Y_2)=(|x_1-x_2|^2+|t_1-t_2|^{\frac{1}{\sigma}})^{\frac{1}{2}}$
and  $|x_1-x_2|$ is understood as the Euclidean distance from $x_1$ to $x_2$ in $\mathbb{R}^{n+1}$.

\begin{lemma}[Maximum Principle] \label{lem:mp} For $T>0$ and  let $g(t)\in [g_{\Sn}]$ for $t\in [0,T]$. Suppose that $w\in C^{2\sigma,1}(\Sn\times [0,T])$ is a solution of
\[
w_t+a P^g_\sigma w +b w \ge 0 \quad \mathrm{in~~}\Sn\times [0,T],
\]
where $a,b\in C( \Sn\times [0,T])$ and $a\ge \delta>0$. If $w(\cdot, 0)\ge 0$, then $w\ge 0$ in $\Sn\times [0,T]$.

\end{lemma}

\begin{proof} Write $g=u(t)^{4/(n-2\sigma)} g_{\Sn}$ for some positive $u(t)\in C^\infty(\Sn)$.  By \eqref{conformal_invariance} and  \eqref{1.2.4}, for every fixed $t$ we have
\begin{align*}
u(x)^{\frac{n+2\sigma}{n-2\sigma}} P^g_\sigma w (x)=&P_\sigma (uw)(x)\\
=&c_{n,-\sigma}  \int_{\Sn} \frac{u(x)w(x)-u(y)w(y)}{|x-y|^{n+2\sigma}}\,\ud V_{g_{\Sn}}(y)+ R_\sigma u(x)w(x)\\
=&c_{n,-\sigma}\int_{\Sn} \frac{(w(x)-w(y))u(y)}{|x-y|^{n+2\sigma}}\,\ud V_{g_{\Sn}}(y) +  h(x) w(x),
\end{align*}
where we have dropped the $t$ variable in the above computations for convenience, and
\[
h(x)= c_{n,-\sigma}\int_{\Sn} \frac{u(x)-u(y)}{|x-y|^{n+2\sigma}}\,\ud V_{g_{\Sn}}(y)+R_\sigma u(x).
\]
Let $\tilde b:= b+ a u^{-(n+2\sigma)/(n-2\sigma)} h$,
$
\tilde M=\max_{\Sn\times [0,T]}|\tilde b|+1<\infty $ by the hypotheses of this lemma,  and $\tilde w= e^{-\tilde M t} w$. It follows that
\begin{equation} \label{eq:diffineq-1}
\tilde w_t + c_{n,-\sigma} a u^{-\frac{n+2\sigma}{n-2\sigma}} \int_{\Sn} \frac{(\tilde w(x)-\tilde w(y))u(y)}{|x-y|^{n+2\sigma}}\,\ud V_{g_{\Sn}}(y) +(\tilde M-\tilde b) \tilde w \ge 0  \quad \mathrm{in~~}\Sn\times [0,T].
\end{equation}

By negation, we assume that $\tilde w$ achieves its negative minimum at $(x_0,t_0) \in \Sn\times [0,T]$. Obviously, $t_0>0$. Then we have
\[
\tilde w_t(x_0, t_0)\le 0, \qquad  \int_{\Sn} \frac{(\tilde w(x_0,t_0)-\tilde w(y,t_0))u(y,t_0)}{|x-y|^{n+2\sigma}}\ud V_{g_{\Sn}}(y) \le 0, 
\]
and 
$$ 
(\tilde M-\tilde b) \tilde w(x_0,t_0)<0.$$
We now arrive at a contradiction with \eqref{eq:diffineq-1}. Therefore, $\tilde w\ge 0$ and thus $w\ge 0$.
\end{proof}

The following proposition  was proved in  \cite[Proposition 4.2]{Jin&Xiong}:

\begin{prop}\label{prop4.2_in_Jin&Xiong}
Let $0<\beta<\min\{1,2\sigma\}$ such that
$2\sigma+\beta$ is not an integer.
Let $a, b, d\in C^{\beta,\frac{\beta}{2\sigma}}(\Sn\times(0,1])$,
$v_0\in C^{2\sigma+\beta}(\Sn)$ and $\lambda^{-1}\leq a(x,t)\leq\lambda$ for some constant $\lambda\geq 1$.
Then there exists a unique function $v\in C^{2\sigma+\beta,1+\frac{\beta}{2\sigma}}(\Sn\times(0,1])$ such that
\begin{equation*}
\left\{
  \begin{array}{ll}
    av_t+P_\sigma(v)+bv=d, &\quad  \mathrm{ in~~} \Sn\times(0,1]; \\
    v(x,0)=v_0(x), & \quad \mathrm{ on~~} \Sn.
  \end{array}
\right.
\end{equation*}
Moreover, there exists a constant $C$ depending only on $n$, $\sigma$, $\lambda$, $\|a\|_{\beta,\frac{\beta}{2\sigma};\Sn\times(0,1]}$ and $\|b\|_{\beta,\frac{\beta}{2\sigma};\Sn\times(0,1]}$
such that
\begin{equation*}
\|v\|_{2\sigma+\beta,1+\frac{\beta}{2\sigma};\Sn\times(0,1]}\leq C(\|v_0\|_{2\sigma+\beta;\Sn\times(0,1]}+\|d\|_{\beta,\frac{\beta}{2\sigma};\Sn\times(0,1]}).
\end{equation*}
\end{prop}

\begin{convention}
From now on, we let $\omega_n=\mathrm{Vol}(\Sn,g_{\Sn})$ and $\fint_{\Sn}$ denote $\omega_n^{-1} \int_{\Sn}$.
\end{convention}

\section{Basic properties of the flow}\label{Sect:Property_flow}

First of all, by the same proof of \cite[Proposition 4.8]{Jin&Xiong}, for any given $g_0\in [g_{\Sn}]$ there exist a constant $T>0$ and a unique positive smooth solution of \eqref{1.3} for all $0<t\leq T$. Next, we establish some basic properties of the flow.

By \eqref{conformal_invariance} we have
\begin{equation}\label{1.1}
\mathcal{S}(g)=\frac{\int_{\Sn}uP_\sigma(u)\ud V_{g_{\Sn}}}{(\int_{\Sn}f|u|^{\frac{2n}{n-2\sigma}}\ud V_{g_{\Sn}})^{\frac{n-2\sigma}{n}}}:=E_f[u].
\end{equation}

\begin{prop} \label{prop:properties} 
Let $g(t)$ be a smooth solution of \eqref{1.3} with $\sigma \in (0,1)$. Then
\begin{itemize}
\item[(1)] $
\frac{\partial}{\partial t}\ud V_g=-\frac{n}{2}(R_\sigma^g-\alpha f)\ud V_{g}$ and thus the volume $\mathrm{Vol}_{g(t)}(\Sn)$ is preserved,
\item[(2)] $\frac{\partial}{\partial t}(R_\sigma^g-\alpha f)=
-\frac{n-2\sigma}{4}P_\sigma^g(R_\sigma^g-\alpha f)+\frac{n+2\sigma}{4}R_\sigma^g(R_\sigma^g-\alpha f)-\alpha' f,$
\item[(3)] $\frac{\ud}{\ud t}\mathcal{S}(g)=-\frac{n-2\sigma}{2}\frac{\int_{\Sn}(R_\sigma^g-\alpha f)^2\ud V_{g}}{(\int_{\Sn}f\ud V_{g})^{\frac{n-2\sigma}{n}}}\le 0$.
\end{itemize}

\end{prop}

\begin{proof} Item (1) follows immediately from \eqref{1.7}.

Write $g(t)=u(t)^{4/(n-2\sigma)} g_{\Sn}$. By definition of $R_\sigma^g$, \eqref{conformal_invariance} and \eqref{1.7}, we have
\begin{align}
\frac{\partial}{\partial t}R_\sigma^g&=-\frac{n+2\sigma}{n-2\sigma}u^{-\frac{2n}{n-2\sigma}}u_t P_\sigma(u)
+u^{-\frac{n+2\sigma}{n-2\sigma}}P_\sigma(u_t) \nonumber \\
&=\frac{n+2\sigma}{4}(R_\sigma^g-\alpha f)u^{-\frac{n+2\sigma}{n-2\sigma}}P_\sigma(u)
-\frac{n-2\sigma}{4}u^{-\frac{n+2\sigma}{n-2\sigma}}P_\sigma\big((R_\sigma^g-\alpha f)u\big)\nonumber \\
&=\frac{n+2\sigma}{4}R_\sigma^g(R_\sigma^g-\alpha f)-\frac{n-2\sigma}{4}P_\sigma^g(R_\sigma^g-\alpha f).
\label{1.11}
\end{align}
Since $(\al f)'=\al' f$, item (2) follows.

By item (1) and the fact that $P_\sigma$ is self-adjoint,  we have
\begin{align}\label{1.13}
&\frac{\ud}{\ud t}\left(\int_{\Sn}R_\sigma^g\ud V_{g}\right)=2\int_{\Sn}uP_\sigma(u_t)\ud V_{g_{\Sn}}\no\\
=&\frac{n-2\sigma}{2}\int_{\Sn}uP_\sigma((\alpha f-R_\sigma^g)u)\ud V_{g_{\Sn}}\no\\
=&\frac{n-2\sigma}{2}\int_{\Sn}P_\sigma(u)(\alpha f-R_\sigma^g)u\ud V_{g_{\Sn}}\no\\
=&-\frac{n-2\sigma}{2}\int_{\Sn}R_\sigma^g(R_\sigma^g-\alpha f)\ud V_{g}.
\end{align}
Using item (1) again and \eqref{1.13}, we have
\begin{align*}
\frac{\ud}{\ud t}\mathcal{S}(g)=&\frac{\frac{\ud}{\ud t}(\int_{\Sn}R_\sigma^g\ud V_{g})}{(\int_{\Sn}f\ud V_{g})^{\frac{n-2\sigma}{n}}}
-\frac{n-2\sigma}{n}\frac{(\int_{\Sn}R_\sigma^g\ud V_{g})(\int_{\Sn}f\frac{\partial}{\partial t}\ud V_{g})
}{(\int_{\Sn}f\ud V_{g})^{\frac{n-2\sigma}{n}+1}}\\
=&-\frac{n-2\sigma}{2}\frac{\int_{\Sn}R_\sigma^g(R_\sigma^g-\alpha f)\ud V_{g}}{(\int_{\Sn}f\ud V_{g})^{\frac{n-2\sigma}{n}}}
+\frac{n-2\sigma}{2}\frac{\alpha\int_{\Sn}f(R_\sigma^g-\alpha f)\ud V_{g}
}{(\int_{\Sn}f\ud V_{g})^{\frac{n-2\sigma}{n}}}\\
=&-\frac{n-2\sigma}{2}\frac{\int_{\Sn}(R_\sigma^g-\alpha f)^2\ud V_{g}}{(\int_{\Sn}f\ud V_{g})^{\frac{n-2\sigma}{n}}}.
\end{align*}
Hence, item (3) is proved.
\end{proof}

Without loss of generality, we assume $\int_{\Sn}\ud V_{g(0)}=\omega_n$. It follows from item (1) of Proposition \ref{prop:properties} that along the flow \eqref{1.3},
\begin{equation}\label{1.9}
\int_{\Sn}\ud V_{g(t)}=\omega_n.
\end{equation}

\begin{lemma}\label{lemma1.4} 
Let $0<m:=\min_{\Sn}f\le  \max_{\Sn}f=:M<\infty$. Then there exist positive constants $\alpha_1$, $\alpha_2$ and $\al_3$, depending only on $n,\sigma, g_0$, $m$ and $M$,  such that
$$0<\alpha_1\leq\alpha\leq\alpha_2 \quad\mathrm{and} \quad \al' \le \al_3$$
for all $0\leq t\leq T$.
\end{lemma}
\begin{proof}
By the Sobolev inequality \eqref{eq:sobolev} and the normalization \eqref{1.9}, we have
\begin{align*}
\alpha=\frac{\int_{\Sn}R_\sigma^g\ud V_g}{\int_{\Sn}f\ud V_g}=
\frac{\int_{\Sn}uP_\sigma(u)\ud V_{g_{\Sn}}}{\int_{\Sn}f\ud V_g}\geq&\frac{Y_\sigma(\Sn)(\int_{\Sn} u^{\frac{2n}{n-2\sigma}}\,\ud V_{g_{\Sn}})^{\frac{n-2\sigma}{n}}}{M \omega_n}\\
=&\frac{Y_\sigma(\Sn) \omega_n^{-\frac{2\sigma}{n}}}{M}:=\alpha_1.
\end{align*}

On the other hand, it follows from item (3) of Proposition \ref{prop:properties} that
\begin{equation*}\label{1.19.5}
\alpha \leq \mathcal{S}(g_0)\left(\int_{\Sn}f\ud V_{g}\right)^{-\frac{2\sigma}{n}}
\leq \mathcal{S}(g_0) (m\omega_n)^{-\frac{2\sigma}{n}}:=\alpha_2.
\end{equation*}

Differentiating both sides of the equation $\al \int_{\Sn}f\ud V_g= \int_{\Sn}R_\sigma^g\ud V_g $ with respect to $t$, by Proposition \ref{prop:properties} we have
\begin{equation}\label{eq:alpha'}
\alpha'\int_{\Sn}f\ud V_g-\frac{n}{2}\int_{\Sn}\alpha f(R_\sigma^g-\alpha f)\ud V_{g}=-\frac{n-2\sigma}{2}\int_{\Sn}R_\sigma^g(R_\sigma^g-\alpha f)\ud V_{g}.
\end{equation}
It follows from Young's inequality that
\begin{align*}
\alpha'\int_{\Sn}f\ud V_g=&-\frac{n-2\sigma}{2}\int_{\Sn}(R_\sigma^g-\alpha f)^2\ud V_{g}+\sigma\int_{\Sn}\alpha f(R_\sigma^g-\alpha f)\ud V_{g}\\
\le& C(n,\sigma)\,\alpha^2\int_{\Sn}f^2\ud V_g.
\end{align*}
Then,
$$\alpha'\leq \frac{C\alpha^2\int_Mf^2\ud V_g}{\int_{\Sn}f\ud V_g}\leq \frac{C\alpha_2^2M^2}{m}=:\al_3.$$
Therefore, the lemma is proved.
\end{proof}

\begin{lemma}\label{lemma1.6}
There holds
$$R_\sigma^g-\alpha f\geq \min\left\{\inf_{\Sn}R_\sigma^{g(0)}-\alpha_2M,-\frac{M}{\sigma \alpha_1m}(\alpha_3+\sigma\alpha_2^2M)\right\}:=\gamma$$
for all $t\geq 0$.
\end{lemma}
\begin{proof}
Let $L:=\partial_t+\frac{n-2\sigma}{4}(P_\sigma^g-R_\sigma^g)+\sigma \alpha f.$
By \eqref{1.11}, we have
\begin{equation*}
L(R_\sigma^g)=\frac{n-2\sigma}{4}[P_\sigma^g(\alpha f)-R_\sigma^g \alpha f]+\frac{n+2\sigma}{4}(R_\sigma^g)^2.
\end{equation*}
Set $w=\alpha f+\gamma$. Recalling that $P_\sigma^g(1)=R_\sigma^g$ we have
\begin{align*}
L(w)&=\alpha' f+\frac{n-2\sigma}{4}[P_\sigma^g(\alpha f)-R_\sigma^g \alpha f]+\sigma\alpha f(\alpha f+\gamma)\\
&\leq\alpha_3 M+\frac{n-2\sigma}{4}[P_\sigma^g(\alpha f)-R_\sigma^g \alpha f]+\sigma(\alpha_2^2 M^2+\alpha_1m\gamma)\\
&\leq \frac{n-2\sigma}{4}[P_\sigma^g(\alpha f)-R_\sigma^g \alpha f]\leq L(R_\sigma^g).
\end{align*}
Since $
w(0)\leq \alpha_2 M+\gamma\leq \inf_{\Sn}R_\sigma^{g(0)}\leq R_\sigma^{g(0)}$, we apply Lemma \ref{lem:mp} to $R_\sigma^g- w$ and obtain that $R_\sigma^g- w \ge 0$. Therefore, the lemma is proved.
\end{proof}

\begin{lemma}\label{lemma1.9} There exist two positive constants $C_1$ and $C_2$, depending only on $n,\sigma, u_0,m$ and $M$,  such that
$$\frac{1}{C_1} e^{-C_2T} \leq u(x,t)\leq C_1 e^{C_2T} \quad \mathrm{in~~} \Sn\times[0,T]. $$
\end{lemma}
\begin{proof}
It follows from \eqref{1.7} and Lemma \ref{lemma1.6} that
$$\frac{\partial u}{\partial t}=-\frac{n-2\sigma}{4}(R_\sigma^g-\alpha f)u\leq-\frac{n-2\sigma}{4}\gamma\, u$$
for all $t\geq 0$. Hence,
\begin{equation}\label{est:upper_bd_u_finite_time}
u(x,t)\leq u_0(x)e^{-\frac{n-2\sigma}{4}\gamma t}\leq (\max_{\Sn}u_0) e^{-\frac{n-2\sigma}{4}\gamma T}
\end{equation}
for any $(x,t)\in \Sn\times [0,T]$. This gives the upper bound of $u$.

By Lemma \ref{lemma1.6}, we have
\begin{align*}
0&\le (R_\sigma^g -\alpha f -\gamma)u^{\frac{n+2\sigma}{n-2\sigma}}
=P_\sigma(u) -(\alpha f +\gamma)u^{\frac{n+2\sigma}{n-2\sigma}}.
\end{align*}
By stereographic projection, we define $w(y)=\det \ud\Psi(y)^{\frac{n-2\sigma}{2n}} u\circ \Psi(y)$ and $W(X)=\mathcal{P}_\sigma[w](X)$ in the same way of \eqref{eq:poisson}. By Lemma \ref{lemma1.4}, definition of $\gamma$ in Lemma \ref{lemma1.6} and \eqref{est:upper_bd_u_finite_time}, we have
$$a_0:=\max_{\Sn\times [0,T]}\left(|\alpha f+\gamma|u^{\frac{4\sigma}{n-2\sigma}}\right)\le C(T).$$
Thus by definition of $w$ and \eqref{eq:r1} we have
\begin{align*}
&(-\Delta)^\sigma w+a_0 \left(\frac{2}{1+|y|^2}\right)^{2\sigma}w\\
\geq& (-\Delta)^\sigma w- (\alpha f\circ \Psi+\gamma ) w^{\frac{n+2\sigma}{n-2\sigma}}\\
=&\left(\frac{2}{1+|y|^2}\right)^{\frac{n+2\sigma}{2}}[P_\sigma(u) -(\alpha f +\gamma)u^{\frac{n+2\sigma}{n-2\sigma}}]\circ \Psi\ge 0
\end{align*}
and  $w>0$ in $\mathbb{R}^n$.

By the weak Harnack inequality (see \cite[Proposition 2.6 (ii)]{Jin&Li&Xiong1} or Jin-Xiong \cite{TX}), we have
\begin{align*}
\inf_{B_1} w\ge \frac{1}{C} \int_{B_2\times [0,2]} W(X) \ud X &
= \frac{1}{C}\int_{B_2\times [0,2]}\mathcal{P}_\sigma[w](X)\,\ud X\\&
\ge \frac{1}{C} \int_{\mathbb{R}^{n}}\frac{w(y)}{(1+|y|)^{n+2\sigma}}\,\ud y.
\end{align*}
where the last inequality follows from
\begin{align*}
\int_{B_2\times [0,2]} \mathcal{P}_\sigma[w](X)\,\ud X
&=\int_{\mathbb{R}^n}w(y)\ud y\int_{B_2\times [0,2]}\frac{\beta(n,\sigma)\tau^{2\sigma}}{(|x-y|^2+\tau^2)^{\frac{n+2\sigma}{2}}}\ud X\\
&\geq\int_{\mathbb{R}^n}w(y)\ud y\int_{B_2\times [1,2]}\frac{\beta(n,\sigma)}{(|x-y|^2+4)^{\frac{n+2\sigma}{2}}}\ud X\\
&\geq C(n,\sigma) \int_{\mathbb{R}^n}\frac{w(y)}{(1+|y|)^{n+2\sigma}}\ud y.
\end{align*}
Pulling back to $\Sn$ by stereographic projection and using a standard partition of unity argument, we have
\[
\inf_{\Sn\times [0,T]}u \ge \frac{1}{C} \int_{\Sn} u \,\ud V_{g_{\Sn}}.
\]
By \eqref{1.9}, we obtain
$$\omega_n=\int_{\Sn}u^{\frac{2n}{n-2\sigma}}\ud V_{g_{\Sn}} \leq C\Big( \inf_{\Sn\times [0,T]} u\Big)\Big(  \sup_{\Sn \times [0,T]} u\Big)^{\frac{n+2\sigma}{n-2\sigma}}.$$ 
This together with the upper bound implies the lower bound.
\end{proof}

\begin{proof}[Proof of Theorem \ref{thm:1}] It follows from Lemmas \ref{lemma1.9}, \ref{lemma1.4}, the H\"older estimates for parabolic nonlocal equations, as well as  Proposition \ref{prop4.2_in_Jin&Xiong}.
\end{proof}

\section{Curvature convergences in integral norms}\label{Sect:Integ_curv_est}

For $p\geq 2$, we let
$$F_p(t)=\int_{\Sn}|\alpha f-R_\sigma^g|^{p}\ud V_g$$
and
$$G_p(t)=\int_{\Sn}|\alpha f-R_\sigma^g|^{p-2}(\alpha f-R_\sigma^g)P_\sigma^g(\alpha f-R_\sigma^g)\ud V_g.$$

By Proposition \ref{prop:properties}, we have
\begin{equation}\label{1.39}
\begin{split}
&\frac{\ud}{\ud t}\int_{\Sn}|\alpha f-R_\sigma^g|^{p}\ud V_g\\
=&-\frac{p(n-2\sigma)}{4}\int_{\Sn}|\alpha f-R_\sigma^g|^{p-2}(\alpha f-R_\sigma^g)P_\sigma^g(\alpha f-R_\sigma^g)\ud V_g\\
&+\frac{p(n+2\sigma)}{4}\int_{\Sn}\alpha f|\alpha f-R_\sigma^g|^{p}\ud V_g+p\alpha'\int_{\Sn}f|\alpha f-R_\sigma^g|^{p-2}(\alpha f-R_\sigma^g)\ud V_g\\
&+\left(\frac{n}{2}-\frac{p(n+2\sigma)}{4}\right)\int_{\Sn}|\alpha f-R_\sigma^g|^{p}(\alpha f-R_\sigma^g)\ud V_g.
\end{split}
\end{equation}

\begin{lemma}\label{lem:F_p_mid} 
For $2 \leq p \leq \max \{2,n/(2\sigma)\}$, there hold
\begin{equation}\label{est:F_p_intermediate}
\int_0^\infty F_p(t)\, \ud t<\infty \quad \mathrm{and} \quad \lim_{t \to \infty}F_p(t)=0.
\end{equation}
\end{lemma}
\begin{proof}
By item (3) of Proposition \ref{prop:properties}, we have for every $T>0$,
\[
\frac{n-2\sigma}{2}\int_0^T \frac{F_2(t)}{(\int_{\Sn} f\,\ud V_g)^{\frac{n-2\sigma}{n}}}\,\ud t= \mathcal{S}(g(0)) -\mathcal{S}(g(T))\le \mathcal{S}(g(0)).
\]
Sending $T \to \infty$ and using \eqref{1.9}, we obtain 
\begin{equation}\label{est:int_F_2}
 \int_{0}^\infty F_2(t)\,\ud t\leq \frac{2}{n-2\sigma} (M\omega_n)^{\frac{n-2\sigma}{n}}\mathcal{S}(g(0))<\infty.
 \end{equation}

(i) $n<4\sigma$, which forces $p=2$. It follows from (\ref{1.39}) with $p=2$ that
\begin{equation}\label{dF_2_I}
\begin{split}
\frac{\ud}{\ud t}F_2(t)
&=-\frac{n-2\sigma}{2}\int_{\Sn}(\alpha f-R_\sigma^g)P_\sigma^g(\alpha f-R_\sigma^g)\ud V_g+\frac{n+2\sigma}{2}\int_{\Sn}\alpha f|\alpha f-R_\sigma^g|^{2}\ud V_g\\&\quad +2\alpha'\int_{\Sn}f(\alpha f-R_\sigma^g)\ud V_g-\sigma\int_{\Sn}|\alpha f-R_\sigma^g|^{2}(\alpha f-R_\sigma^g)\ud V_g\\
&:=I_1+I_2+I_3+I_4.
\end{split}
\end{equation}
By Sobolev inequality \eqref{eq:sobolev} we have
\begin{align*}
I_1=&-\frac{n-2\sigma}{2}\int_{\Sn}(\alpha f-R_\sigma^g)P_\sigma^g(\alpha f-R_\sigma^g)\ud V_g\\
\leq& -\frac{n-2\sigma}{2}Y_\sigma(\Sn) \left(\int_{\Sn}|\alpha f-R_\sigma^g|^{\frac{2n}{n-2\sigma}}\ud V_g\right)^{\frac{n-2\sigma}{n}}.
\end{align*}
By Lemma \ref{lemma1.4}, we have
\begin{equation*}
I_2=\frac{n+2\sigma}{2}\int_{\Sn}\alpha f|\alpha f-R_\sigma^g|^{2}\ud V_g\leq C F_2(t)
\end{equation*}
and using \eqref{eq:alpha'} and Young's inequality,
\begin{equation*}
\begin{split}
I_3&=2\alpha'\int_{\Sn}f(\alpha f-R_\sigma^g)\ud V_g\\
&=\frac{2\int_{\Sn}f(\alpha f-R_\sigma^g)\ud V_g}{\int_{\Sn}f\ud V_g}\left[-\frac{n-2\sigma}{2}\int_{\Sn}(R_\sigma^g-\alpha f)^2\ud V_{g}+\sigma\int_{\Sn}\alpha f(R_\sigma^g-\alpha f)\ud V_{g}\right]
\\
&\leq CF_2(t)(1+F_2(t)^{\frac{1}{2}}).
\end{split}
\end{equation*}
By H\"{o}lder's and Young's inequalities, we have for every small $\epsilon >0$,
\begin{align*}
I_4=&-\sigma\int_{\Sn}|\alpha f-R_\sigma^g|^{2}(\alpha f-R_\sigma^g)\ud V_g\\
\leq& \sigma\left(\int_{\Sn}|\alpha f-R_\sigma^g|^{\frac{2n}{n-2\sigma}}\ud V_g\right)^{\frac{n-2\sigma}{4\sigma}}
\left(\int_{\Sn}|\alpha f-R_\sigma^g|^{2}\ud V_g\right)^{\frac{6\sigma-n}{4\sigma}}\\
\leq& \epsilon\left(\int_{\Sn}|\alpha f-R_\sigma^g|^{\frac{2n}{n-2\sigma}}\ud V_g\right)^{\frac{n-2\sigma}{n}}
+C_\epsilon\left(\int_{\Sn}|\alpha f-R_\sigma^g|^{2}\ud V_g\right)^{\frac{6\sigma-n}{4\sigma-n}}.
\end{align*}
Let $0<\epsilon \le \frac{n-2\sigma}{2}Y_\sigma(\Sn)$, then it follows that
\begin{equation}\label{ineq:dF_2_n=3}
\frac{\ud}{\ud t}F_2(t)\leq CF_2(t)(1+F_2(t)^{\frac{2\sigma}{4\sigma-n}}),
\end{equation}
where we have used $\frac{2\sigma}{4\sigma-n}>\frac12>0$ and the basic inequality
\begin{equation} \label{eq:b-ineq}
a^p \le a^{m_1}+a^{m_2}\quad \mathrm{for~~any~~} a\ge 0 \mathrm{~~and~~}0\le m_1\le p\le m_2.
\end{equation}
By \eqref{est:int_F_2}, let $t_j \to \infty$ as $j\to \infty$ be an increasing sequence such that $F_2(t_j)\to 0$, and let  $H(t)=\rho(F_2(t))$, where $\rho(\tau)=\int_0^{\tau}(1+s^{\frac{2\sigma}{4\sigma-n}})^{-1}\,\ud s$. It follows from \eqref{ineq:dF_2_n=3} that
\begin{equation}\label{eq:seqtoall}
H(t)\leq H(t_j)+C\int_{t_j}^{\infty} F_2(s)\,\ud s,\quad \forall~ t_j\le t.
\end{equation}
Therefore, $H(t)\to 0$ as $t\to \infty$. Since $\rho(\cdot)$ is increasing and $\rho(0)=0$, it yields $F_2(t)\to 0$ as $t\to \infty$.

(ii) $n\ge 4\sigma$.
We rearrange \eqref{1.39} as
\begin{align}\label{dF_p_rearrangement}
\frac{\ud}{\ud t}F_p(t)=&-\left(\frac{n}{2}-p\sigma\right)\int_{\Sn}|\alpha f-R_\sigma^g|^{p}(R_\sigma^g-\alpha f)\ud V_g \nonumber \\
&-\frac{p(n-2\sigma)}{4}\int_{\Sn}\left[|\alpha f-R_\sigma^g|^{p-2}(\alpha f-R_\sigma^g)P_\sigma^g(\alpha f-R_\sigma^g)-R_\sigma^g |\alpha f-R_\sigma^g|^p \right]\ud V_g \nonumber \\
&+p\sigma\int_{\Sn}\alpha f|\alpha f-R_\sigma^g|^{p}\ud V_g+p\alpha'\int_{\Sn}f|\alpha f-R_\sigma^g|^{p-2}(\alpha f-R_\sigma^g)\ud V_g \nonumber  \\
\le &-\left(\frac{n}{2}-p\sigma\right)\int_{\Sn}|\alpha f-R_\sigma^g|^{p}(R_\sigma^g-\alpha f)\ud V_g +CF_p(t)(1+F_p(t)^{1/p}),
\end{align}
where we have used \eqref{eq:coercive-1} and similar arguments for estimating $I_2$ and $I_3$. Observe that
\begin{align*}
&\int_{\Sn}|\alpha f-R_\sigma^g|^{p+1} \ud V_g\\
=&\int_{\{\alpha f\leq R_\sigma^g\}}|\alpha f-R_\sigma^g|^p (R_\sigma^g-\alpha f)\ud V_g+\int_{\{\alpha f>R_\sigma^g\}}|\alpha f-R_\sigma^g|^{p+1} \ud V_g\\
=&\int_{\Sn}|\alpha f-R_\sigma^g|^{p}(R_\sigma^g-\alpha f)\ud V_g+2\int_{\{\alpha f>R_\sigma^g\}}|\alpha f-R_\sigma^g|^p (\alpha f-R_\sigma^g) \ud V_g.
\end{align*}
By Lemma \ref{lemma1.6} we obtain
 $$\int_{\{\alpha f>R_\sigma^g\}}|\alpha f-R_\sigma^g|^p (\alpha f-R_\sigma^g) \ud V_g \le -\gamma F_p (t)$$
 and then 
\begin{equation}\label{eq:conv-1}
\frac{\ud}{\ud t}F_p(t) +(\frac{n}{2}-p\sigma) F_{p+1}(t)\le C F_p(t)(1+F_p(t)^{1/p}).
\end{equation}
If we let $p=2$ in \eqref{eq:conv-1}, then the above inequality leads to $\frac{\ud}{\ud t}F_2(t) \le C F_2(t)(1+F_2(t)^{1/2}) $ and thus \eqref{eq:seqtoall} holds with $\rho(\tau)= \int_0^{\tau}(1+s^{1/2})^{-1}\,\ud s$. Since $\rho(\cdot)$ is continuous, increasing and $\rho(0)=0$, it yields $F_2(t)\to 0$ as $t \to \infty$. Hence, we prove \eqref{est:F_p_intermediate} for $p=2$.

If $2<p\le \frac{n}{2\sigma}$, then $n>4\sigma$. Integrating both sides of \eqref{eq:conv-1} over $(0,\infty)$ with $p=2$ to show
\[
\int_0^\infty F_3(t)\,\ud t \le \frac{2}{n-4\sigma } \Big(C \int_0^\infty F_2(t)\,\ud t+F_2(0)\Big)<\infty,
\]
where we have used  \eqref{est:F_p_intermediate} for $p=2$.  For any $2\le p \le 3$,  using \eqref{eq:b-ineq}, we have $\int_0^\infty F_p(t)\,\ud t <\infty$. For any $2\le p\le \min\{3, n/(2\sigma)\}$, taking an increasing sequence  $t_j\to \infty$ as $j\to \infty$ such that $F_p(t_j)\to 0$, by \eqref{eq:conv-1} we have $\frac{\ud}{\ud t}F_p(t) \le C F_p(t)(1+F_p(t)^{1/p}) $ and thus \eqref{eq:seqtoall} holds with $F_2$ replaced by $F_p$ and $\rho(\tau)= \int_0^{\tau}(1+s^{1/p})^{-1}\,\ud s$. Hence, $F_p(t)\to 0$ as $t\to \infty$. If $n/(2\sigma)>3$, repeating this process, we have \eqref{est:F_p_intermediate} for $2\le p\le \min\{4, n/(2\sigma)\}$. By repeating such a process finite times, the lemma follows.
\end{proof}

By Lemma \ref{lem:F_p_mid}, integrating \eqref{dF_2_I} over $(0,\infty)$ when $n<4\sigma$ and using the estimates of $I_4$ with $0<\epsilon \le \frac{n-2\sigma}{4}Y_\sigma(\Sn)$, otherwise integrating the first identity of \eqref{dF_p_rearrangement} over $(0,\infty)$ with $p=2$, we conclude that
\begin{equation}\label{est:int_G_2}
\int_0^\infty \int_{\Sn}(\alpha f-R_\sigma^g)P_\sigma^g(\alpha f-R_\sigma^g)\,\ud V_g \ud t <\infty
\end{equation}
and then by \eqref{eq:sobolev}
\begin{equation}\label{est:induction1}
\int_0^\infty\left(\int_{\Sn}|\alpha f-R_\sigma^g|^{\frac{2n}{n-2\sigma}}\,\ud V_g\right)^{\frac{n-2\sigma}{n}}\,\ud t<\infty.
\end{equation}

\begin{lemma}\label{lemma1.13}
For $p> n/(2\sigma)$, there holds
$$\frac{\ud}{\ud t}F_p(t)
+\kappa F_{\frac{pn}{n-2\sigma}}(t)^{\frac{n-2\sigma}{n}}\leq CF_p(t)(1 +F_p(t)^{\frac{2\sigma}{2\sigma p-n}}),$$
where $\kappa$ is a positive constant depending only on $n$ and $\sigma$.
\end{lemma}
\begin{proof}
By \eqref{eq:alpha'}  and Lemma \ref{lem:F_p_mid}, we have 
\begin{equation*}
|\alpha'|\leq CF_2(t)+CF_2(t)^{\frac{1}{2}}
\leq CF_2(t)^{\frac{1}{2}}.
\end{equation*}
 Going back to the first identity of \eqref{dF_p_rearrangement}, we have
\begin{align}
\frac{\ud}{\ud t}F_p(t)\le &-\frac{p(n-2\sigma)}{4}\int_{\Sn}|\alpha f-R_\sigma^g|^{p-2}(\alpha f-R_\sigma^g)P_\sigma^g(\alpha f-R_\sigma^g)\ud V_g \nonumber  \\
& +C F_{p+1}(t)+ CF_p(t).
\label{dF_p_rearrangement-1}
\end{align}
By Proposition \ref{lemma1.11} and Sobolev inequality \eqref{eq:sobolev} we have
\begin{align*}
 &-\frac{p(n-2\sigma)}{4}G_p(t)\\
\ge& -\frac{(p-1)(n-2\sigma)}{ p} \int_{\Sn} |\alpha f-R_\sigma^g|^{\frac{p}{2}} P^g_\sigma(|\alpha f-R_\sigma^g|^{\frac{p}{2}})\,\ud V_g\\
&-\frac{(p-2)^2(n-2\sigma)}{4p}\left(
\int_{\Sn} \alpha f|\alpha f-R_\sigma^g|^p \,\mathrm{d} V_{g}
-\int_{\Sn} (\alpha f-R_\sigma^g) |\alpha f-R_\sigma^g|^p \,\mathrm{d} V_{g}\right)\\
\geq&  -\frac{(p-1)(n-2\sigma)}{p}Y_\sigma(\Sn)\left( \int_{\Sn} |\alpha f-R_\sigma^g|^{\frac{pn}{n-2\sigma}}\,\ud V_g\right)^{\frac{n-2\sigma}{n}}-CF_{p+1}(t)-C F_p(t).
\end{align*}
Notice that $p> n/(2\sigma)$, $\frac{np}{n-2\sigma}-(p+1)>\frac{2\sigma}{n-2\sigma}>0$. Using H\"older's and Young's inequalities, for any $\epsilon>0$,
we can find  $C_\epsilon>0$ such that
$$F_{p+1}(t)\leq F_{\frac{pn}{n-2\sigma}}(t)^{\frac{n-2\sigma}{2p\sigma}}F_{p}(t)^{\frac{2(p+1)\sigma-n}{2p\sigma}}
\leq \epsilon F_{\frac{pn}{n-2\sigma}}(t)^{\frac{n-2\sigma}{n}}+C_\epsilon F_{p}(t)^{\frac{2(p+1)\sigma-n}{2p\sigma-n}}.$$
Choosing  $\epsilon $ small enough, by \eqref{eq:b-ineq} we obtain
\[
\frac{\ud}{\ud t}F_p(t)+ \frac{1}{2} (n-2\sigma)^2Y_\sigma(\Sn) F_{\frac{pn}{n-2\sigma}}(t)^{\frac{n-2\sigma}{n}} \le   CF_p(t)(1+F_p(t)^{2\sigma/(2\sigma p-n)}).
\]
Therefore, the lemma follows with $\kappa= \frac{1}{2} (n-2\sigma)^2Y_\sigma(\Sn)$.
\end{proof}

\begin{lemma}\label{lemma1.14}
For any $p \geq 2$, there holds $F_p(t) \to 0$ as $t \to \infty$.
\end{lemma}
\begin{proof}
First we claim that there exist $p_0>n/(2\sigma)$ and
$\nu_0\in (0,1]$ such that
\begin{equation}\label{1.200}
\int_0^\infty\left(\int_{\Sn}|\alpha f-R_\sigma^g|^{p_0}\ud V_g\right)^{\nu_0}\ud t<\infty.
\end{equation}

Indeed, if $n< 4\sigma$, we choose $p_0=2$ and $\nu_0=1$ by Lemma \ref{lem:F_p_mid}.

If $n=4\sigma$, we choose $p_0=2n/(n-2\sigma)=4>2$ and $\nu_0=(n-2\sigma)/n=1/2$ by estimate \eqref{est:induction1}.

If $n>4\sigma$, then for any $p<n/(2\sigma)$, integrating \eqref{eq:conv-1} over $(0,\infty)$ and using Lemma \ref{lem:F_p_mid}, we have
\[
\int_0^\infty F_{p+1}(t)\,\ud t \le \frac{2}{n-2\sigma p } \Big(C \int_0^\infty F_p(t)\,\ud t+  F_p(0)\Big)<\infty.
\]
This together with Lemma \ref{lem:F_p_mid} implies that for any $p<1+\frac{n}{2\sigma}$, $\int_0^\infty F_p(t)\,\ud t<\infty$. Hence, \eqref{1.200} is proved.

By Lemma \ref{lem:F_p_mid}, we only need to consider  $p>n/(2\sigma)$.
Let $p_0$, $\nu_0$ be given in \eqref{1.200} and $p_k=p_0\left(\frac{n}{n-2\sigma}\right)^k$, $\nu_k =\frac{n-2\sigma}{n}$ for $k\in\mathbb{N}_+$. We claim that for all $k\in \mathbb{N}$,
\begin{equation}\label{1.100}
\int_0^\infty F_{p_k}(t)^{\nu_k}\ud t<\infty \quad \mathrm{and}\quad \lim_{t\to\infty} F_{p_k}(t) =0.
\end{equation}
Suppose \eqref{1.100} holds for $k$, and we will prove that it holds for $k+1$.
By Lemma \ref{lemma1.13} with $p=p_k$, we have 
$$\frac{\ud}{\ud t}F_{p_k}(t)
+\theta F_{p_{k+1}}(t)^{\nu_{k+1}}\leq CF_{p_k}(t)^{\nu_k}.$$
Integrating the above differential inequality over $(0,\infty)$, by  \eqref{1.100} we have
\begin{equation}
\label{eq:k-integral}
\int_0^\infty F_{p_{k+1}}(t)^{\nu_{k+1}}\,\ud t \le \frac{1}{\theta} \left(C \int_0^\infty F_{p_{k}}(t)^{\nu_k}\,\ud t+F_{p_k}(0)\right)<\infty.
\end{equation}
By Lemma \ref{lemma1.13} with $p=p_{k+1}$, we have
\begin{equation}\label{1.103}
\frac{\ud}{\ud t}F_{p_{k+1}}(t)
\leq CF_{p_{k+1}}(t)^{\nu_{k+1}}(F_{p_{k+1}}(t)^{\beta_1}+F_{p_{k+1}}(t)^{\beta_2}),
\end{equation}
where 
$$\beta_1=1-\nu_{k+1}=\frac{2\sigma}{n} \quad \mathrm{and~~}
\beta_2=\beta_1
+\frac{2p\sigma}{2p\sigma-n}>\beta_1.$$
Let $H_1(t)=\rho(F_{p_{k+1}}(t))$, where $\rho(t)=\int_0^{t}(s^{\beta_1}+s^{\beta_2})^{-1}\ud s$. By (\ref{eq:k-integral}), we choose $\{t_j\}$ to be an increasing sequence with $t_j\to\infty$ as $j\to\infty$ such that $
F_{p_{k+1}}(t_j)\to 0$ as $j\to\infty$. By \eqref{1.103}, we have
$$H_1(t)\leq H_1(t_j)+C\int_{t_j}^\infty F_{p_{k+1}}(t)^{\nu_k}\, \ud t \quad \forall~ t_j\le t\le  t_{j+1}.$$
Thus, $\lim_{t\to\infty}H_1(t)=0$. Since $\rho(\cdot)$ is continuous, increasing and $\lim_{\tau\searrow 0}\rho(\tau)=0$, we have $F_{p_{k+1}}(t)\to \infty$ as $t \to \infty$. From this together with \eqref{eq:k-integral}, we prove \eqref{1.100} for $k+1$. Hence, \eqref{1.100} holds for all $k\in \mathbb{N}$.

By Lemma \ref{lem:F_p_mid}, \eqref{1.100} and the basic inequality \eqref{eq:b-ineq},  the proof is complete.
\end{proof}

\begin{lemma}\label{lemma1.15}
There hold $G_2(t)\to 0$ as $t\to\infty$, and
\[
\int_0^\infty\int_{\Sn}|P_\sigma^g(\alpha f-R_\sigma^g)|^2 \,\ud V_g \,\ud t<\infty.
\]
\end{lemma}

\begin{proof}
By (\ref{conformal_invariance}), (\ref{1.7}) and Proposition \ref{prop:properties}, we have
\begin{align*}
&\frac{\ud}{\ud t}G_2(t)
=\frac{\ud}{\ud t}\left(\int_{\Sn}u(\alpha f-R_\sigma^g)
P_\sigma(u(\alpha f-R_\sigma^g)) \ud V_{g_{\Sn}}\right)\\
=&2\int_{\Sn}\frac{\partial}{\partial t}(u(\alpha f-R_\sigma^g))
P_\sigma(u(\alpha f-R_\sigma^g)) \ud V_{g_{\Sn}}\\
=&\frac{n-2\sigma}{2}\int_{\Sn}(\alpha f-R_\sigma^g)^2
P_\sigma^g(\alpha f-R_\sigma^g) \ud V_{g}\\
&+2\int_{\Sn}\left[-\frac{n-2\sigma}{4}P_\sigma^g(\alpha f-R_\sigma^g)+\frac{n+2\sigma}{4}R_\sigma^g(\alpha f-R_\sigma^g)+\alpha' f\right]
P_\sigma^g(\alpha f-R_\sigma^g) \ud V_{g}\\
=&-2\sigma\int_{\Sn}(\alpha f-R_\sigma^g)^2
P_\sigma^g(\alpha f-R_\sigma^g) \ud V_{g}-\frac{n-2\sigma}{2}\int_{\Sn}|P_\sigma^g(\alpha f-R_\sigma^g)|^2 \ud V_{g}
\\
&+\frac{n+2\sigma}{2}\int_{\Sn}\alpha f(\alpha f-R_\sigma^g)
P_\sigma^g(\alpha f-R_\sigma^g) \ud V_{g}
+2\alpha' \int_{\Sn}f P_\sigma^g(\alpha f-R_\sigma^g) \ud V_{g}\\
=:&J_1+J_2+J_3+J_4.
\end{align*}
By Young's inequality, for any $\epsilon>0$ we estimate
\begin{align*}
|J_1|\leq \epsilon \int_{\Sn}|P_\sigma^g(\alpha f-R_\sigma^g)|^2 \ud V_{g}+C_\epsilon F_4(t)
\end{align*}
and by H\"older's inequality and Lemma \ref{lemma1.14}
$$F_4(t)\leq F_{\frac{n}{\sigma}}(t)^{\frac{2\sigma}{n}}F_{\frac{2n}{n-2\sigma}}(t)^{\frac{n-2\sigma}{n}}\leq CF_{\frac{2n}{n-2\sigma}}(t)^{\frac{n-2\sigma}{n}}.$$
By Lemma \ref{lemma1.4}, \eqref{1.9} and using Young's inequality, we have for any $\epsilon>0$
\begin{equation*}
J_3\leq \epsilon \int_{\Sn}|P_\sigma^g(\alpha f-R_\sigma^g)|^2 \ud V_{g}+C_\epsilon
F_2(t)
\end{equation*}
and
\begin{equation*}
J_4\leq \epsilon \int_{\Sn}|P_\sigma^g(\alpha f-R_\sigma^g)|^2 \ud V_{g}+C|\alpha'|^2\\
\leq \epsilon \int_{\Sn}|P_\sigma^g(\alpha f-R_\sigma^g)|^2 \ud V_{g}+CF_2(t).
\end{equation*}
Choosing $\epsilon $ sufficiently small, we have
\begin{equation}\label{1.73}
\frac{\ud}{\ud t}G_2(t)+\frac{n-2\sigma}{4}
\int_{\Sn}|P_\sigma^g(\alpha f-R_\sigma^g)|^2 \ud V_{g}
\leq CF_2(t)+CF_{\frac{2n}{n-2\sigma}}(t)^{\frac{n-2\sigma}{n}}.
\end{equation}
By \eqref{est:int_G_2} we can find an increasing sequence $\{t_j\}$ with $t_j\to\infty$ as $j\to\infty$ such that $G_2(t_j) \to 0$.
Integrating (\ref{1.73}) over $(t_j,t)$, we obtain
$$G_2(t)\leq G_2(t_j)
+C\int_{t_j}^\infty F_2(s)\,\ud s+C\int_{t_j}^\infty F_{\frac{2n}{n-2\sigma}}(s)^{\frac{n-2\sigma}{n}}\,\ud s$$
and then $G_2(t)\to 0$ as $t\to \infty$ due to \eqref{est:induction1} and Lemma \ref{lem:F_p_mid}. By integrating \eqref{1.73} over $(0,\infty)$, the second conclusion follows.
\end{proof}

\section{Blow-up analysis}\label{Sect:Blow-up_anal}

Define
$$C^\infty_\ast :=\left\{0<u\in C^\infty(\Sn); \int_{\Sn}u^{\frac{2n}{n-2\sigma}}\ud V_{g_{\Sn}}=\omega_n\right\}.$$

\begin{lemma}\label{lemma3.3}
Let $u$ be a positive smooth solution of \eqref{1.7} with initial datum $u_0 \in C_\ast^\infty$. Then for any $t_k\to\infty$ as $k\to \infty$, $\{u_k:=u(t_k)\}$ is a Palais-Smale sequence of $E_f$ in $H^\sigma(\Sn)$.
\end{lemma}
\begin{proof}
By items (1) and (3) of Proposition \ref{prop:properties}, $\{u_k\}$ is uniformly bounded in $H^\sigma(\Sn)$ and $E_f[u_k]\to e_\infty$ for some $e_\infty>0$. Hence, it remains to show $\ud E_f[u_k]\to 0$. To that end,
for any $\varphi\in H^\sigma(\Sn)$,
there holds
\begin{align*}
&\frac{1}{2}\left(\int_{\Sn}f u_k^{\frac{2n}{n-2\sigma}}\ud V_{g_{\Sn}}\right)^{\frac{n-2\sigma}{n}}
|\langle \ud E_f[u_k],\varphi \rangle|\\
=&\left|\int_{\Sn}\varphi P_\sigma (u_k)\ud V_{g_{\Sn}}-\alpha(t_k)\int_{\Sn}fu_k^{\frac{n+2\sigma}{n-2\sigma}}\varphi \ud V_{g_{\Sn}}\right|\\
=&\left|\int_{\Sn}(R_\sigma^{g_{(t_k)}}-\alpha(t_k) f)u_k^{\frac{n+2\sigma}{n-2\sigma}}\varphi \ud V_{g_{\Sn}}\right|\\
\leq&\left(\int_{\Sn}(R_\sigma^{g(t_k)}-\alpha(t_k) f)^{\frac{2n}{n+2\sigma}}u_k^{\frac{2n}{n-2\sigma}}\ud V_{g_{\Sn}}\right)^{\frac{n+2\sigma}{2n}}
\left(\int_{\Sn}|\varphi|^{\frac{2n}{n-2\sigma}}\ud V_{g_{\Sn}}\right)^{\frac{n-2\sigma}{2n}}\\
\leq&\|R_\sigma^{g(t_k)}-\alpha(t_k) f\|_{L^{\frac{2n}{n+2\sigma}}(\Sn,g(t_k))}\|\varphi\|_{H^\sigma(\Sn)}=o(1)\|\varphi\|_{H^\sigma(\Sn)},
\end{align*}
as $k\to\infty$ by Lemma \ref{lemma1.14}, where $g(t_k)=u_k^{4/(n-2\sigma)}g_{\Sn}$.  Therefore, the lemma is proved.
\end{proof}

\begin{lemma}[Concentration-compactness] \label{lem:c-c} Let $u$ be a positive smooth solution of \eqref{1.7} with $u_0 \in C_\ast^\infty$.
For any $t_k\to\infty$ as $k\to \infty$, let $u_k:=u(t_k)$. Then, after passing to a subsequence, there exist a non-negative integer $L$, a convergent sequence $\{x_{k,\nu}\}\subset \Sn$ and a non-negative smooth function $u_\infty$, a sequence of real numbers $\{\lambda_{k,\nu}\}$ with  $\lda_{k,\nu}\to \infty$  and  $\alpha(t_k)\to \alpha_\infty$ as $k\to \infty$ for any fixed $\nu=1,2,\cdots, L$ such that  
\begin{equation*}
u_k=\sum_{\nu=1}^L \bar u_{x_{k,\nu},\lda_{k,\nu}}+u_\infty+o(1) \quad \mathrm{in~~} H^\sigma(\Sn),
\end{equation*}
where
\[
\bar u_{x_{k,\nu},\lda_{k,\nu}}(x)=\Big(\frac{R_\sigma}{\alpha_\infty f(-x_{k,\nu})}\Big)^{\frac{n-2\sigma}{4\sigma}}
\Big(\frac{2\lda_{k,\nu}}{2+(\lda_{k,\nu}^2-1)(1- \langle x,x_{k,\nu}\rangle)}\Big)^{\frac{n-2\sigma}{2}}
\]
satisfies 
\begin{equation}\label{eq:bubble_u_k}
P_\sigma (\bar u_{x_{k,\nu},\lda_{k,\nu}})=\alpha_\infty f(-x_{k,\nu}) \bar u_{x_{k,\nu},\lda_{k,\nu}}^{\frac{n+2\sigma}{n-2\sigma}} \quad \mathrm{on~~} \Sn,
\end{equation}
and $u_\infty$ satisfies
\begin{equation}\label{eq:u-infty}
P_\sigma u_\infty =\alpha_\infty f u_\infty^{\frac{n+2\sigma}{n-2\sigma}} \qquad \mathrm{on~~} \Sn.
\end{equation}

\end{lemma}

\begin{proof} 
Thanks to Lemma \ref{lemma3.3}, the proof of Lemma \ref{lem:c-c} is standard. Indeed, one can use Theorem \ref{thm:lf} and the proofs in Fang-Gonz\'alez \cite{FGon}. We omit the details.
\end{proof}

In view of (\ref{sbc}), we can choose a small
$\epsilon_0>0$ such that
\begin{equation}\label{epsilon_0}
\left[(1+\epsilon_0)\frac{\max_{\Sn}f}{\min_{\Sn}f}\right]^{\frac{n-2\sigma}{n}}< 2^{\frac{2\sigma}{n}}.
\end{equation}
We choose
\begin{equation}\label{beta}
\beta=R_\sigma\omega_n^{\frac{2\sigma}{n}}(1+\epsilon_0)^{\frac{n-2\sigma}{n}}(\min_{\Sn}f)^{-\frac{n-2\sigma}{n}}
\end{equation}
 and define
$$C_f^\infty:=\{u\in C^\infty_*; E_f[u]\leq \beta\}.$$

\begin{lemma}\label{lemma3.1}
 Let $u$ be a positive smooth solution of \eqref{1.7} with $u(0)=u_0 \in C_f^\infty$. For any $t_k\to\infty$ as $k\to \infty$, let $u_k:=u(t_k)$. Suppose that $f$ satisfies \eqref{sbc}, i.e.,
$\max_{\Sn} f<2^{2\sigma/(n-2\sigma)}\min_{\Sn} f$. Let $L$ be the nonnegative integer defined in Lemma \ref{lem:c-c}. Then $L\le 1$.
\end{lemma}
\begin{proof}
Suppose $L>0$ otherwise it is trivial. By Lemmas \ref{lemma1.14}, \ref{lem:c-c} and \eqref{eq:bubble_measures}, we have
\begin{align*}
\int_{\Sn} |R_\sigma^{g_{(t_k)}}|^{\frac{n}{2\sigma}}\,\ud V_{g_{(t_k)}} =& \int_{\Sn} |\alpha(t_k) f|^{\frac{n}{2\sigma}}\,\ud V_{g_{(t_k)}}+o(1)\\
=&\sum_{\nu=1}^L  \int_{\Sn} |\alpha(t_k) f|^{\frac{n}{2\sigma}} \bar u_{x_{k,\nu},\lambda_{k,\nu}}^{\frac{2n}{n-2\sigma}}\,\ud V_{g_{\Sn}} +  \int_{\Sn} |\alpha(t_k) f|^{\frac{n}{2\sigma}} u_\infty^{\frac{2n}{n-2\sigma}}\,\ud V_{g_{\Sn}}+o(1)\\
\geq& L R_\sigma^{\frac{n}{2\sigma}} \omega_n +o(1).
\end{align*}
On the other hand,
\begin{align*}
\alpha(t_k)\left(\int_{\Sn}|f|^{\frac{n}{2\sigma}}\ud V_{g_{(t_k)}}\right)^{\frac{2\sigma}{n}}
&\leq
E_f[u_0]\left(\int_{\Sn}f\ud V_{g_{(t_k)}}\right)^{-\frac{2\sigma}{n}}\left(\int_{\Sn}|f|^{\frac{n}{2\sigma}}\ud V_{g_{(t_k)}}\right)^{\frac{2\sigma}{n}}\\
&\leq \beta \left(\int_{\Sn}f\ud V_{g_{(t_k)}}\right)^{-\frac{2\sigma}{n}} \Big(\max_{\Sn}f\Big)^{\frac{n-2\sigma}{n}} \left(\int_{\Sn}f\ud V_{g_{(t_k)}}\right)^{\frac{2\sigma}{n}}\\
&=R_\sigma\omega_n^{\frac{2\sigma}{n}}\left[(1+\epsilon_0)\frac{\max_{\Sn}f}{\min_{\Sn}f}\right]^{\frac{n-2\sigma}{n}}\\
&< 2^{\frac{2\sigma}{n}}R_\sigma\omega_n^{\frac{2\sigma}{n}},
\end{align*}
where the first inequality follows from item (3) of Proposition \ref{prop:properties},
the second inequality follows from (\ref{sbc}) and (\ref{beta}),
and the last inequality follows from (\ref{epsilon_0}). It follows by sending $k\to \infty$ that
$$L^{\frac{2\sigma}{n}} R_\sigma < 2^{\frac{2\sigma}{n}}R_\sigma.$$ 
Namely, $L<2$. We complete the proof.
\end{proof}

For $p \geq 1$, we define
\begin{equation}\label{def:fractional_Sobolev_space}
H^{2\sigma,p}(\Sn):=\{u \in H^\sigma(\Sn); P_\sigma u \in L^p(\Sn)\}.
\end{equation}
In particular, we set $H^{2\sigma}(\Sn)=H^{2\sigma,2}(\Sn)$. For $s>0$ and $1 <p <\infty$, we first state an embedding theorem for $H^{s,p}(\Sn)$ (see \cite[p.1587]{Jin&Li&Xiong2}) as follows: if $sp<n$, then the embedding $H^{s,p}(\Sn)\hookrightarrow L^\frac{np}{n-sp}(\Sn)$ is continuous, and $H^{s,p}(\Sn)\hookrightarrow L^q(\Sn)$ is compact for all $q< \frac{np}{n-sp}$; if $0<s-\frac{n}{p}<1$, then the embedding  $H^{s,p}(\Sn)\hookrightarrow C^{s-\frac{n}{p}}(\Sn)$ is continuous.

\begin{lemma}\label{lem:no-bubble} 
Assume as in Lemma \ref{lemma3.1}. If $L=0$, then as $k \to \infty$, up to a subsequence, $u_k\to u_\infty$ in $H^{2\sigma, p}(\Sn)$, $R^{g_{(t_k)}}_\sigma\to \alpha_\infty f$ in $L^p(\Sn)$ for any $p\ge 1$ and $u_\infty>0$ is a smooth solution of \eqref{eq:u-infty}.
\end{lemma}

\begin{proof} Since $L=0$, by Lemma \ref{lem:c-c}, after passing to a subsequence, $u_k\to u_\infty$ in $H^{\sigma}(\Sn)$ as $k \to \infty$. It follows that $u_\infty\ge 0$ and is not identically zero. By \eqref{eq:r1} and the strong maximum principle for $(-\Delta)^\sigma$ in $\mathbb{R}^n$ (see Silvestre \cite[Proposition 2.17]{Silvestre}),  we obtain $u_\infty>0$. By the regularity theory (see \cite{Jin&Li&Xiong1}), $u_\infty\in C^2(\Sn)$.

For any domain $\om\subset \Sn$, we have
\begin{align}
&\int_{\om} |R^{g_{(t_k)}}_\sigma u_k^{\frac{4\sigma}{n-2\sigma}}|^{\frac{n}{2\sigma}}\,\ud V_{g_{\Sn}}\no\\
\le& 2^{\frac{n-2\sigma}{2\sigma}}\int_{\om}  \left[|R^{g_{(t_k)}}_\sigma-\alpha(t_k) f|^{\frac{n}{2\sigma}} u_k^{\frac{2n}{n-2\sigma}}+(\alpha(t_k) f)^{ \frac{n}{2\sigma} } \left(u_k^{\frac{2n}{n-2\sigma}}- u_\infty^{\frac{2n}{n-2\sigma}}+u_\infty^{\frac{2n}{n-2\sigma}}\right) \right]\,\ud V_{g_{\Sn}} \nonumber\\
=&o(1)+ 2^{\frac{n-2\sigma}{2\sigma}}\int_{\om}  (\alpha(t_k) f)^{ \frac{n}{2\sigma} } u_\infty^{\frac{2n}{n-2\sigma}} \,\ud V_{g_{\Sn}}\nonumber \\
\le& o(1)+C|\om|,
\label{eq:smallness}
\end{align}
where $|\om|=\mathrm{Vol}(\Omega, g_{\Sn})$, $C>0$ is independent of $k$, and we have used Lemma \ref{lemma1.14} and $u_k\to u_\infty$ in $L^{2n/(n-2\sigma)}(\Sn)$ as $k\to \infty$. Applying Moser's iteration to the equation 
$$P_\sigma u_k=(R^{g_{(t_k)}}_\sigma u_k^{\frac{4\sigma}{n-2\sigma}}) u_k \quad \mathrm{on~~} \Sn$$ 
and making use of \eqref{eq:smallness} together with a finite covering of $\Sn$, we have 
$$\|u_k\|_{L^{\bar p}(\Sn)}\le C\|u_k\|_{L^\frac{2n}{n-2\sigma}(\Sn)}\le C$$ 
for some constants $\bar p>\frac{2n}{n-2\sigma}$ and $C>0$  which are independent of $k$; see \cite{Jin&Li&Xiong1}. Choose $\delta,q>0$ such that $\frac{2n+2\delta}{n-2\sigma}<\bar p$, $\frac{2n}{n-2\sigma} \frac{1+q}{q} \le \bar p$ and $q\delta<n$. By H\"older's inequality, we obtain
\begin{align*}
&\int_{\Sn} |R^{g_{(t_k)}}_\sigma u_k^{\frac{4\sigma}{n-2\sigma}}|^{\frac{n+\delta}{2\sigma}}\,\ud V_{g_{\Sn}}\\
=& \int_{\Sn} |R^{g_{(t_k)}}_\sigma|^{\frac{n+\delta}{2\sigma}} u_k^{\frac{2n-2q \delta}{n-2\sigma}}  u_k^{\frac{2(1+q)\delta}{n-2\sigma}}\,\ud V_{g_{\Sn}}  \\
\le& \left(\int_{\Sn} |R^{g_{(t_k)}}_\sigma|^{\frac{n(n+\delta)}{2\sigma(n-q \delta)}}  \,\ud V_{g_{(t_k)}}\right)^{\frac{n-q \delta}{n}}  \left(\int_{\Sn} u_k^{\frac{2n}{n-2\sigma} \frac{1+q}{q}} \,\ud V_{g_{\Sn}}\right)^{\frac{q \delta}{n}}
\le C.
\end{align*}
Applying Moser's iteration again, by the above inequality we have
\begin{equation} \label{eq:holderestimates}
\|u_k\|_{C^\gamma(\Sn)}\le C
 \end{equation} for some constants $\gamma\in (0,1)$ and $C>0$ which are independent of $k$. It follows that, after passing to a subsequence, $u_k \to u_\infty$ uniformly on $\Sn$ as $k\to \infty$. Hence, for large $k$, we have $0<\frac12 \min_{\Sn} u_\infty \le u_k \le 2 \max_{\Sn} u_\infty<\infty$. By Lemma \ref{lemma1.14}, we conclude that as $k \to \infty$
 \[
 R^{g_{(t_k)}}_\sigma -\alpha(t_k) f\to 0 \quad \mathrm{in~~}L^p(\Sn)
  \]
for any $ 2\le p<\infty$. By H\"older's inequality, this also holds for $p<2$.  Together with \eqref{eq:holderestimates}, we have $P_\sigma u_k=R^{g_{(t_k)}}_\sigma u_k^{(n+2\sigma)/(n-2\sigma)} \in L^p(\Sn)$ for any $p>1$. Hence, $u_k \in H^{2\sigma,p}(\Sn)$. By the compactness, after passing to a subsequence,  $u_k\to u_\infty$ in $H^{2\sigma,p}(\Sn)$ for any $1\le p<\infty$.

 Therefore, we complete the proof.
\end{proof}

As a direct consequence of Lemma \ref{lem:no-bubble}, if $u_\infty>0$, then it is well-done. Thus, we now assume $u_\infty=0$. Then it follows from Lemma \ref{lemma3.1} that there holds
\begin{equation} \label{eq:one-bubble}
u_k= \bar u_{x_{k},\lda_{k}}+o(1) \quad \mathrm{~~in~~} H^\sigma(\Sn),
\end{equation}
where $\lda_k\to \infty$ as $k\to \infty$ and
\[
\bar u_{x_{k},\lda_{k}}(x)=b_k
\Big(\frac{2\lda_{k}}{2+(\lda_{k}^2-1)(1-\langle x,x_{k}\rangle)}\Big)^{\frac{n-2\sigma}{2}}:=b_k u_{x_{k},\lambda_k}(x) , \quad x \in \Sn
\] for some $x_k\in \Sn$ and $\lambda_k>1$, and $b_k=(\frac{R_\sigma}{\alpha_\infty f(-x_{k})})^{(n-2\sigma)/4\sigma}$. In particular, let $\phi_{x_k,\lambda_k}=\Psi_{x_k}\circ \delta_{\lambda_k}\circ \Psi_{x_k}^{-1}$ be a conformal transformation on $\Sn$ such that
$$\phi_{x_k,\lambda_k}^\ast(g_{\Sn})=u_{x_k,\lambda_k}^{\frac{4}{n-2\sigma}}g_{\Sn},$$
where $\Psi^{-1}_{x_k}$ is the stereographic projection on $\Sn$ from $x_k$. Furthermore, for $k$ sufficiently large we have
\begin{equation}\label{eq:bubble_measures}
\left|\fint_{\Sn}\varphi u_{x_k,\lambda_k}^{\frac{2n}{n-2\sigma}}\ud V_{\Sn}-\varphi(-x_k)\right|=o(1)
\end{equation}
for any $\varphi \in C(\Sn)$. Up to a subsequence, let $b_k\to b_\infty$ as $k\to \infty$.

For $t\geq 0$, let
$$P(t)=\int_{\Sn}x\,\ud V_{g(t)}$$
be the center of mass of $g(t)$, and whenever $P(t)\neq 0$, let
$$p=p(t)=P/|P|\in \Sn$$
be its image under the radial projection. 
Clearly $p(t)$ smoothly depends on the time $t$ if $u$ does. 
For smooth metric $g(t)=u(t)^{4/(n-2\sigma)}g_{\Sn}$, there exists a family of
conformal diffeomorphisms $\phi(t)=\phi_{q(t),r(t)}:\Sn\to \Sn$, with $q(t)\in \Sn$ and $r(t)\geq 1$,
which are explicitly given by
\begin{equation} \label{eq:m-diffeo}
\phi_{q,r}(x)=\frac{2r(x - \langle q, x\rangle q) +[ r^2(1 +
\langle q, x\rangle) -(1 - \langle q, x\rangle )]q}{ r^2( 1 + \langle q, x\rangle) + (1 -
\langle q, x\rangle)}
\end{equation}
such that
\begin{equation}\label{2.2}
\int_{\Sn}x\,\ud V_{h}=0\quad \mathrm{ for~~all~~ }t>0,
\end{equation}
where
the new metric
$$h=\phi(t)^*(g)=v(t)^{\frac{4}{n-2\sigma}}g_{\Sn}$$
is called the normalized metric, where
$v(t)=(u(t)\circ\phi(t))|\det \ud \phi(t)|^{(n-2\sigma)/2n}$
and $\ud V_h=v(t)^{2n/(n-2\sigma)}\ud V_{g_{\Sn}}.$
Meanwhile, the normalized function $v$ satisfies
\begin{equation}\label{2.3}
P_\sigma(v)=R^h_\sigma v^{\frac{n+2\sigma}{n-2\sigma}}\qquad \mathrm{on~~} \Sn,
\end{equation}
where $R^h_\sigma=R^g_\sigma\circ\phi(t)$ is $2\sigma$-order $Q$-curvature
of $h$. From now on, we set $f_\phi=f\circ\phi(t)$.
It follows from \cite[Lemma 3.1]{Jin&Li&Xiong2} that $v$ enjoys the following properties:
\begin{equation}\label{2.4}
\int_{\Sn}vP_\sigma(v)\,\ud V_{g_{\Sn}}=\int_{\Sn}uP_\sigma(u) \,\ud V_{g_{\Sn}}
\mathrm{~~and~~}\fint_{\Sn}v^{\frac{2n}{n-2\sigma}}\,\ud V_{g_{\Sn}}=\fint_{\Sn}u^{\frac{2n}{n-2\sigma}}\,\ud  V_{g_{\Sn}}=1.
\end{equation}

\begin{lemma} \label{lem:one-bubble-1} 
Assume as in Lemma \ref{lemma3.1}. Assume further that $L=1$ and $u_\infty=0$. Then there exists a sequence of  conformal transformations  $\phi_k$ as \eqref{eq:m-diffeo} on $\Sn$ such that
\begin{align*}
v_k\to 1 \quad \mathrm{~~in~~}H^{2\sigma,p}(\Sn) \quad \mathrm{for~~any~~}1\le p<\infty,
\end{align*}
where $v_k =|\det \ud \phi_k|^{(n-2\sigma)/2n} u_k\circ \phi_k $ satisfies \eqref{2.2}.
\end{lemma}
\begin{proof} By our assumption, we have \eqref{eq:one-bubble}. Let $P_k=\displaystyle\fint_{\Sn}x u_k^{2n/(n-2\sigma)}\,\ud V_{g_{\Sn}}$. Again by \eqref{eq:one-bubble},
\[
P_k= \fint_{\Sn}x \bar u_{x_k,\lda_k}^{\frac{2n}{n-2\sigma}}\,\ud V_{g_{\Sn}}+o(1)=-x_k+o(1)\neq 0.
\]
Let $\phi_k$ be the conformal transformations on $\Sn$ given by \eqref{eq:m-diffeo}
and $\bar v_k= |\det \ud \phi_k|^{(n-2\sigma)/2n} \bar u_{x_k,\lambda_k} \circ\phi_k$, then by \eqref{eq:bubble_u_k} we have
\begin{equation}\label{eq:bubble_v_k}
P_\sigma (\bar v_k)=\alpha_\infty f(-x_k) \bar v_k^{\frac{n+2\sigma}{n-2\sigma}} \quad \mathrm{on~~} \Sn.
\end{equation}

 Let $\Psi^{-1}$ be the stereographic projection with $P_k/|P_k|$ as the south pole. By \eqref{1.9}, \eqref{2.2} and \eqref{eq:one-bubble}, we obtain as $k \to \infty$
\begin{equation}\label{eq:one-bubble_center-zero}
\int_{\Sn} x \bar v_k^{\frac{2n}{n-2\sigma}}\,\ud V_{g_{\Sn}}\to 0,
\end{equation}
equivalently, the classification theorem of \eqref{eq:bubble_v_k} for $\bar v_k$ gives
\begin{align*}
&\int_{\R^n} \frac{y}{1+|y|^2} \left(\frac{\lda_k/r_k}{1+|y- r_k y_k|^2\lda_k^2/r_k^2}\right)^{n}\,\ud y \to 0,\\&
\int_{\R^n} \frac{1-|y|^2}{1+|y|^2} \left(\frac{\lda_k/r_k}{1+|y-r_k y_k|^2\lda_k^2/r_k^2}\right)^{n}\,\ud y \to 0,
\end{align*}
where $y_k\to 0$ is $\Psi^{-1}(x_k)$ up to some uniform constant, and $r_k=r(t_k)$.
It forces $\lda_k/r_k \to 1$ and $r_k y_k\to 0$. Indeed, if after passing to a subsequence $\lda_k/r_k \to 0$ or $\infty$, then there exists $P_\infty \in \Sn$ such that $\int_{\Sn} x \bar v_k^{2n/(n-2\sigma)}\,\ud V_{g_{\Sn}}\to P_\infty$ or $-P_\infty$ along the subsequence, which contradicts \eqref{eq:one-bubble_center-zero}. Hence, after passing to a subsequence we assume  $\lda_k/r_k \to \bar c\in (0,\infty)$. If $r_k |y_k|\to \infty$, by Lebesgue dominated convergence theorem we have, as $k\to \infty$, 
\begin{align*}
\int_{\R^n} \frac{1-|y|^2}{1+|y|^2} \left(\frac{\lda_k/r_k}{1+|y-r_k y_k|^2\lda_k^2/r_k^2}\right)^{n}\,\ud y&= \int_{\R^n} \frac{1-|y+r_k y_k|^2}{1+|y+r_k y_k|^2} \left(\frac{\lda_k/r_k}{1+|y|^2\lda_k^2/r_k^2}\right)^{n}\,\ud y\\& \to 
-\int_{\R^n} \left(\frac{\bar c}{1+\bar c^2 |y|^2}\right)^{n}\,\ud y.
\end{align*} 
This again yields a contradiction with \eqref{eq:one-bubble_center-zero}. 

Hence, after passing to a subsequence we assume  $\lda_k/r_k \to \bar c\in (0,\infty)$ and $r_k y_k \to y_0$ as $k \to \infty$. By \eqref{eq:one-bubble_center-zero}, we apply Lebesgue dominated convergence theorem again to obtain
\begin{align*}
&\int_{\R^n} \frac{y}{1+|y|^2} \left(\frac{\bar c}{1+\bar c^2 |y- y_0|^2}\right)^{n}\,\ud y = 0,\\&
\int_{\R^n} \frac{1-|y|^2}{1+|y|^2}\left( \frac{\bar c}{1+\bar c^2 |y- y_0|^2}\right)^{n}\,\ud y = 0.
\end{align*}
Thus, we claim that $y_0=0$ and $\bar c=1$. To that end, by contradiction, if $y_0\neq 0$, then without loss of generality we assume the first component of $y_0$ is positive. Let $\tilde y=(-y_1,\cdots,y_n)$ for $y=(y_1,\cdots,y_n)$, then
 \begin{align*}
 &\int_{\R^n} \frac{y_1}{1+|y|^2} \left(\frac{\bar c}{1+\bar c^2 |y- y_0|^2}\right)^{n}\,\ud y\\
 =&\int_{\{y_1>0\}} +\int_{\{y_1<0\}} \frac{y_1}{1+|y|^2} \left(\frac{\bar c}{1+\bar c^2 |y- y_0|^2}\right)^{n}\,\ud y\\
 =&\int_{\{y_1>0\}} \frac{y_1}{1+|y|^2} \left[\left(\frac{\bar c}{1+\bar c^2 |y- y_0|^2}\right)^{n}- \left(\frac{\bar c}{1+\bar c^2 |\tilde y- y_0|^2}\right)^{n}\right]\,\ud y>0
 \end{align*}
by using  $|\tilde y|=|y|, |y-y_0|<|\tilde y-y_0|$ on $\{y_1>0\}$. This yields a contradiction. Now we have $y_0=0$ and need to show $\bar c=1$. Observe that
\begin{align*}
0=&\int_{\R^n} \frac{1-|y|^2}{1+|y|^2}\left( \frac{\bar c}{1+\bar c^2 |y|^2}\right)^{n}\,\ud y\\
=&\int_{\R^n} \frac{\bar c^2-|y|^2}{\bar c^2+|y|^2}\left( \frac{1}{1+|y|^2}\right)^{n}\,\ud y\\
=&\int_{B_1}\left(\frac{\bar c^2-|y|^2}{\bar c^2+|y|^2}+\frac{\bar c^2|y|^2-1}{\bar c^2 |y|^2+1}\right)\left( \frac{1}{1+|y|^2}\right)^{n}\,\ud y\\
=&2(\bar c^4-1)\int_{B_1}\frac{|y|^2}{(\bar c^2+|y|^2)(\bar c^2 |y|^2+1)}\left( \frac{1}{1+|y|^2}\right)^{n}\,\ud y.
\end{align*}
This forces $\bar c=1$. 

Therefore, with a positive constant $b_\infty$, $v_k \to b_\infty$ in $H^\sigma (\Sn)$ as $k \to \infty$. By Lemma \ref{lemma1.14}, for any $p \geq 1$ we have
\[
\int_{\Sn} |R^{g_{(t_k)}}_\sigma-\alpha(t_k) f_{\phi_k}|^p \,\ud V_{g_{(t_k)}} \to 0 \quad \mathrm{as~~}k\to \infty.
\]
Meanwhile, by the conservation of the volume \eqref{1.9}, we have
$$1=\fint_{\Sn}u_k^{\frac{2n}{n-2\sigma}}\ud V_{g_{\Sn}}=\fint_{\Sn}v_k^{\frac{2n}{n-2\sigma}}\ud V_{g_{\Sn}}\to b_\infty^{\frac{2n}{n-2\sigma}} \mathrm{~~as~~} k \to \infty,$$
which yields $b_\infty=1$.
Using the above facts and the same argument of \eqref{eq:holderestimates} we have
$$
\|v_k\|_{C^\gamma(\Sn)}\le C
$$
for some $\gamma \in (0,1)$ and $C>0$ independent of $k$. It follows that $\|v_k\|_{H^{2\sigma,p}(\Sn)}\le C(p)$ for any $1\le p<\infty$. It follows that, up to a further subsequence,
\begin{equation} \label{eq:one-bubble-1}
v_k \to 1 \quad \mathrm{in~~}H^{2\sigma,p}(\Sn) \mathrm{~~as~~} k \to \infty.
\end{equation}
This completes the proof.
\end{proof}

\begin{lemma}\label{lem:lim_seq_v}
Let $f$ be a positive smooth Morse function on $\Sn$ and satisfy condition \eqref{sbc}. 
Suppose $f$ cannot be realized as $2\sigma$-order $Q$-curvature of a conformal metric.
Let $u(t)$ be a positive smooth solution of the flow \eqref{1.7}, and $v(t)$
be the corresponding normalized flow defined in \eqref{2.3}. Then, as $t\to\infty$, we have $v(t)\to 1$,
$h(t)=v(t)^{4/(n-2\sigma)}g_{\Sn}\to g_{\Sn}$ in $H^{2\sigma,p}(\Sn)$ for all $p \geq 1$.
Furthermore, $\|\phi(t)-p(t)\|_{L^2(\Sn)}\to 0$; hence
also
$\|f\circ\phi(t)-f(p(t))\|_{L^2(\Sn)}\to 0$
and $\alpha(t) f(p(t))\to R_\sigma$.
\end{lemma}
\begin{proof}
Since the idea of the proof is nearly identical to the one of \cite[Lemma 4.11]{Chen&Xu}, we omit it here.
\end{proof}

Notice that
$$F_2(t)=\int_{\Sn}|\alpha f-R_\sigma^g|^2\ud V_g=\int_{\Sn}|\alpha f_\phi-R_\sigma^h|^2\ud V_h$$
and
$$G_2(t)=\int_{\Sn}(\alpha f-R_\sigma^g)P^g_\sigma(\alpha f-R_\sigma^g)\ud V_g=
\int_{\Sn}(\alpha f_\phi-R_\sigma^h)P^h_\sigma(\alpha f_\phi-R_\sigma^h)\ud V_h.$$

\begin{lemma}\label{lemma3.7}
With error $o(1)\to 0$ as $t\to\infty$, there holds
$$\frac{\ud}{\ud t}F_2\leq-\left(\frac{n-2\sigma}{2}+o(1)\right)G_2
+\left(\frac{n+2\sigma}{2}R_\sigma+o(1)\right)F_2.$$
\end{lemma}
\begin{proof}
By (\ref{1.39}) and \eqref{eq:alpha'}, we have
\begin{equation*}
\begin{split}
&\frac{\ud}{\ud t}F_2=-\frac{n-2\sigma}{2}G_2
+\frac{n+2\sigma}{2}\int_{\Sn}\alpha f_\phi|\alpha f_\phi-R_\sigma^h|^{2}\ud V_h+2\alpha'\int_{\Sn}f_\phi(\alpha f_\phi-R_\sigma^h)\ud V_h\\
&\qquad\quad~~-\sigma\int_{\Sn}|\alpha f_\phi-R_\sigma^h|^{2}(\alpha f_\phi-R_\sigma^h)\ud V_h\\
=&-\frac{n-2\sigma}{2}G_2
+\frac{n+2\sigma}{2}\int_{\Sn}\alpha f_\phi|\alpha f_\phi-R_\sigma^h|^{2}\ud V_h\\
&+\frac{2\int_{\Sn}f_\phi(\alpha f_\phi-R_\sigma^h)\ud V_h}{\int_{\Sn}f_\phi \ud V_h}\left[-\frac{n-2\sigma}{2}\int_{\Sn}(R_\sigma^h-\alpha f_\phi)^2\ud V_{h}+\sigma\int_{\Sn}\alpha f_\phi (R_\sigma^h-\alpha f_\phi)\ud V_{h}\right]
\\
&-\sigma\int_{\Sn}|\alpha f_\phi-R_\sigma^h|^{2}(\alpha f_\phi-R_\sigma^h)\ud V_h\\
\leq&-\frac{n-2\sigma}{2}G_2
+\frac{n+2\sigma}{2}R_\sigma F_2+\frac{n+2\sigma}{2}(\alpha f(p(t))-R_\sigma)\int_{\Sn}|\alpha f_\phi-R_\sigma^h|^{2}\ud V_h\\
&+\frac{n+2\sigma}{2}\alpha\int_{\Sn}(f_\phi-f(p(t)))|\alpha f_\phi-R_\sigma^h|^{2}\ud V_h+C \int_{\Sn}|\alpha f_\phi-R_\sigma^h|^3 \ud V_h.
\end{split}
\end{equation*}
By Lemma \ref{lem:lim_seq_v}, we have
$\alpha f(p(t))\to R_\sigma$ as $t\to\infty$ and then
$$(\alpha f(p(t))-R_\sigma)\int_{\Sn}|\alpha f_\phi-R_\sigma^h|^{2}\ud V_h=o(1)F_2(t).$$
By H\"older's inequality, \eqref{eq:sobolev} and Lemma \ref{lemma1.14}, we estimate
\begin{align*}
\int_{\Sn}|\alpha f_\phi-R_\sigma^h|^3\ud V_h\leq& \|\alpha f_\phi-R_\sigma^h\|_{L^{\frac{n}{2\sigma}}(\Sn,h)}\|\alpha f_\phi-R_\sigma^h\|_{L^{\frac{2n}{n-2\sigma}}(\Sn,h)}^2\\
\leq& Y_\sigma(\Sn)^{-1}F_{\frac{n}{2\sigma}}(t)G_2(t)=o(1)G_2(t).
\end{align*}
By a similar argument in the proof of \cite[Lemma 6.1]{Chen&Xu}, using Lemma \ref{lemma5.1} and Sobolev inequality \eqref{eq:sobolev}  we may bound
\begin{align*}
&\left|\alpha\int_{\Sn}(f_\phi-f(p(t))|\alpha f_\phi-R_\sigma^h|^{2}\ud V_h\right|\\
\leq& \alpha\|f_\phi-f(p(t))\|_{L^{\frac{n}{2\sigma}}(\Sn,h)}
\|\alpha f_\phi-R_\sigma^h\|_{L^{\frac{2n}{n-2\sigma}}(\Sn,h)}^2
= o(1) G_2(t).
\end{align*}
Therefore, the desired estimate follows by collecting all the above estimates together.
\end{proof}

\section{Finite dimensional dynamics}\label{Sect:Finite_dim_dynamics}

We recall the following Kazdan-Warner identity (see \cite[Proposition A.1]{Jin&Li&Xiong1}).

\begin{prop}\label{Kazdan-Warner}
If $u \in C^2(\Sn)$ is a positive function satisfying \eqref{1.2.5}, then
\begin{equation*}
\int_{\Sn}\langle X,\nabla f\rangle_{g_{\Sn}} u^{\frac{2n}{n-2\sigma}}\ud V_{g_{\Sn}}=0
\end{equation*}
for any conformal Killing vector field $X$ on $\Sn$.
\end{prop}

In particular, if we take $X=\nabla_{g_{\Sn}}x_i$ for $i=1,2,\cdots,n+1$ in Proposition \ref{Kazdan-Warner}
where $(x_1,x_2,\cdots, x_{n+1})\in \Sn$, then by Proposition \ref{Kazdan-Warner} we have 
\begin{equation}\label{5.1}
\int_{\Sn}\langle \nabla x_i,\nabla R_\sigma^g\rangle_{g_{\Sn}} u^{\frac{2n}{n-2\sigma}}\ud V_{g_{\Sn}}=0\quad\mathrm{~~for~~ }i=1,2,\cdots,n+1,
\end{equation}
where $R_\sigma^g$ is $2\sigma$-order $Q$-curvature of $g=u^{4/(n-2\sigma)}g_{\Sn}$.

If we let $\xi=(\ud\phi)^{-1}\frac{\ud\phi}{\ud t}$, then there holds (see also \cite[(5.7)]{Chen&Xu})
\begin{equation}\label{5.2}
v_t=(u_t\circ\phi(t))|\det\ud\phi(t)|^{\frac{n-2\sigma}{2n}}+\frac{n-2\sigma}{2n}v\mathrm{div}_h(\xi),
\end{equation}
where
\begin{equation}\label{5.3}
v(t)=(u(t)\circ\phi(t))|\det\ud\phi(t)|^{\frac{n-2\sigma}{2n}} \quad \mathrm{with~~} \phi(t)^\ast g(t)=
|\det \ud \phi(t)|^{\frac{2}{n}}g_{\Sn} .
\end{equation}
Here we use Hadamard identity to give an alternative proof of \eqref{5.2}. The Hadamard identity states that: Let $\Omega$ be a domain in $\mathbb{R}^n$ and $F \in
C^2(\bar{\Omega};\mathbb{R}^n)$, denote by $B_{ij}$ the algebraic
cofactor of the element $\frac{\partial F^i} {\partial x^j}, 1 \leq
i,j \leq n$ in the determinant $J_F(x)$. Then there holds
\begin{equation*}
\sum_{i=1}^n \frac{\partial}{\partial x^j} B_{ij}(x)=0,
\quad i=1, \cdots, n.
\end{equation*}
In order to show \eqref{5.2}, it reduces to proving
\begin{equation}\label{eq:d_t_log det d phi}
\frac{\ud}{\ud t}\log \det(\ud \phi)=\frac{1}{\det(\ud \phi) G_0}
\frac{\partial}{\partial x_j}[\det(\ud \phi) G_0
\xi_j]=\mathrm{div}_{\phi^\ast (g_0)}(\xi)
\end{equation}
by virtue of \eqref{5.3}. To this end, denote by $G_0(x) \ud x_1 \wedge \cdots \wedge \ud x_n$ and
$J_{\phi}$ the volume form of $g_{\Sn}$ and the determinant of
$n$-th order matrix $(\frac{\partial \phi_i}{\partial x_j})$,
respectively, then it is easy to show that
$$\det (\ud\phi)=(G_0 \circ \phi) J_{\phi} G_0^{-1}.$$
Let $y=\phi(t)x$, a direct computation yields
\begin{equation}\label{in5.3}
\frac{\ud}{\ud t}\log \det(\ud \phi)=\frac{\ud}{\ud t}
\det(\ud \phi)\det(\ud \phi)^{-1}
\end{equation}
and
\begin{align}\label{in5.4}
\frac{\ud}{\ud t} \det(\ud \phi)=&\xi\cdot \ud(G_0 \circ \phi)J_\phi
G_0^{-1}+\det (\ud\phi)\sum\limits_{i,j=1}^{n}\frac{\partial x_j
}{ \partial y_i} \frac{\ud}{\ud t} \frac{\partial
\phi_i}{\partial x_j}\no\\
=&\xi\cdot \ud(G_0 \circ \phi)J_\phi G_0^{-1}+\det
(\ud\phi)\sum\limits_{i,j=1}^{n}\frac{\partial x_j}
{\partial y_i} \frac{\partial}{\partial x_j} \frac{\ud \phi_i }
{\ud t}\no\\
&+(G_0 \circ \phi) G_0^{-1}\sum\limits_{i,j=1}^n
\frac{\partial}{\partial x^j} \left(J_\phi \frac{\partial x_j}{\partial y_i}\right)\frac{\ud \phi_i}{\ud t},
\end{align}
where the third term on the right hand of \eqref{in5.4} vanishes
 by virtue of the Hadamard identity. Thus, \eqref{eq:d_t_log det d phi} follows from (\ref{in5.3}) and
(\ref{in5.4}). 

Hence, it follows from (\ref{2.2}) that
\begin{equation}\label{5.4}
\begin{split}
0&=\frac{2n}{n-2\sigma}\int_{\Sn}x v^{-1}v_t\ud V_h\\
&=\frac{n}{2}\int_{\Sn}x (\alpha f_\phi-R^h_\sigma)\ud V_h+\int_{\Sn}x\,\mathrm{div}_h(\xi)\ud V_h\\
&=\frac{n}{2}\int_{\Sn}x (\alpha f_\phi-R^h_\sigma)\ud V_h+\int_{\Sn}\xi \ud V_h.
\end{split}
\end{equation}

\subsection{Scaled stereographic projection}

Let $\pi:\Sn\setminus\{S\}\to\mathbb{R}^n$ be the stereographic projection from the south pole $S$, i.e.
$$\pi(x)=\frac{(x_1,\cdots,x_n)}{1+x_{n+1}},\hspace{2mm} x\in \Sn.$$
Then its inverse $\Psi=\pi^{-1}$ is given by
$$\Psi(z)=\frac{(2z_1,\cdots,2z_n,1-|z|^2)}{1+|z|^2},\hspace{2mm} z\in\mathbb{R}^n.$$
For $q\in\mathbb{R}^n$ and $r>0$, let $\psi_{q,r}:\mathbb{R}^n\to \Sn\setminus\{S\}$ be the conformal map given by
$$\psi_{q,r}=\Psi\circ\delta_{q,r},$$
where $\delta_{q,r}=q+rz$ for $z\in\mathbb{R}^n$. A direct computation yields that, for $1\leq i, j\leq n$,
\begin{equation*}
\frac{\partial\Psi_i}{\partial z_j}=\frac{2\delta_{ij}}{1+|z|^2}-\frac{4z_iz_j}{(1+|z|^2)^2}=\delta_{ij}(1+x_{n+1})-x_ix_j,
\end{equation*}
and
\begin{equation*}
\frac{\partial\Psi_{n+1}}{\partial z_j}=-\frac{4z_j}{(1+|z|^2)^2}=-(1+x_{n+1})x_j.
\end{equation*}
Then we have
\begin{equation*}
\sum_{i=1}^nz_i\frac{\partial\Psi_k}{\partial z_i}=x_kx_{n+1}\quad \mathrm{ for~~}1\leq k\leq n
\end{equation*}
and
\begin{equation*}
\sum_{i=1}^nz_i\frac{\partial\Psi_{n+1}}{\partial z_i}=|x_{n+1}|^2-1.
\end{equation*}
This implies that
\begin{equation}\label{est:bases_bdd}
 \left\|\sum_{i=1}^nz_i\frac{\partial\Psi}{\partial z_i}\right\|_{L^\infty(\R^n)}\leq 1
\quad \mathrm{and~~} \left\|\frac{\partial\Psi}{\partial z_i}\right\|_{L^\infty(\R^n)}\leq 2 \quad \mathrm{for~~} i=1,2,\cdots,n.
\end{equation}

At time $t$, we always can express  $\phi(t)=\Psi\circ\delta_{q(t),r(t)}\circ\pi$ on $\Sn$, and its differential as
$\ud\phi=\ud\Psi\circ r(t)I\circ \ud\pi$, so $(\ud\phi)^{-1}=r^{-1}(\ud\pi)^{-1}\circ (\ud\Psi)^{-1}$.
Also notice that $(\ud\pi)^{-1}=\ud\Psi$. Let $\epsilon=r^{-1}$ as before. As in \cite[(5.5)]{Chen&Xu}, we compute the conformal vector field
\begin{equation}\label{5.9}
\xi=(\ud\phi)^{-1}\frac{\ud\phi(t)}{\ud t}
=\epsilon\sum_{i=1}^n\left(\frac{\ud q_i}{\ud t}+z_i\frac{\ud r}{\ud t}\right)\frac{\partial\Psi}{\partial z_i}.
\end{equation}
For simplicity of computations, we can assume $q(t)=0$ for time $t$ by a suitable selection of the coordinates on $\Sn$. Again by \cite[(5.6), (5.7)]{Chen&Xu}, we obtain
\begin{equation}\label{5.9a}
X=\fint_{\Sn}\xi \ud V_{g_{\Sn}}= \frac{n}{n+1}\epsilon\left(\frac{\ud q}{\ud t},-\frac{\ud r}{dt}\right)^\top.
\end{equation}

\begin{lemma}\label{lemma5.2}
There exists a constant $C>0$ independent of $t$ such that
$$\|\xi\|^2_{L^\infty(\Sn)}\leq C\fint_{\Sn}(\alpha f_\phi-R^h_\sigma)^2\ud V_h
\quad \mathrm{ and }\quad
\|\mathrm{div}_{\Sn}\xi\|^2_{L^\infty(\Sn)}
\leq C\fint_{\Sn}(\alpha f_\phi-R^h_\sigma)^2\ud V_h.
$$
\end{lemma}
\begin{proof}

By (\ref{5.9}), \eqref{est:bases_bdd} and \eqref{5.9a} we have
$$\|\xi\|\leq 2\epsilon\left(\left|\frac{\ud q}{\ud t}\right|+\left|\frac{\ud r}{\ud t}\right|\right)=\frac{2(n+1)}{n}\|X\|.$$
On the other hand, it follows from  (\ref{5.4}) and definition of $X$ that
\begin{align*}
\|X\|\leq& 
\left|\frac{n}{2}\fint_{\Sn}x (\alpha f_\phi-R^h_\sigma)\ud V_h\right|
+\left|\fint_{\Sn}\xi(1-v^{\frac{2n}{n-2\sigma}}) \ud V_{g_{\Sn}}\right|\\
\leq&
\left|\frac{n}{2}\fint_{\Sn}x (\alpha f_\phi-R^h_\sigma)\ud V_h\right|
+\|\xi\|_{L^\infty(\Sn)}\|1-v^{\frac{2n}{n-2\sigma}}\|_{C(\Sn)}.
\end{align*}
It follows from Lemma \ref{lem:lim_seq_v} that $\|1-v^{2n/(n-2\sigma)}\|_{C(\Sn)}\to 0$ as $t\to\infty$.
and $\|\xi\|$ is continuous on $\Sn\times [0,T+1]$  for any finite $T>0$. Thus, the first assertion follows from these two estimates.

For the second assertion, by (\ref{5.9}) we have 
\begin{align*}
\mathrm{div}_{\Sn}(\xi)=&\sum_{i,j=1}^n \frac{\partial \Psi_j }{\partial x_i} \frac{\partial \xi_i}{\partial z_j }+\sum_{i=1}^n  \frac{\partial \Psi_i }{\partial x_{n+1}} \frac{\partial \xi_{n+1}}{\partial z_i}\\
=&\epsilon\left[\frac{\ud r}{\ud t}\left(n-\frac{2(n+1)|z|^2}{1+|z|^2}\right)-2(n+1)\sum_{i=1}^n\frac{\ud q_i}{\ud t}\frac{z_i}{1+|z|^2}\right].
\end{align*}
Thus, we obtain
\begin{align*}
\|\mathrm{div}_{\Sn}\xi\|_{L^\infty(\Sn)}&\leq (n+2)\epsilon\left(\left|\frac{\ud q}{\ud t}\right|+\left|\frac{\ud r}{\ud t}\right|\right)=\frac{(n+1)(n+2)}{n}\|X\|_{L^\infty(\Sn)}\\
&\leq \frac{(n+1)(n+2)}{n}\|\xi\|_{L^\infty(\Sn)}.
\end{align*}
This together with the first assertion implies the second one. 
\end{proof}

\subsection{Shadow flow}\label{Subsect:SF}
When $0<\sigma<1/2$,  the loss of continuous embedding of $H^{2\sigma,p}(\Sn)$ into $C^1(\Sn)$ for $v(t)$ will bring some technical difficultes in Section \ref{Subsect:SF}, so we assume henceforth $1/2<\sigma<1$.

We define a shadow flow by
\begin{equation}\label{shadow_flow}
\Theta=\Theta(t)=\fint_{\Sn}\phi(t) \ud V_{g_{\Sn}}
\end{equation}
with $\theta(t)=\Theta(t)/|\Theta(t)|$ whenever $\Theta(t)\neq 0$, which approximates the center $p(t)$ of mass of $g(t)$ by virtue of Lemma \ref{lem:lim_seq_v}. 

For convenience, we recall an elementary result in \cite[Lemma 6.2]{Chen&Xu} here.

\begin{lemma}\label{lemma5.1}
With a uniform consitant $C>0$, there holds
$$\|f_\phi-f(\theta(t))\|_{L^2(\Sn)}
+\|\nabla f_\phi\|_{L^{\frac{2n}{n+2}}(\Sn)}\leq C\epsilon,$$
where $\epsilon=r(t)^{-1}$.
\end{lemma}

Let $\{\varphi_i; i\in\mathbb{N}\}_{}$ be an $L^2(\Sn)$-orthonormal basis of eigenfunctions of
$-\Delta_{g_{\Sn}}$, satisfying
$$-\Delta_{g_{\Sn}}\varphi_i=\lambda_i\varphi_i,$$
with the eigenvalues $0=\lambda_0<\lambda_1=\cdots=\lambda_{n+1}<\lambda_{n+2}\leq\cdots$. Now
in terms of the orthonormal bases $\{\varphi_i^g\}$, $\{\varphi_i^h\}$
of the eigenfunctions of $-\Delta_g$, $-\Delta_h$ respectively, we expand
\begin{equation*}
\alpha(t) f-R^g_\sigma=\sum_{i=0}^\infty\beta^i_g\varphi_i^g\quad\mathrm{ and }\quad
\alpha(t) f_\phi-R^h_\sigma=\sum_{i=0}^\infty\beta^i_h\varphi_i^h
\end{equation*}
with coefficients
\begin{equation*}
\beta_h^i=\fint_{\Sn}(\alpha(t) f_\phi-R^h_\sigma)\varphi_i^h\ud V_h=\fint_{\Sn}(\alpha(t) f-R^g_\sigma)\varphi_i^g\ud V_g=\beta^i_g
\end{equation*}
for all $i\in\mathbb{N}$.

First notice that
\begin{equation}\label{5.13}
\beta_g^0=0
\end{equation}
by virtue of (\ref{1.4}) and 
$\varphi_i^h=\varphi_i^g\circ\phi$.

\begin{lemma}\label{lemma5.3}
As $t\to\infty$, we have $\lambda_i^g=\lambda_i^h\to\lambda_i$ and we can choose
$\varphi_i$ such that $\varphi_i^h\to\varphi_i$ in $L^2(\Sn)$ for all $i \in \mathbb{N}$.
\end{lemma}

The proof of Lemma \ref{lemma5.3} is standard, for instance, \cite[Appendix B]{Chen&Xu}; see also \cite[Lemma 4.2 and Remark 4.3]{Schwetlick&Struwe}.

The eigenvalues of $P_\sigma$ are denoted by
$$R_\sigma=\Lambda_0<\Lambda_1\leq\Lambda_2\leq\cdots\leq \Lambda_{n+1}<\Lambda_{n+2}\leq \cdots$$
counting with multiplicity.
It is well-known (see \cite[p. 479]{Morpurgo}) that $P_\sigma$ has spherical harmonics as eigenfunctions with the corresponding eigenvalues
\begin{equation*}
\frac{\Gamma(k+\frac{n}{2}+\sigma)}{\Gamma(k+\frac{n}{2}-\sigma)}, \quad k\in \mathbb{N},
\end{equation*}
and multiplicity $(2k+n-1)(k+n-2)!/[(n-1)!k!]$. In particular,
$x_1, x_2,\cdots, x_{n+1}$ are the eigenfunctions of eigenvalues
$$\Lambda_1=\Lambda_2=\cdots=\Lambda_{n+1}=\frac{\Gamma(1+\frac{n}{2}+\sigma)}{\Gamma(1+\frac{n}{2}-\sigma)}=\frac{n+2\sigma}{n-2\sigma}R_\sigma,$$ 
that is,
\begin{equation}\label{eigenfunction}
P_\sigma(x)=\Lambda_1 x, \quad\mathrm{ for ~~}x\in \Sn.
\end{equation}
Let $\{\Lambda_i^g; i\in\mathbb{N}\}$
and $\{\Lambda_i^h; i\in\mathbb{N}\}$ be the sets of eigenvalues of $P_\sigma^g$ and $P_\sigma^h$, respectively.
Then it follows from Lemmas \ref{lem:lim_seq_v} and \ref{lemma5.3} that for every $i \in \mathbb{N}$,
\begin{equation}\label{5.21}
\Lambda_i^g=\Lambda_i^h\to \Lambda_i \quad\mathrm{~~as~~}t\to\infty.
\end{equation}

\begin{lemma}\label{lemma6.6}
With a uniform constant $C>0$, there holds
$$\|P_\sigma(v(t)-1)\|_{L^2(\Sn)}\leq C(F_2(t)^{\frac{1}{2}}+\|f_\phi-f(\theta(t))\|_{L^2(\Sn)})$$
for all sufficiently large $t>0$.
\end{lemma}
\begin{proof}
Expand $v^{2n/(n-2\sigma)}-1$ and $v-1$ in terms of eigenfunctions of $P_\sigma$ to get
\begin{equation*}
v^{\frac{2n}{n-2\sigma}}-1=\sum_{i=0}^\infty V^i\varphi_i
~~\mathrm{ and }~~v-1=\sum_{i=0}^\infty v^i\varphi_i.
\end{equation*}
By (\ref{1.9}), we have
$$
\int_{\Sn}(v^{\frac{2n}{n-2\sigma}}-1)\ud V_{g_{\Sn}}=\int_{\Sn}(u^{\frac{2n}{n-2\sigma}}-1)\ud V_{g_{\Sn}}=0,
$$
which implies that $V^0=0$. From this and Taylor's expansion, we have
$$0=\fint_{\Sn}(v^{\frac{2n}{n-2\sigma}}-1)\ud V_{g_{\Sn}}=\frac{2n}{n-2\sigma}\fint_{\Sn}(v-1)\ud V_{g_{\Sn}}+O(1)\fint_{\Sn}(v-1)^2\ud V_{g_{\Sn}},$$
which implies that
\begin{equation*}
v^0=\fint_{\Sn}(v-1)\ud V_{g_{\Sn}}=O(1)\|v-1\|^2_{H^\sigma(\Sn)}.
\end{equation*}
On the other hand, it follows from (\ref{2.2}) that
$V^i=0$ for $1\leq i\leq n+1$.
By Taylor's expansion, we obtain
\begin{align*}
\frac{2n}{n-2\sigma}v^i&=\frac{2n}{n-2\sigma}
\fint_{\Sn}(v-1)\varphi_i\ud V_{g_{\Sn}}\\
&=\fint_{\Sn}(v^{\frac{2n}{n-2\sigma}}-1)\varphi_i\ud V_{g_{\Sn}}
+O(\|v-1\|^2_{H^\sigma(\Sn)})\\
&=V^i+o(1)\|v-1\|_{H^\sigma(\Sn)} \mathrm{~~for~~} 1 \leq i \leq n+1.
\end{align*}
Thus, we have
\begin{equation}\label{6.28}
\sum_{i=0}^{n+1}|v^i|^2=o(1)\|v-1\|_{H^\sigma(\Sn)}^2.
\end{equation}
We may decompose
\begin{align}\label{6.29}
P_\sigma(v-1)
=&P_\sigma(v)-P_\sigma(1)=R_\sigma^h v^{\frac{n+2\sigma}{n-2\sigma}}-R_\sigma\no\\
=&\left[(R_\sigma^h-\alpha f_\phi)+(\alpha f_\phi-\alpha f(\theta))
+\left(\alpha f(\theta)-\fint_{\Sn} R_\sigma^h \ud V_h\right)\right.\no\\
&~~\left. +\left(\fint_{\Sn} R_\sigma^h \ud V_h-R_\sigma\right)\right]v^{\frac{n+2\sigma}{n-2\sigma}}+R_\sigma(v^{\frac{n+2\sigma}{n-2\sigma}}-1).
\end{align}
By Lemma \ref{lem:lim_seq_v}, we have
\begin{equation}\label{6.30}
\left|\alpha f(\theta)-\fint_{\Sn} R_\sigma^h \ud V_h\right|
=\left|\fint_{\Sn}\alpha(f(\theta)-f_\phi) \ud V_h\right|\leq C\|f(\theta)-f_\phi\|_{L^2(\Sn)}.
\end{equation}
Since
\begin{equation*}
\fint_{\Sn} R_\sigma^h \ud V_h-R_\sigma
=\fint_{\Sn} (v-1)P_\sigma(v-1) \ud V_{g_{\Sn}}
+2R_\sigma\fint_{\Sn} (v-1) \ud V_{g_{\Sn}},
\end{equation*}
by \eqref{6.28} we have
\begin{equation}\label{6.31}
\fint_{\Sn} R_\sigma^h \ud V_h-R_\sigma
=\fint_{\Sn} (v-1)P_\sigma(v-1) \ud V_{g_{\Sn}}
+o(1) \|v-1\|_{H^\sigma(\Sn)}.
\end{equation}
Also, by Taylor's expansion, we have
\begin{equation}\label{6.32}
v^{\frac{n+2\sigma}{n-2\sigma}}-1=
\frac{n+2\sigma}{n-2\sigma}(v-1)+O(|v-1|^2).
\end{equation}

On one hand, by (\ref{6.28}) we have
$$\fint_{\Sn}|P_\sigma(v-1)|^2\ud V_{g_{\Sn}}=\sum_{i=0}^\infty \Lambda_i^2|v^i|^2
=\sum_{i=n+2}^\infty \Lambda_i^2|v^i|^2+o(1)\|v-1\|^2_{H^\sigma(\Sn)}.$$ 
On the other hand, it follows from (\ref{6.29})-(\ref{6.32}) and Young's inequality that
for any $\delta>0$, there exists a constant $C(\delta)>0$ such that
\begin{align}\label{6.33}
&\fint_{\Sn}|P_\sigma(v-1)|^2\ud V_{g_{\Sn}}\no\\
\leq&C(\delta)(F_2+\|f(\theta)-f_\phi\|^2_{L^2(\Sn)})
+(1+\delta)\left[\fint_{\Sn} (v-1)P_\sigma(v-1) \ud V_{g_{\Sn}}\right]^2\no\\
&+(1+\delta)\left(\frac{n+2\sigma}{n-2\sigma}\right)^2R_\sigma^2\|v-1\|^2_{L^2(\Sn)}+o(1)\|v-1\|^2_{L^2(\Sn)}\no\\
\leq& C(\delta)(F_2+\|f(\theta)-f_\phi\|^2_{L^2(\Sn)})
+(1+\delta)\|v-1\|^2_{L^2(\Sn)}
\left[\fint_{\Sn} |P_\sigma(v-1)|^2 \ud V_{g_{\Sn}}\right]\no\\
&+(1+\delta)\left(\frac{n+2\sigma}{n-2\sigma}\right)^2R_\sigma^2\sum_{i=n+2}^\infty |v^i|^2+o(1)\|v-1\|^2_{H^\sigma(\Sn)}.
\end{align}
Notice that
\begin{equation}\label{eq:Lambda_n+2}
\Lambda_{n+2}=\frac{\Gamma(2+\frac{n}{2}+\sigma)}{\Gamma(2+\frac{n}{2}-\sigma)}=
\frac{2+n+2\sigma}{2+n-2\sigma}
\frac{n+2\sigma}{n-2\sigma}
\frac{\Gamma(\frac{n}{2}+\sigma)}{\Gamma(\frac{n}{2}-\sigma)}>\frac{n+2\sigma}{n-2\sigma}R_\sigma
\end{equation}
and $\|v-1\|_{L^2(\Sn)}\to 0$ as $t\to\infty$.
Hence by \eqref{6.33}, we can choose sufficiently large $t_0>0$ and sufficiently small $\delta>0$ such that
\begin{equation}\label{6.35}
\int_{\Sn}|P_\sigma(v-1)|^2\ud V_{g_{\Sn}}
\leq C(\delta)(F_2+\|f(\theta)-f_\phi\|^2_{L^2(\Sn)})
\end{equation}
for all $t\geq t_0$.
\end{proof}

Now we define
$$b=\fint_{\Sn}x(\alpha f_\phi-R_\sigma^h)\ud V_h.$$
For brevity, set $B=b/\sqrt{n+1}$ and $\beta_g=(\beta^1_g,\cdots,\beta^{n+1}_g)$, then
\begin{equation}\label{5.20}
|b^i-\sqrt{n+1}\beta^i_g|\leq \|\varphi_i^h-\varphi_i\|_{L^2(\Sn)}F_2(h(t))^{\frac{1}{2}}=o(1)F_2(t)^{\frac{1}{2}}
\end{equation}
for $i=1,2,\cdots,n+1$ by Lemma \ref{lemma5.3}, where $o(1)\to 0$ as $t\to\infty$.

\begin{lemma}\label{lemma5.4} 
Assume $\sigma\in (1/2,1)$, then with error $o(1)\to 0$ as $t\to\infty$, there holds
$$\frac{\ud B}{\ud t}=o(1)F_2(t)^{\frac{1}{2}}.$$
\end{lemma}
\begin{proof}
Up to a constant, we may argue for $b$ instead of $B$. Observe that
\begin{equation*}
\fint_{\Sn}xR_\sigma^h\ud V_h=\fint_{\Sn}xvP_\sigma(v)\ud V_{g_{\Sn}}\quad\mathrm{and}\quad
b=\fint_{\Sn}x\alpha f_\phi \ud V_h-\fint_{\Sn}xvP_\sigma(v)\ud V_{g_{\Sn}}.
\end{equation*}
Differentiating the second equation with respect to $t$, we obtain
\begin{align}\label{5.15}
\frac{\ud b}{\ud t}&=\alpha'\fint_{\Sn}xf_\phi \ud V_h+\fint_{\Sn}x\alpha (\ud f_\phi\cdot\xi) \ud V_h
+\frac{2n}{n-2\sigma}\fint_{\Sn}x\alpha f_\phi v^{\frac{n+2\sigma}{n-2\sigma}}v_t\ud V_{g_{\Sn}}\no\\
&\hspace{4mm}-\fint_{\Sn}xv_tP_\sigma(v)\ud V_{g_{\Sn}}-\fint_{\Sn}xvP_\sigma(v_t)\ud V_{g_{\Sn}}.
\end{align}
Since $P_\sigma$ is self-adjoint, we have
$$\fint_{\Sn}xvP_\sigma(v_t)\ud V_{g_{\Sn}}=\fint_{\Sn}P_\sigma(xv)v_t\ud V_{g_{\Sn}}.$$
Hence, we can rewrite (\ref{5.15})  as
\begin{align}\label{5.16}
\frac{\ud b}{\ud t}=&\alpha'\fint_{\Sn}xf_\phi \ud V_h+\left[\fint_{\Sn}x\alpha (\ud f_\phi\cdot\xi) \ud V_h
+\frac{2n}{n-2\sigma}\fint_{\Sn}x(\alpha f_\phi v^{\frac{n+2\sigma}{n-2\sigma}}-R_\sigma v)v_t\ud V_{g_{\Sn}}\right]\no\\
&-\fint_{\Sn}xv_t\left(P_\sigma(v)-R_\sigma v\right)\ud V_{g_{\Sn}}-\fint_{\Sn}\left(P_\sigma (xv)-\frac{n+2\sigma}{n-2\sigma}R_\sigma x v\right) v_t \ud V_{g_{\Sn}}\no\\
:=&I_5+I_6+I_7+I_8.
\end{align}

We are going to estimate each of the terms $I_i, 5 \leq i \leq 8$. To this end, by (\ref{5.2}) and \eqref{1.7}  we have 
\begin{equation}\label{eq:v_t}
v_t=\frac{n-2\sigma}{4}(\alpha f_\phi-R_\sigma^h)v+\frac{n-2\sigma}{2n}v \mathrm{div}_h(\xi)
\end{equation}
and then
\begin{align}\label{est_v_t}
\|v_t\|_{L^2(\Sn)}\leq& 
C\|v\|_{L^\infty(\Sn)}\|\alpha f_\phi-R^h_\sigma\|_{L^2(\Sn)}
+C\|\nabla v\|_{L^2(\Sn)}\|\xi\|_{L^\infty(\Sn)}\no\\
&+C\|v\|_{L^2(\Sn)}\|\mathrm{div}_{\Sn}(\xi)\|_{L^\infty(\Sn)}\no\\
\leq& CF_2(t)^{\frac{1}{2}}
\end{align}
by  Lemmas \ref{lemma5.2} and \ref{lem:lim_seq_v}.

By (\ref{2.4}) and Lemma \ref{lem:lim_seq_v}, we have
\begin{equation*}
I_5=\alpha'\fint_{\Sn}xf_\phi \ud V_h=\alpha'\fint_{\Sn}x(f_\phi-f(p(t)))\ud V_h=o(1)F_2(t)^{\frac{1}{2}}.
\end{equation*}

By \eqref{eq:v_t}, Lemmas \ref{lem:lim_seq_v} and \ref{lemma5.2}, we have
\begin{align*}
I_6=&\fint_{\Sn}x\alpha (\ud f_\phi\cdot\xi) \ud V_h+\fint_{\Sn}x\alpha f_\phi\mathrm{div}_{\Sn}(v^{\frac{2n}{n-2\sigma}}\xi)\ud V_{g_{\Sn}}\\
&+\frac{n}{2}\fint_{\Sn}x(\alpha f_\phi-R_\sigma^h)(\alpha f_\phi v^{\frac{2n}{n-2\sigma}}-R_\sigma v^2)\ud V_{g_{\Sn}}\\
&
-R_\sigma\fint_{\Sn}x v^{2-\frac{2n}{n-2\sigma}}
\mathrm{div}_{\Sn}(v^{\frac{2n}{n-2\sigma}}\xi)\ud V_{g_{\Sn}}\\
=&\frac{n}{2}\fint_{\Sn}x(\alpha f_\phi-R_\sigma^h)(\alpha f_\phi-R_\sigma v^{2-\frac{2n}{n-2\sigma}})\ud V_{h}\\
&
-\fint_{\Sn}\xi(\alpha f_\phi-R_\sigma v^{2-\frac{2n}{n-2\sigma}})\ud V_{h}
+\left(2-\frac{2n}{n-2\sigma}\right)R_\sigma\fint_{\Sn}v(\ud v\cdot\xi)\ud V_{g_{\Sn}}\\
\leq& C\|\alpha f_\phi  -R_\sigma v^{2-\frac{2n}{n-2\sigma}}\|_{L^2(\Sn,h)}
(\|\alpha f_\phi-R_\sigma^h\|_{L^2(\Sn,h)}+\|\xi\|_{L^\infty(\Sn)})\\
&+C\|\xi\|_{L^\infty(\Sn)}\|\nabla v\|_{L^2(\Sn)}\\
=&o(1)F_2(t)^{\frac{1}{2}}.
\end{align*}

Using Lemmas \ref{lem:lim_seq_v}, \ref{lemma6.6} and \eqref{est_v_t}, we have
\begin{align*}
I_7=&-\fint_{\Sn}xv_t[P_\sigma(v-1)+R_\sigma(1- v)]\ud V_{g_{\Sn}}\\
\leq& C(\|1-v\|_{L^2(\Sn)}+\|P_\sigma(v-1)\|_{L^2(\Sn)})\|v_t\|_{L^2(\Sn)}\\
=&o(1)F_2(t)^{\frac{1}{2}}.
\end{align*}

It remains to estimate the last term $I_8$. Notice that
\begin{align*}
y(v(y)-1)-z(v(z)-1)=(y-z)(v(z)-v(y))+y(v(y)-v(z))+(y-z)(v(y)-1),
\end{align*}
then it follows that
\begin{align*}
P_\sigma(x(v(x)-1))(y)=& y P_\sigma(v-1)(y) + C_{n,-\sigma} \int_{\Sn} \frac{(y-z) (v(z)-v(y))}{|y-z|^{n+2\sigma}} \,\ud V_{g_{\Sn}}(z)\\
&+(v(y)-1)(P_\sigma(y)-R_\sigma y)\\
=&y P_\sigma(v-1)(y) + C_{n,-\sigma} \int_{\Sn} \frac{(y-z) (v(z)-v(y))}{|y-z|^{n+2\sigma}} \,\ud V_{g_{\Sn}}(z)\\
&+(v(y)-1)(\Lambda_1-R_\sigma) y.
\end{align*}
Thus, by \eqref{eigenfunction}, H\"older inequality we have
\begin{align*}
&\|P_\sigma (xv)-P_\sigma(x)v\|_{L^2(\Sn)} \\
\le& \|P_\sigma (x(v-1))\|_{L^2(\Sn)} + \|P_\sigma(x) (v-1)\|_{L^2(\Sn)} \\
\le& \|P_\sigma (v-1)\|_{L^2(\Sn)} +  C_{n,-\sigma}\left\| \int_{\Sn} \frac{(y-z) (v(z)-v(y))}{|y-z|^{n+2\sigma}}\, \ud V_{g_{\Sn}}(z)\right\|_{L^2(\Sn)}+C\|v-1\|_{L^2(\Sn)} \\
\le& C\|P_\sigma (v-1)\|_{L^2(\Sn)}.
\end{align*}
Using the above inequality, \eqref{est_v_t} and Lemma \ref{lemma6.6}, we have
\begin{equation*}
I_8=-\fint_{\Sn}\left(P_\sigma (xv)-P_\sigma(x)v\right) v_t \ud V_{g_{\Sn}}
\leq  C\|P_\sigma (v-1)\|_{L^2(\Sn)} \|v_t\|_{L^2(\Sn)}=o(1)F_2(t)^{\frac{1}{2}}.
\end{equation*} 

Therefore, inserting all these estimates into (\ref{5.16}), we obtain the desired estimate.
\end{proof}

\begin{lemma}\label{lemma5.5} Let $\sigma\in(1/2,1)$,
then with error $o(1)\to 0$ as $t\to\infty$, there holds
$F_2(t)= (1+o(1))|B(t)|^2$.
\end{lemma}
\begin{proof}
Let $\hat{F}_2=\sum_{i=n+2}^\infty|\beta^i_g|^2$ for brevity, then with error $o(1)\to 0$ as $t\to\infty$ and (\ref{5.20}), there holds
$$F_2=|\beta_g|^2+\hat{F}_2=
|B|^2+\hat{F}_2+
o(1)F_2^{\frac{1}{2}}.$$
Then by (\ref{5.13}), (\ref{5.21})  and \eqref{eq:Lambda_n+2},
we have
\begin{align}\label{5.23}
-\frac{n-2\sigma}{2}G_2+\frac{n+2\sigma}{2}R_\sigma F_2=&-\frac{n-2\sigma}{2}\sum_{i=0}^\infty \Lambda_i^g |\beta_g^i|^2+\frac{n+2\sigma}{2}R_\sigma\sum_{i=0}^\infty |\beta_g^i|^2\no\\
\leq&-\frac{2\sigma(n+2\sigma)}{2+n-2\sigma}R_\sigma\sum_{i=n+2}^\infty|\beta_g^i|^2+o(1)F_2.
\end{align}
Combining \eqref{5.23} and Lemma \ref{lemma3.7}, we obtain
\begin{align}\label{5.24}
&\frac{\ud}{\ud t}F_2\leq -\frac{2\sigma(n+2\sigma)}{2+n-2\sigma}R_\sigma\hat F_2+o(1)F_2.
\end{align}

We first assume that $|B(t_1)|^2\geq \hat{F}_2(t_1)$
for some large $t_1\geq 0$. For $t$ near $t_1$, we may then express
$F_2$ as
$$F_2(t)=(1+\delta(t))|B(t)|^2$$
with $-1/2<\delta(t)\leq 1$ if $t_1$ is sufficiently large.
Inserting the above expression into (\ref{5.24}), we obtain
\begin{equation}\label{5.25}
\frac{\ud}{\ud t}F_2=
\frac{\ud\delta}{\ud t}|B|^2+2(1+\delta)B\cdot\frac{\ud B}{\ud t}
\leq-\left[\frac{2\sigma(n+2\sigma)}{2+n-2\sigma}R_\sigma\frac{\delta}{1+\delta}+o(1)\right]
F_2.
\end{equation}
By Lemma \ref{lemma5.4} we have
$$\left|\frac{\ud B}{\ud t}\right|=o(1)F_2(t)^{\frac{1}{2}}.$$
Therefore, we can conclude from (\ref{5.25}) that
\begin{equation*}
\begin{split}
\frac{\ud\delta}{\ud t}|B|^2\leq
&-\left[\frac{2\sigma(n+2\sigma)}{2+n-2\sigma}R_\sigma\frac{\delta}{1+\delta}+o(1)\right]
F_2\\
=&-\left[\frac{2\sigma(n+2\sigma)}{2+n-2\sigma}R_\sigma\delta+o(1)\right]
|B|^2.
\end{split}
\end{equation*}
Canceling the (nonvanishing) factor $|B|^2$, we have
\begin{equation*}
\frac{\ud\delta}{\ud t}
\leq-\frac{2\sigma(n+2\sigma)}{2+n-2\sigma}R_\sigma \delta+o(1).
\end{equation*}
This implies that
$\delta(t)\to 0$ as $t\to\infty$; whence,
$F_2=(1+o(1))|B|^2.$

It remains to show  $|B(t_1)|^2\geq \hat{F}_2(t_1)$
for some large $t_1$. If not, we assume
that $|B(t)|^2< \hat{F}_2(t)$
for all sufficiently large $t$. Thus,
for sufficiently small $\epsilon>0$,
it follows from (\ref{5.24})
that
$$\frac{\ud}{\ud t}F_2\leq -\left(\frac{\sigma(n+2\sigma)}{2+n-2\sigma}R_\sigma+o(1)\right)F_2$$
for all sufficiently large $t$.
This implies that
$$F_2(t)\leq Ce^{-\frac{\sigma(n+2\sigma)}{2(2+n-2\sigma)}t}$$
for all $t\geq t_0$ and $C=C(t_0)>0$.
However, given any $r_0>0$ and any $Q\in \Sn$, we have
$$\left|\frac{\ud}{\ud t}\mathrm{Vol}(B_{r_0}(Q),g)\right|
=\left|\frac{n}{2}\int_{B_{r_0}(Q)}(\alpha f-R_\sigma^g)\ud V_g\right|
\leq CF_2(t)^{\frac{1}{2}}\leq Ce^{-\frac{\sigma(n+2\sigma)}{4(2+n-2\sigma)}t},$$
where $B_{r_0}(Q)=B_{r_0}(Q,g_{\Sn})$.
This implies that
\begin{equation}\label{5.26}
\mathrm{Vol}(B_{r_0}(Q),g(t))
\leq \mathrm{Vol}(B_{r_0}(Q),g(t_0))
+Ce^{-\frac{t_1}{2C}}<\frac{1}{2}\omega_n
\end{equation}
uniformly for $t\geq t_1>t_0$ and sufficiently small $r_0$ .
On the other hand, by Lemma \ref{lem:lim_seq_v} we infer that for the concentration point $Q$ and any $r_0>0$,
$$\mathrm{Vol}(B_{r_0}(Q),g(t))\to \omega_n$$
as $t\to\infty$, which contradicts (\ref{5.26}).
This completes the proof of Lemma \ref{lemma5.5}.
\end{proof}

\begin{lemma}[cf. {\protect  \cite[Lemma 6.7]{Chen&Xu}}]\label{lemma6.7}
For $\sigma\in(1/2,1)$ and all $t>0$, with $O(1) \to 0$ as $t \to \infty$ there hold
$$b-\langle b,\theta\rangle \theta=\epsilon\left(\frac{4}{n}\alpha\, \ud f(\theta)+O(\epsilon)\right)$$
and
$$\langle b,\theta\rangle=\epsilon^2\left(-\frac{8}{n(n-2)}\alpha \Delta_{g_{\Sn}}f(\theta)+O(1)|\nabla f(\theta)|_{g_{\Sn}}^2+O(\epsilon|\log\epsilon|)\right).$$
\end{lemma}
\begin{proof}
Since $-\Delta_{g_{\Sn}}x=nx$, we have
$$nb=\fint_{\Sn}nx(\alpha f_\phi-R_\sigma^h) \ud V_h=\fint_{\Sn}-\Delta_{g_{\Sn}}x(\alpha f_\phi-R_\sigma^h) \ud V_h.$$
By Kazdan-Warner condition (\ref{5.1}) and an integrating by parts, we have
\begin{equation*}
nb=\alpha\fint_{\Sn}\langle\nabla x,\nabla f_\phi\rangle_{g_{\Sn}} \ud V_h+E_1
\end{equation*}
with the error term $E_1$ given by
\begin{equation*}
E_1=\frac{2n}{n-2\sigma}\fint_{\Sn}v^{\frac{n+2\sigma}{n-2\sigma}}\langle\nabla x,\nabla v\rangle_{g_{\Sn}}(\alpha f_\phi-R_\sigma^h) \ud V_{g_{\Sn}}.
\end{equation*}
This can be estimated by
\begin{equation}\label{6.34}
|E_1|\leq C\|\nabla v\|_{L^2(\Sn)}\|\alpha f_\phi-R_\sigma^h\|_{L^2(\Sn,h)}
\leq C\|v-1\|_{H^{2\sigma}(\Sn)}F_2^{\frac{1}{2}}.
\end{equation}

An integration by parts gives
\begin{align*}
\alpha\fint_{\Sn}\langle\nabla x,\nabla f_\phi\rangle_{g_{\Sn}} \ud V_h=&n\alpha\fint_{\Sn}\langle\nabla x,\nabla f_\phi\rangle_{g_{\Sn}} \ud V_{g_{\Sn}}+E_2\\
=&n\alpha\fint_{\Sn} x (f_\phi-f(\theta)) \ud V_{g_{\Sn}}+E_2
\end{align*}
with the error term $E_2$ given by
\begin{align*}
E_2=&\alpha\fint_{\Sn}\langle\nabla x,\nabla f_\phi\rangle_{g_{\Sn}}(v^{\frac{2n}{n-2\sigma}}-1)\ud V_{g_{\Sn}}\\
=&n\alpha\fint_{\Sn} x(f_\phi-f(\theta))(v^{\frac{2n}{n-2\sigma}}-1)\ud V_{g_{\Sn}}
\\
&-\frac{2n}{n-2\sigma}\alpha\fint_{\Sn}
\langle\nabla x,\nabla v\rangle_{g_{\Sn}}
(f_\phi-f(\theta))
v^{\frac{n+2\sigma}{n-2\sigma}}\ud V_{g_{\Sn}},
\end{align*}
which can be estimated by
\begin{align}\label{6.36}
|E_2|&\leq C(\|v-1\|_{L^2(\Sn)}+\|\nabla v\|_{L^2(\Sn)})\|f_\phi-f(\theta)\|_{L^2(\Sn)}\no\\
&\leq C\|v-1\|_{H^{2\sigma}(\Sn)}\|f_\phi-f(\theta)\|_{L^2(\Sn)}.
\end{align}
With the help of Lemmas \ref{lemma5.5} and \ref{lemma6.6}, the estimates of both tangent and normal (i.e.  projection on the direction $\theta$) parts of
$$\fint_{\Sn} x (f_\phi-f(\theta)) \ud V_{g_{\Sn}}$$
have been settled in the proof of \cite[Lemma 6.7]{Chen&Xu} with minor modifications. This together with  estimates (\ref{6.34}) and (\ref{6.36}) directly implies the desired assertion.
\end{proof}

With the help of  \cite[(6.18)]{Chen&Xu}, we can obtain a precise estimate of 
\begin{equation}\label{intemediate_est}
\|f_\phi-f(\theta)\|_{L^2(\Sn)}^2\leq C|\nabla f(\theta)|_{g_{\Sn}}^2\epsilon^2+C\epsilon^3.
\end{equation}
By (\ref{intemediate_est}) and Lemmas \ref{lemma5.5}, \ref{lemma6.7}, we have
\begin{equation}\label{6.38}
\begin{split}
F_2&=O\left(|\nabla f(\theta)|_{g_{\Sn}}^2\right)\epsilon^2+O(\epsilon^4),\\
\|v-1\|_{H^{2\sigma}(\Sn)}^2&\leq C|\nabla f(\theta)|_{g_{\Sn}}^2\epsilon^2+C\epsilon^3.
\end{split}
\end{equation}

\begin{lemma}
As $t\to\infty$, there holds
$$b=\frac{2\epsilon}{n+1}\left(\frac{\ud q^1}{\ud t},\cdots,\frac{\ud q^n}{\ud t},-\frac{\ud r}{\ud t}\right)
+O(|\nabla f(\theta)|_{g_{\Sn}}^2)\epsilon^2+O(\epsilon^3).$$
\end{lemma}
\begin{proof}
By (\ref{5.4}), we have
$$\frac{n}{2}b=\frac{n}{2}\fint_{\Sn}x(\alpha f_\phi-R_\sigma^h)\ud V_h=\fint_{\Sn}\xi \ud V_h=X+I,$$
where $X$ is defined in  (\ref{5.9a}), and
$I=\displaystyle\fint_{\Sn}\xi (v^{\frac{2n}{n-2\sigma}}-1)\ud V_{g_{\Sn}}$
can be estimated by
\begin{equation*}
|I|\leq C\|\xi\|_{L^\infty(\Sn)}\|v^{\frac{2n}{n-2\sigma}}-1\|_{L^2(\Sn)}
\leq C F_2^{\frac{1}{2}}\|v-1\|_{L^2(\Sn)}.
\end{equation*}
Now the assertion follows from
\eqref{5.9a} and (\ref{6.38}).
\end{proof}

\begin{lemma}\label{lemma6.9}
As $t\to\infty$, there holds
$$1-|\Theta|^2=\left(\frac{4n}{n-2}+o(1)\right)\epsilon^2.$$
\end{lemma}

We omit the proof of Lemma \ref{lemma6.9}
since it is the same as the proof of
\cite[Lemma 6.9]{Chen&Xu}.

\begin{prop}\label{prop6.10}
(i) As $t\to\infty$, we have
\begin{equation*}
\begin{split}
b-\langle b,\theta\rangle\theta&=\epsilon (\frac{4}{n}\alpha \ud f(\theta)+O(\epsilon)),\\
\langle b,\theta\rangle&=\epsilon^2\left(-\frac{8}{n(n-2)}\alpha\Delta_{\Sn} f(\theta)+O(1)|\nabla f(\theta)|_{g_{\Sn}}^2+O(\epsilon|\log\epsilon|)\right),
\\
\frac{\ud}{\ud t}(\Theta-\langle \Theta,\theta\rangle \theta)
&=4n^{-1}(n+1)\alpha\epsilon^2(\ud f(\theta)+O(\epsilon)),
\end{split}
\end{equation*}
and
$$\frac{\ud}{\ud t}(1-|\Theta|^2)=\frac{32(n+1)}{(n-2)^2}\alpha\epsilon^4(\Delta_{\Sn}f(\theta)+O(1)|\nabla f(\theta)|_{g_{\Sn}}^2+O(\epsilon|\log\epsilon|)).$$

(ii) As $t\to\infty$, the metrics $g(t)$ concentrate at a critical point $Q$ of $f$ where $\Delta_{\Sn} f(Q)\leq 0$.
\end{prop}
\begin{proof}
The proof of part \textit{(i)} follows directly from Lemmas \ref{lemma6.7}-\ref{lemma6.9}.
The proof of part \textit{(ii)} follows the same lines of \cite[Proposition 6.1(ii)]{Chen&Xu}, so we omit here.
\end{proof}

\begin{lemma}\label{lem:lim_energy}
As $t \to \infty$, there holds
$$E_f[u(t)]\to R_\sigma f(Q)^{\frac{2\sigma-n}{n}}\omega_n^{\frac{2\sigma}{n}},$$
where $Q=\lim_{t\to\infty}\Theta(t)$ is the unique limit of shadow flow 
$\Theta(t)$ asscoiated with $u(t)$. 
\end{lemma}

\section{Existence of conformal metrics}\label{Sect:Existence}

For $\gamma\in\mathbb{R}$, we define a sub-level set of $E_f$ by 
$$L_\gamma=\{u\in C_*^\infty; E_f[u]\leq \gamma\}.$$
For convenience, we label all critical points of $f$ by $p_1,\cdots, p_N$ such that $f(p_i)\leq f(p_{i+1})$ for $i=1,2,\cdots, N-1$. 
We set 
$$\beta_i=R_\sigma \omega_n^{\frac{2\sigma}{n}}f(p_i)^{\frac{2\sigma-n}{n}}.$$

Without loss of generality, we assume all critical levels $f(p_i)$ are different, where $1\leq i\leq N$. 
Hence, there exists a $\nu_0>0$ such that $\beta_i-2\nu_0>\beta_{i+1}$ for all $1\leq i\leq N-1$.

\begin{prop}\label{proposition7.1}
(i) If $\beta_1<\beta_0\leq\beta$, where $\beta$ has been chosen as in \eqref{beta}, then the set $L_{\beta_0}$ is contractible.\\
(ii) For any $0<\nu\leq \nu_0$, the sets 
$L_{\beta_i-\nu}$ and $L_{\beta_{i+1}+\nu}$ are homotopy equivalent for each $1\leq i\leq N-1$.\\
(iii) For each critical point $p_i$ of $f$ with $\Delta_{\Sn} f(p_i)>0$, the sets $L_{\beta_i+\nu_0}$ and $L_{\beta_i-\nu_0}$
are homotopy equivalent.\\
(iv) For each critical point $p_i$ with $\Delta_{\Sn} f(p_i)<0$, the set $L_{\beta_i+\nu_0}$ is homotopic to the set $L_{\beta_i-\nu_0}$ with $(n-\mathrm{ind}(f,p_i))$-cell attached. 
\end{prop}

With Proposition \ref{proposition7.1} at hands, we are now in a position to finish the proof of our main theorem.

\begin{proof}[Proof of Theorem \ref{main_thm}]
By contradiction, we suppose that the flow does not converge and $f$ cannot be realized as $2\sigma$-order $Q$-curvature of any conformal metric of $g_{\Sn}$. Then Proposition \ref{proposition7.1} shows that $L_{\beta_0}$ is contractible for some suitable chosen  $\beta_0$; in addition, the flow gives a 
homotopy equivalence of the set $L_{\beta_0}$ with a set 
$E_\infty$ whose homotopy type consists of a point $\{\theta_0\}$
with 
$(n-\mathrm{ind}(f,\theta))$-dimensional cell attached for each critical point $\theta$ of $f$ where $\Delta_{\Sn} f(\theta)<0$. 
By \cite[Theorem 4.3 on p. 36]{Chang}, we can conclude that 
the following identity 
\begin{equation}\label{7.1}
\sum_{j=0}^ns^j\gamma_j=1+(1+s)\sum_{j=0}^ns^j k_j
\end{equation}
holds with $k_j\geq 0$ and $\gamma_j$ is given in (\ref{gamma_j}). 
Equating the coefficients of $s$ in the polynomials on both sides of (\ref{7.1}), we obtain (\ref{assump:Morse_Sys_Cond}), 
which violates the hypothesis in Theorem \ref{main_thm}
and thus leads to the desired contradiction. 
\end{proof}

Before going to prove Proposition \ref{proposition7.1}, we first need continuous dependence  of the flow \eqref{1.8} on initial data, which is necessary in the construction of  homotopy equivalences.
\begin{lemma}\label{lem:cont_dep_u_0}
Given any real number $T>0$, let $u_i(t)=u(t,u_i^0)$ be the solutions 
of the flow \eqref{1.8} with initial data
$u_i^0\in C^\infty_f$, where $i=1, 2$. 
Then there exists a constant $C>0$, depending on $T$, $n$
and $\|u_i\|_{L^\infty([0,T]; C^{2N}(\Sn))}$ for $i=1,2$ such that
$$\sup_{0\leq t\leq T}\|u_1(t)-u_2(t)\|_{H^{2N}(\Sn)}
\leq C\|u_1^0-u_2^0\|_{H^{2N}(\Sn)} \quad \mathrm{with~~} N=\left[\frac{n}{2\sigma-1}\right]+1.$$
\end{lemma}
\begin{proof}
It follows from Theorem \ref{thm:1} that
$u_i(t), i=1,2$, are smooth in the time interval $[0,T]$.
Moreover, by Lemma \ref{lemma1.9}, there exist
two constants $C_i(T)$ depending on $\|u_i^0\|_{L^\infty(\Sn)}$
such that
\begin{equation}\label{7.2a}
C_i^{-1}\leq \|u_i(t)\|_{L^\infty(\Sn\times[0,T])}\leq C_i
\quad\mathrm{ for~~}i=1,2.
\end{equation}
For any positive smooth function $u$, we set
$$\alpha[u]=\frac{\int_{\Sn} R_\sigma^{g} \ud V_{g}}{\int_{\Sn} f \ud V_{g}}=\frac{\int_{\Sn} uP_\sigma(u) \ud V_{g_{\Sn}}}{\int_{\Sn} f u^{\frac{2n}{n-2\sigma}}\ud V_{g_{\Sn}}},$$
where $g=u^{4/(n-2\sigma)}g_{\Sn}$.
We let $g_i=u_i(t)^{4/(n-2\sigma)}g_{\Sn}$ and
\begin{equation*}
\alpha[u_i]=\alpha_i(t)
=\frac{\int_{\Sn} R_\sigma^{g_i} \ud V_{g_i}}{\int_{\Sn} f \ud V_{g_i}}.
\end{equation*}
Letting $w=u_2-u_1$, by (\ref{7.2a}), Proposition \ref{prop:properties} and Lemma \ref{lemma1.4} we have
\begin{align}\label{7.2}
&\alpha[u_2]-\alpha[u_1]\no\\
=&\int_0^1\frac{\partial}{\partial s}\alpha[u_1+sw] \ud s\no\\
=&\int_0^1 \left[\frac{2\int_{\Sn}((1-s)P_\sigma(u_1)+sP_\sigma(u_2))w\ud V_{g_{\Sn}}}{\int_{\Sn}f (u_1+sw)^{\frac{2n}{n-2\sigma}}\ud V_{g_{\Sn}}}\right.\no\\
&\left.-\frac{2n}{n-2\sigma}
\frac{\alpha[u_1+sw]}{\int_{\Sn}f (u_1+sw)^{\frac{2n}{n-2\sigma}}\ud V_{g_{\Sn}}}\int_{\Sn}f (u_1+sw)^{\frac{n+2\sigma}{n-2\sigma}}w\ud V_{g_{\Sn}}\right]\ud s\no\\
=&\int_0^1 \left[\frac{2\int_{\Sn}((1-s)R^{g_1}_\sigma u_1^{\frac{n+2\sigma}{n-2\sigma}}+sR^{g_2}_\sigma u_2^{\frac{n+2\sigma}{n-2\sigma}})w\ud V_{g_{\Sn}}}{\int_{\Sn}f (u_1+sw)^{\frac{2n}{n-2\sigma}}\ud V_{g_{\Sn}}}\right.\no\\
&\left.-\frac{2n}{n-2\sigma}
\frac{\alpha[u_1+sw]}{\int_{\Sn}f (u_1+sw)^{\frac{2n}{n-2\sigma}}\ud V_{g_{\Sn}}}\int_{\Sn}f (u_1+sw)^{\frac{n+2\sigma}{n-2\sigma}}w\ud V_{g_{\Sn}}\right]\ud s\no\\
\leq& C(\|R_\sigma^{g_1}\|^2_{L^2(\Sn)}+\|R_\sigma^{g_2}\|^2_{L^2(\Sn)}+\mathcal{S}(g_1)+\mathcal{S}(g_2))\|w\|_{L^2(\Sn)}\no\\
\leq& C\|w\|_{L^2(\Sn)}.
\end{align}
From (\ref{1.7}),
we have
\begin{align}\label{7.3}
w_t&=-\frac{n-2\sigma}{4}\left[(R_\sigma^{g_2}-\alpha[u_2] f)u_2
-(R_\sigma^{g_1}-\alpha[u_1] f)u_1\right]\no\\
&=-\frac{n-2\sigma}{4}\left[\left(P_\sigma(u_2)u_2^{-\frac{n+2\sigma}{n-2\sigma}}-\alpha[u_2] f\right)u_2
-\left(P_\sigma(u_1)u_1^{-\frac{n+2\sigma}{n-2\sigma}}-\alpha[u_1] f\right)u_1\right]\no\\
&=-\frac{n-2\sigma}{4}\left[u_2^{-\frac{4\sigma}{n-2\sigma}}P_\sigma(w)+d(x,t)w\right]+
\frac{n-2\sigma}{4}(\alpha[u_2]-\alpha[u_1])u_2f,
\end{align}
where
$$d(x,t)=\frac{u_2^{-\frac{4\sigma}{n-2\sigma}}-u_1^{-\frac{4\sigma}{n-2\sigma}}}{u_2-u_1}P_\sigma(u_1)-\alpha[u_1] f.$$
Thus, from (\ref{7.3}), we have
\begin{equation*}
\begin{split}
\frac{\ud}{\ud t}\int_{\Sn} w^2 \ud V_{g_{\Sn}}
=&2\int_{\Sn} w w_t \ud V_{g_{\Sn}}\\
=&-\frac{n-2\sigma}{2}\int_{\Sn}\left[u_2^{-\frac{4\sigma}{n-2\sigma}}wP_\sigma(w)+d(x,t)w^2\right]\ud V_{g_{\Sn}}\\
&+\frac{n-2\sigma}{2}\int_{\Sn}(\alpha[u_2]-\alpha[u_1])u_2f w\ud V_{g_{\Sn}}.
\end{split}
\end{equation*}
By H\"{o}lder's and Young's inequalities, we have
\begin{align*}
&\int_{\Sn}u_2^{-\frac{4\sigma}{n-2\sigma}}wP_\sigma(w)\ud V_{g_{\Sn}}\\
=&\frac{C_{n,-\sigma}}{2}\int_{\Sn}\int_{\Sn} \frac{(w(x)-w(y))\left(w(x)u_2^{-\frac{4\sigma}{n-2\sigma}}(x)-w(y)u_2^{-\frac{4\sigma}{n-2\sigma}}(y) \right)}{|x-y|^{n+2\sigma}}\ud V_{g_{\Sn}}(x)\ud V_{g_{\Sn}}(y)\\
&+R_\sigma \int_{\Sn} u_2^{-\frac{4\sigma}{n-2\sigma}} w^2 \ud V_{g_{\Sn}}\\
\ge& \frac{1}{2} \min_{\Sn} u_2^{-\frac{4\sigma}{n-2\sigma}} \cdot \|w\|_{H^{\sigma}(\Sn)}^{2}-C \max_{\Sn} u_2^{-\frac{4\sigma}{n-2\sigma}} \int_{\Sn} w^2\ud V_{g_{\Sn}}. 
\end{align*}
Hence, we have 
\[
\frac{\ud}{\ud t}\int_{\Sn} w^2 \ud V_{g_{\Sn}} + C_1 \|w\|_{H^{\sigma}(\Sn)}^{2} \le C_2\int_{\Sn} w^2 \ud V_{g_{\Sn}}
\]
for some constants $C_1, C_2>0$. Integrating the above inequality over $(0,t)$, we have  
\begin{equation}\label{ineq:L^2_stability}
\int_{\Sn} w(t)^2 \ud V_{g_{\Sn}}\le e^{C_1 T}\int_{\Sn} w(0)^2 \ud V_{g_{\Sn}} \quad \mathrm{for~~} 0<t \leq T.
\end{equation}

Next, for any positive integer $k\le N$, by \eqref{7.3} we have 
\begin{align*}
&\frac{\ud }{\ud t} \int_{\Sn} |(-\Delta_{g_{\Sn}})^k w |^2 \ud V_{g_{\Sn}}= 2\int_{\Sn} w_t (-\Delta_{g_{\Sn}})^{2k} w \ud V_{g_{\Sn}}\\
=&
-\frac{n-2\sigma}{2}\int_{\Sn} (-\Delta_{g_{\Sn}})^{2k} w \left[u_2^{-\frac{4\sigma}{n-2\sigma}}P_\sigma(w)+d(x,t)w\right]\ud V_{g_{\Sn}}\\
&
+\frac{n-2\sigma}{2}\int_{\Sn} (-\Delta_{g_{\Sn}})^{2k} w (\alpha[u_2]-\alpha[u_1])u_2f \ud V_{g_{\Sn}}.
\end{align*}
Note that $(-\Delta_{g_{\Sn}})P_\sigma=P_\sigma (-\Delta_{g_{\Sn}})$. Hence, 
\begin{align*}
&\int_{\Sn}(-\Delta_{g_{\Sn}})^{2k} w u_2^{-\frac{4\sigma}{n-2\sigma}}P_\sigma(w) \ud V_{g_{\Sn}}\\
=& \int_{\Sn}(-\Delta_{g_{\Sn}})^{k} w (-\Delta_{g_{\Sn}})^k \left(u_2^{-\frac{4\sigma}{n-2\sigma}}P_\sigma(w)\right)\ud V_{g_{\Sn}}\\
\ge& \frac{1}{2} \min_{\Sn} u_2^{-\frac{4\sigma}{n-2\sigma}} \cdot \|(-\Delta_{g_{\Sn}})^kw\|_{H^{\sigma}(\Sn)}^{2} -C \int_{\Sn} w^2 \ud V_{g_{\Sn}},
\end{align*}
where we have used the following interpolation inequalities: $\forall~ \epsilon>0$, there holds
\begin{align*}
\int_{\Sn} |\nabla P_{\sigma} \Big((-\Delta_{g_{\Sn}})^l w\Big)|_{g_{\Sn}}^2 \ud V_{g_{\Sn}} \le& C \|w\|_{H^{2l+2\sigma+1}(\Sn)}^2\\
 \le& \va \|(-\Delta_{g_{\Sn}})^kw\|_{H^{\sigma}(\Sn)}^{2} +C(\va)\int_{\Sn} w^2 \ud V_{g_{\Sn}}
\end{align*}
for any integer $0\le l<k$. Other terms can be estimated similarly. Therefore, we have 
\[
\frac{\ud }{\ud t} \int_{\Sn} |(-\Delta_{g_{\Sn}})^k w |^2 \ud V_{g_{\Sn}} +C_3 \|(-\Delta_{g_{\Sn}})^kw\|_{H^{\sigma}(\Sn)}^{2} \le C_4 \int_{\Sn} w^2  \ud V_{g_{\Sn}}
\]
for some constants $C_3,C_4>0$ depending only $T$, $n$
and $\|u_i\|_{L^\infty([0,T]; C^{2n}(\Sn))}$ with $i=1,2$. For any $t \in (0,T)$, integrating the above inequality over $(0,t)$ and using \eqref{ineq:L^2_stability}, we complete the proof. 
\end{proof}

\begin{proof}[Proof of Proposition \ref{proposition7.1} (i) and (ii)]
Let $u_0 \in L_{\beta_0}$, and $u(t,u_0)$ be the solution of the flow \eqref{1.8} with the initial datum $u_0$ and $v(t,u_0)$ be its normalized flow, then by Proposition \ref{prop:properties} we have
$$E_f[u(t,u_0)]\leq \beta_0.$$

\textit{(i)} Given any finite time $T_1>0$, we define a map on $L_{\beta_0}$ by
$$
H(s,u_0)=
\begin{cases}
u(3sT_1, u_0) & \quad\mathrm{if~~} 0 \le s \le \frac{1}{3}, \\
[(2-3s)u(T_1,
u_0)^{\frac{2n}{n-2\sigma}} + (3s -
1)\det(\ud\phi^{-1}) ]^{\frac{n-2\sigma}{2n}}
&\quad \mathrm{if~~}\frac{1}{3} \le s \le \frac{2}{3}, \\
[\det(\ud[\Psi \circ \delta_{- q(T_1), {3(1 - s) r(T_1) + (3s -
2)}} \circ \pi ] )]^{\frac{n-2\sigma}{2n}} &\quad
\mathrm{if~~} \frac{2}{3} \le s \le 1.
\end{cases}
$$

Due to the length of the proof, we divide it into three steps.

\vskip 8pt

\emph{Step 1.} For any $\va\in (0,\va_0)$ with $\va_0\leq 1/2$ to be determined later, if  \footnote{Only one $T_1$ is enough in the proof of  Proposition \ref{proposition7.1} \textit{(i)}. Indeed, the smallness of $\|v(t)-1\|_{C^1(\Sn)}$ in \cite[formula (7.6)]{Chen&Xu} is not necessarily required for all $t \geq T_1$, but just at some fixed $T_1$. Comparing to \cite{Chen&Xu}, an explicit gap $\va_0$ of $\|v(T_1,u_0)-1\|_{C^1(\Sn)}$ to guarantee \eqref{E_f_below_beta_0} is presented.}  
\begin{equation*}
\|v(T_1,u_0)-1\|_{C^1(\Sn)}<\va
\end{equation*}
for some $T_1>0$, then
\begin{equation}\label{E_f_below_beta_0}
E_f[ H(s, u_0)] \le
\beta_0 \qquad  \mathrm{for} \quad s \in \left[\frac{1}{3},\frac{2}{3}\right].
\end{equation}

\vskip 8pt

To this end,  for convenience, let $F(s)=E_f[H(s,u_0)]$ for
$1/3 \leq s \leq 2/3$, and $v=v(T_1,u_0)$. our strategy is to show that any critical point of $E_f[H(s,u_0)]$ on the interval $s \in (1/3, 2/3)$ (if exists) must be a local minimum point, that is, if  $\dot F(s_0)=0$ for some $s_0 \in (1/3, 2/3)$, then $\ddot F(s_0)>0$.

First by conformal invariance of the
energy, we have
$$ E_f[H(s, u_0)] = E_{f\circ\phi}[ H(s, u_0)\circ\phi (\det \ud\phi )^{\frac{n-2\sigma}{2n}}].$$
Thus, we conclude that $F(s)$ achieves its maximum value only at
$s=1/3$ or $s=2/3$, namely,
\begin{equation*}\label{eq7.1}
E_f[H(s,u_0)] \leq  \max\left\{E_f[u(T_1,
u_0)],E_f[(\det \ud\phi^{-1})^{\frac{n-2\sigma}{2n}}]\right\} \leq \beta_0,
 \mathrm{~~for~~} s \in \left[\frac{1}{3},\frac{2}{3}\right].
\end{equation*}

If we set
$$u_s^{\frac{2n}{n-2\sigma}}=(2-3s)u^{\frac{2n}{n-2\sigma}}+(3s-1)\det(\ud\phi^{-1})$$
and
\begin{equation}
\label{def:v_s} 
v_s^{\frac{2n}{n-2\sigma}} = (2-3s)
v^{\frac{2n}{n-2\sigma}} + (3s - 1)
\end{equation}
for $1/3 \le s \le 2/3$, then $E_f[u_s]=E_{f\circ\phi}[v_s]$. Moreover, it naturally holds
\begin{equation*}\label{normalizations_v_s}
\fint_{\Sn} v_s^{\frac{2n}{n-2\sigma}} \ud V_{g_{\Sn}}=1\;\;\; \mathrm{ and }
\;\;\; \fint_{\Sn} x v_s^{\frac{2n}{n-2\sigma}} \ud V_{g_{\Sn}} = 0.
\end{equation*}
These imply that
\begin{equation}\label{eqn:ds_normalizations}
\int_{\Sn} v_s^{\frac{n+2\sigma}{n-2\sigma}} \dot{v}_s
\ud V_{g_{\Sn}}=0 \;\;\; \mathrm{ and } \;\;\; \int_{\Sn} x v_s^{\frac{n+2\sigma}{n-2\sigma} }{\dot v_s} \ud V_{g_{\Sn}} = 0.
\end{equation}

By \eqref{def:v_s}  we have
$$ \nabla v_s=\left(\frac{v}{v_s}\right)^{\frac{n+2\sigma}{n-2\sigma}}(2-3s)\nabla v, \qquad \dot{v}_s = \frac{3(n-2\sigma)}{2n} v_s^{-\frac{n+2\sigma}{n-2\sigma}}(1 - 
v^{\frac{2n}{n-2\sigma}}) $$
and
$$\ddot{v}_s = -\frac{n+2\sigma}{n-2\sigma}v_s^{-1}|\dot{v}_s|^2.$$
Thus, we obtain
\begin{equation}\label{est:v_Dv}
\frac{1}{2}\leq v_s\leq \frac{3}{2}\quad \mathrm{and} \quad  \|\nabla v_s\|_{C(\Sn)}\leq 3^{\frac{n+2\sigma}{n-2\sigma}}\|\nabla v\|_{C(\Sn)}\leq 3^{\frac{n+2\sigma}{n-2\sigma}}\va . 
\end{equation}
Using an elementary inequality that for any $p>1$,
$$|1-t^p|\leq p(1+t^{p-1})|1-t| \quad \mathrm{~~for~~all~~} t \geq 0,$$
we obtain
$$|v_s^{\frac{n+2\sigma}{n-2\sigma}}-1|\leq |v_s^{\frac{2n}{n-2\sigma}}-1|=(2-3s)|v^{\frac{2n}{n-2\sigma}}-1|\leq 2^\ast (1+v^{\frac{n+2\sigma}{n-2\sigma}})|v-1|,$$
then
\begin{equation}\label{est:v^p-1}
\|v_s^{\frac{n+2\sigma}{n-2\sigma}}-1\|_{C(\Sn)}\leq 4\cdot 3^{\frac{n+2\sigma}{n-2\sigma}} \va.
\end{equation}

For any smooth function $f$ with $\int_{\Sn}f u^{2n/(n-2\sigma)}\ud V_{g_{\Sn}}>0$, a direct computation yields
\begin{align*}
&\ud E_f[u](\eta)\\
=&2\left(\int_{\Sn}f u^{\frac{2n}{n-2\sigma}}\ud V_{g_{\Sn}}\right)^{-\frac{n-2\sigma}{n}}\left[\int_{\Sn}\eta P_\sigma(u) \ud V_{g_{\Sn}}
-\frac{\int_{\Sn}uP_\sigma(u)\ud V_{g_{\Sn}}}{\int_{\Sn}f u^{\frac{2n}{n-2\sigma}}\ud V_{g_{\Sn}}}\int_{\Sn}f u^{\frac{n+2\sigma}{n-2\sigma}}\eta \ud V_{g_{\Sn}}\right]
\end{align*}
and 
\begin{align*}
&\ud^2E_f[u](\eta,\zeta)\\
=&\left.\frac{\ud}{\ud r}\Big(\ud E_f[u+r\zeta](\eta)\Big)\right|_{r=0}\\
=&2\left(\int_{\Sn}f u^{\frac{2n}{n-2\sigma}}\ud V_{g_{\Sn}}\right)^{\frac{2(\sigma-n)}{n}}\times \\
&\left [
\int_{\Sn}\zeta P_\sigma(\eta)\ud V_{g_{\Sn}}\int_{\Sn}f u^{\frac{2n}{n-2\sigma}}\ud V_{g_{\Sn}}
\vphantom{\frac{\int_{\Sn}u P_\sigma(u)\ud V_{g_{\Sn}}}{\int_{\Sn}f u^{\frac{2n}{n-2\sigma}}\ud V_{g_{\Sn}}}} -\frac{n+2\sigma}{n-2\sigma}
\int_{\Sn}u P_\sigma(u)\ud V_{g_{\Sn}}\int_{\Sn}f u^{\frac{4\sigma}{n-2\sigma}}\zeta\eta \ud V_{g_{\Sn}}
\right.\\
&
~~-2\int_{\Sn}\zeta P_\sigma(u)\ud V_{g_{\Sn}}\int_{\Sn}f u^{\frac{n+2\sigma}{n-2\sigma}}\eta \ud V_{g_{\Sn}}-2\int_{\Sn}\eta P_\sigma(u)\ud V_{g_{\Sn}}\int_{\Sn}f u^{\frac{n+2\sigma}{n-2\sigma}}\zeta \ud V_{g_{\Sn}}\\
&\left.~~+\frac{4(n-\sigma)}{n-2\sigma}\frac{\int_{\Sn}u P_\sigma(u)\ud V_{g_{\Sn}}}{\int_{\Sn}f u^{\frac{2n}{n-2\sigma}}\ud V_{g_{\Sn}}}\int_{\Sn}f u^{\frac{n+2\sigma}{n-2\sigma}}\zeta \ud V_{g_{\Sn}}
\int_{\Sn}f u^{\frac{n+2\sigma}{n-2\sigma}}\eta \ud V_{g_{\Sn}}\right]
\end{align*}
for any $\eta, \zeta\in H^\sigma(\Sn)$. 
Note that the Sobolev inequality \eqref{ineq:Sobolev} shows that the mapping 
$$u\mapsto \ud^2E_f[u](\cdot,\cdot)\in L(H^\sigma(\Sn)\times H^\sigma(\Sn);\mathbb{R})$$
is continuous. 

If $\dot F(s_0)=0$, i.e. $\ud E_{f\circ \phi}[v_s](\dot v_s)|_{s=s_0}=0$, then at $s=s_0$,
\begin{equation}\label{eqn:critical_pt}
\int_{\Sn}P_\sigma(v_s) \dot v_s  \ud V_{g_{\Sn}}=a(s)E[v_s],
\end{equation}
where 
$$a(s):=\frac{\int_{\Sn} f\circ \phi v_s^{\frac{n+2\sigma}{n-2\sigma}}\dot v_s  \ud V_{g_{\Sn}}}{\int_{\Sn}fv_s^{\frac{2n}{n-2\sigma}} \ud V_{g_{\Sn}}} \quad \mathrm{and}\quad E[v_s]:=\int_{\Sn}v_s P_\sigma(v_s)\ud V_{g_{\Sn}}.$$
Thus, at $s=s_0$, by \eqref{eqn:critical_pt} we obtain
\begin{align}\label{ineq:2nd_order_F(s)}
\ddot F(s)=&\frac{\ud^2}{\ud s^2}E_{f\circ \phi}[v_s]=\ud^2E_{f\circ \phi}[v_s](\dot{v}_s,\dot{v}_s)+\ud E_{f\circ\phi}[v_s](\ddot{v}_s)\no\\
=&2 \left(\int_{\Sn}f\circ \phi
v_s^{2^\ast}\ud V_{g_{\Sn}}\right)^{\frac{2\sigma-n}{n}}\times\no\\
&\left\{ \int_{\Sn}\dot v_s P_\sigma(\dot v_s)\ud V_{g_{\Sn}}+\frac{4(n-\sigma)}{n-2\sigma} E[v_s]a(s)^2-4a(s)\int_{\Sn}\dot v_s P_\sigma (v_s)\ud V_{g_{\Sn}}\right.\no\\
&\qquad\left.-\frac{n+2\sigma}{n-2\sigma} \int_{\Sn}v_s^{-1} |\dot v_s|^2 P_\sigma (v_s)\ud V_{g_{\Sn}}\right\}\no\\
=&2 \left(\int_{\Sn}f\circ \phi
v_s^{2^\ast}\ud V_{g_{\Sn}}\right)^{\frac{2\sigma-n}{n}}\times\no\\
&\left\{ \int_{\Sn}\dot v_s P_\sigma(\dot v_s)\ud V_{g_{\Sn}}+\frac{4\sigma}{n-2\sigma} E[v_s]a(s)^2-\frac{n+2\sigma}{n-2\sigma} \int_{\Sn}v_s^{-1} |\dot v_s|^2 P_\sigma (v_s)\ud V_{g_{\Sn}}\right\}\no\\
\geq& 2\left(\int_{\Sn}f\circ \phi
v_s^{2^\ast}\ud V_{g_{\Sn}}\right)^{-2/2^\ast}\left\{ \int_{\Sn}\left(\dot v_s P_\sigma(\dot v_s)-\frac{n+2\sigma}{n-2\sigma}R_\sigma|\dot v_s|^2\right)\ud V_{g_{\Sn}}\right.\no\\
&\qquad\left.-\frac{n+2\sigma}{n-2\sigma} \int_{\Sn}\left(v_s^{-1} |\dot v_s|^2 P_\sigma (v_s)-R_\sigma |\dot v_s|^2\right)\ud V_{g_{\Sn}}\right\}.
\end{align}

By \eqref{est:v_Dv} and H\"{o}lder's inequality, we have
\begin{align}\label{est:2nd_diff}
&\left|\int_{\Sn}\left(v_s^{-1} |\dot v_s|^2 P_\sigma (v_s)-R_\sigma |\dot v_s|^2\right)\ud V_{g_{\Sn}}\right|\no\\
=&\frac{C_{n,-\sigma}}{2}\left|\int_{\Sn\times \Sn}\frac{[v_s^{-1}(x)\dot v_s^2(x)-v_s^{-1}(y)\dot v_s^2(y)](v_s(x)-v_s(y))}{|x-y|^{n+2\sigma}}\ud V_{g_{\Sn}}(x)\ud V_{g_{\Sn}}(y)\right|\no\\
\leq& \frac{C_{n,-\sigma}}{2} \int_{\Sn\times \Sn}\frac{|v_s^{-1}(x)-v_s^{-1}(y)|\dot v_s^2(x)|v_s(x)-v_s(y)|}{|x-y|^{n+2\sigma}}\ud V_{g_{\Sn}}(x)\ud V_{g_{\Sn}}(y)\no\\
&+\frac{C_{n,-\sigma}}{2}\int_{\Sn\times \Sn}\frac{|\dot v_s^2(x)-\dot v_s^2(y)|v_s^{-1}(y)|v_s(x)-v_s(y)|}{|x-y|^{n+2\sigma}}\ud V_{g_{\Sn}}(x)\ud V_{g_{\Sn}}(y)\no\\
\leq& 2 C_{n,-\sigma}\|\nabla v_s\|_{C(\Sn)}^2 \int_{\Sn\times \Sn}\frac{\dot v_s^2(x)\ud V_{g_{\Sn}}(x)\ud V_{g_{\Sn}}(y)}{|x-y|^{n+2\sigma-2}}\no\\
&+C_{n,-\sigma} \|\nabla v_s\|_{C(\Sn)}\int_{\Sn\times \Sn}\frac{|\dot v_s(x)+\dot v_s(y)||\dot v_s(x)-\dot v_s(y)|}{|x-y|^{n+2\sigma-1}}\ud V_{g_{\Sn}}(x)\ud V_{g_{\Sn}}(y)\no\\
\leq&2 C_{n,-\sigma}\|\nabla v_s\|_{C(\Sn)}^2c_1^2\|\dot{v}_s\|_{L^2(\Sn)}^2\no\\
&+\sqrt{C_{n,-\sigma}} \|\nabla v_s\|_{C(\Sn)}\left[\frac{C_{n,-\sigma}}{2}\int_{\Sn\times \Sn}\frac{|\dot v_s(x)-\dot v_s(y)|^2}{|x-y|^{n+2\sigma}}\ud V_{g_{\Sn}}(x)\ud V_{g_{\Sn}}(y)\right]^{\frac{1}{2}}\cdot\no\\
&\qquad\left[2\int_{\Sn\times \Sn}\frac{(\dot v_s(x)+\dot v_s(y))^2}{|x-y|^{n+2\sigma-2}}\ud V_{g_{\Sn}}(x)\ud V_{g_{\Sn}}(y)\right]^{\frac{1}{2}}\no\\
\leq& \left(\frac{2 c_1^2C_{n,-\sigma}}{R_\sigma}\|\nabla v_s\|_{C(\Sn)}^2+2c_1\sqrt{\frac{C_{n,-
\sigma}}{R_\sigma}}\|\nabla v_s\|_{C(\Sn)}\right) \|\dot v_s\|_{H^\sigma(\Sn)}^2\no\\
\leq& 3^{\frac{n+2\sigma}{n-2\sigma}}\va\left(2c_1\sqrt{\frac{C_{n,-
\sigma}}{R_\sigma}}+3^{\frac{n+2\sigma}{n-2\sigma}}\frac{ c_1^2C_{n,-\sigma}}{R_\sigma}\right) \|\dot v_s\|_{H^\sigma(\Sn)}^2:=\frac{n-2\sigma}{n+2\sigma}b_0 \va\|\dot v_s\|_{H^\sigma(\Sn)}^2,
\end{align}
where
\begin{align*}
c_1:=&\sqrt{\max_{x \in \Sn}\int_{\Sn}\frac{\ud V_{g_{\Sn}}(y)}{|x-y|^{n+2\sigma-2}}}=\sqrt{\int_{\Sn}\frac{\ud V_{g_{\Sn}}(y)}{|S-y|^{n+2\sigma-2}}}\\
=&\sqrt{2^{2(1-\sigma)}\int_{\R^n}\frac{\ud z}{(1+|z|^2)^{\frac{n}{2}+1-\sigma}}}=\sqrt{2^{1-2\sigma}\omega_{n-1}\frac{\Gamma(\frac{n}{2})\Gamma(1-\sigma)}{\Gamma(\frac{n}{2}+1-\sigma)}}.
\end{align*}

Now we decompose $\dot{v}_s=\psi+w$, where
$$w= \fint_{\Sn}\dot{v}_s\ud V_{g_{\Sn}}
+\sum\limits_{i=1}^{n+1}(\fint_{\Sn}\dot{v}_s
\varphi_i\ud V_{g_{\Sn}})\varphi_i,$$
where 
$$\varphi_i=\sqrt{n+1}~ x_i \qquad \mathrm{for~~}1 \leq i \leq n+1.$$
By \eqref{eqn:ds_normalizations}  we have
\begin{align*}
\frac{2n}{n-2\sigma} \fint_{\Sn}\dot v_s \ud V_{g_{\Sn}}=\frac{2n}{n-2\sigma} \fint_{\Sn}\dot v_s(1-v_s^{\frac{n+2\sigma}{n-2\sigma}}) \ud V_{g_{\Sn}}.
\end{align*}
and 
\begin{align*}
\frac{2n}{n-2\sigma} \fint_{\Sn}\varphi_i\dot v_s \ud V_{g_{\Sn}}=\frac{2n}{n-2\sigma} \fint_{\Sn}\varphi_i\dot v_s(1-v_s^{\frac{n+2\sigma}{n-2\sigma}}) \ud V_{g_{\Sn}}.
\end{align*}
These together with \eqref{est:v^p-1} and H\"older's inequality imply that
\begin{align}\label{est:zero_coeff_w}
\left|\fint_{\Sn}\dot v_s \ud V_{g_{\Sn}}\right|\leq& \left(\fint_{\Sn}|\dot v_s|^2 \ud V_{g_{\Sn}}\right)^{\frac{1}{2}} \|1-v_s^{\frac{n+2\sigma}{n-2\sigma}}\|_{C(\Sn)}\no\\
\leq&4 \cdot 3^{\frac{n+2\sigma}{n-2\sigma}} \va \left(\fint_{\Sn}|\dot v_s|^2 \ud V_{g_{\Sn}}\right)^{\frac{1}{2}}
\end{align}
and
\begin{align}\label{est:1st_coeff_w}
\left|\fint_{\Sn}\varphi_i\dot v_s d \mu_{\Sn}\right|\leq& \sqrt{n+1} \left(\fint_{\Sn}|\dot v_s|^2 \ud V_{g_{\Sn}}\right)^{\frac{1}{2}}\|1-v_s^{\frac{n+2\sigma}{n-2\sigma}}\|_{C(\Sn)}\no\\
\leq& 4\sqrt{n+1} \cdot 3^{\frac{n+2\sigma}{n-2\sigma}} \va \left(\fint_{\Sn}|\dot v_s|^2 \ud V_{g_{\Sn}}\right)^{\frac{1}{2}}.
\end{align}
Then, by \eqref{est:zero_coeff_w} and \eqref{est:1st_coeff_w} we obtain
\begin{align*}
&\fint_{\Sn} \dot{v}_sP_\sigma(\dot v_s)\ud  V_{g_{\Sn}}\\
=& \fint_{\Sn}\psi P_\sigma \psi \ud  V_{g_{\Sn}}+\fint_{\Sn}w P_\sigma w \ud  V_{g_{\Sn}} \\
\geq& \Lambda_2 \fint_{\Sn}|\psi|^2\ud  V_{g_{\Sn}}+\Lambda_1 \sum_{i=1}^{n+1}\left(\fint_{\Sn}\dot{v}_s
\varphi_i \ud V_{g_{\Sn}}\right)^2+R_\sigma \left(\fint_{\Sn} \dot v_s \ud  V_{g_{\Sn}}\right)^2\\
=&\Lambda_2\fint_{\Sn}|\dot{v}_s|^2\ud V_{g_{\Sn}}-(\Lambda_2-\Lambda_1)\sum_{i=1}^{n+1}\left(\fint_{\Sn}\dot{v}_s
\varphi_i \ud V_{g_{\Sn}}\right)^2-(\Lambda_2-\Lambda_0) \left(\fint_{\Sn} \dot v_s \ud  V_{g_{\Sn}}\right)^2\\
=&\left[\Lambda_2-(\Lambda_2-\Lambda_1)(4(n+1) \cdot 3^{\frac{n+2\sigma}{n-2\sigma}} \va)^2-(\Lambda_2-\Lambda_0)(4\cdot 3^{\frac{n+2\sigma}{n-2\sigma}} \va)^2\right]\fint_{\Sn}|\dot{v}_s|^2\ud V_{g_{\Sn}}>0.
 \end{align*}
Here we require that
 \begin{align*}
 &\Lambda_2-[(\Lambda_2-\Lambda_1)(n+1)^2+\Lambda_2-\Lambda_0](4 \cdot 3^{\frac{n+2\sigma}{n-2\sigma}} \va)^2>0 \\
 \Leftrightarrow \quad& 0<\va<\frac{1}{4}\sqrt{\frac{\Lambda_2}{(\Lambda_2-\Lambda_1)(n+1)^2+\Lambda_2-\Lambda_0}}3^{-\frac{n+2\sigma}{n-2\sigma}}\\
 &\qquad\quad=\frac{1}{8}\sqrt{\frac{(n+2+2\sigma)(n+2\sigma)}{\sigma(n+1)((n+1)(n+2\sigma)+2)}}3^{-\frac{n+2\sigma}{n-2\sigma}}.
 \end{align*}
Hence, recalling that $\Lambda_1=\frac{n+2\sigma}{n-2\sigma}R_\sigma$, by \eqref{est:2nd_diff} we conclude that
\begin{align*}
&\int_{\Sn}\left(\dot v_s P_\sigma(\dot v_s)-\frac{n+2\sigma}{n-2\sigma}R_\sigma|\dot v_s|^2\right)\ud V_{g_{\Sn}}-\frac{n+2\sigma}{n-2\sigma} \int_{\Sn}(v_s^{-1} |\dot v_s|^2 P_\sigma (v_s)-R_\sigma |\dot v_s|^2)\ud V_{g_{\Sn}}\\
\geq&\left[1-b_0\va-\frac{\Lambda_1}{\Lambda_2-[(\Lambda_2-\Lambda_1)(n+1)^2+\Lambda_2-\Lambda_0](4 \cdot 3^{\frac{n+2\sigma}{n-2\sigma}} \va)^2}\right]\int_{\Sn}\dot v_s P_\sigma(\dot v_s) \ud  V_{g_{\Sn}}>0,
\end{align*}
if we choose $\va \in (0,\va_0)$ with 
\begin{equation}\label{def:va_0}
\va_0=\min\left\{\frac{1}{2},\frac{1}{8}\sqrt{\frac{(n+2+2\sigma)(n+2\sigma)}{\sigma(n+1)((n+1)(n+2\sigma)+2)}}3^{-\frac{n+2\sigma}{n-2\sigma}},\va_1\right\},
\end{equation}
where $\va_1$ is the positive root of the following algebraic equation
$$1-\frac{\Lambda_1}{\Lambda_2}-b_0\va-\frac{[(\Lambda_2-\Lambda_1)(n+1)^2+\Lambda_2-\Lambda_0](4 \cdot 3^{\frac{n+2\sigma}{n-2\sigma}} )^2}{\Lambda_2}\va^2=0$$
i.e.
$$\frac{[(\Lambda_2-\Lambda_1)(n+1)^2+\Lambda_2-\Lambda_0](4 \cdot 3^{\frac{n+2\sigma}{n-2\sigma}} )^2}{\Lambda_2}\va^2+b_0\va-\frac{4\sigma}{n+2+2\sigma}=0$$
and $b_0$ is given in \eqref{est:2nd_diff}.

Therefore, we conclude that for any $0<\va<\va_0$, there holds $\ddot F(s_0)>0$ whenever $\dot F(s_0)=0$. 
This directly implies \eqref{E_f_below_beta_0}.

\vskip 8pt

\emph{Step 2.} Fix $\va$ as above, there exists a $T_1=T_1(\va,u_0)>0$ which is continuously dependent on $u_0$ in the $H^{2N}(\Sn)$ with $N=[n/(2\sigma-1)]$ as in Lemma \ref{lem:cont_dep_u_0}, such that 
\begin{equation*}\label{smallness:v-1}
\|v(T_1,u_0)-1\|_{C^1(\Sn)}<\va.
\end{equation*}
\vskip 8pt

We first  choose $T_2$ large such that if
$t \ge T_2$, $\|v - 1\|_{C^1(\Sn)} \le 1/2$ which is
possible since the latter goes to $0$ as $t \to \infty$. Thus,  it
follows from the expression for $P_\sigma(v-1)$ as in
\eqref{6.29}, that, for $t \geq T_2$ and some positive constant $C_1$ which
depends on $n$ and $\sigma$, the upper bounds of $F_{2N}$ and
$\alpha(t)$,  $\max_{\Sn} f$ as well as the constant we
have found in Lemma \ref{lemma6.6},
\begin{equation}\label{add2} 
\int_{\Sn}|P_\sigma(v-1)|^{N+1} \ud V_{g_{\Sn}}\leq C_1(F_2^{\frac{1}{2}}+\|f_\phi-f(\theta)\|_{L^2(\Sn)}).
\end{equation}
Then it follows from the Sobolev embedding theorem that there exists a positive constant $C_0$ which
only depends on $n$ and $\sigma$, such that

\begin{equation}
\label{add4} \| v - 1\|_{C^1(\Sn)} \le C_0\left(\int_{\Sn}|P_\sigma(v-1)|^{N+1} \ud V_{g_{\Sn}}\right)^{\frac{1}{N+1}}.
\end{equation}

Now we choose $T_3 > T_2$ such that the quantity 
$$|o(1)| <\frac{n-2\sigma}{4}\left( \omega_n \max_{\Sn} f\right)^{\frac{2\sigma-n}{n}}$$
 in \eqref{5.24} for $t > T_3$. 
 
 Consider
 \begin{equation}\label{add6} 
\mathscr{M}(t) :=\mathscr{M}(t,u_0)= F_2(t) +  E_f[u(t,u_0)].
\end{equation}
Then it follows from Proposition \ref{prop:properties} and \eqref{5.24} that  $\frac{\ud \mathscr{M}(t)}{\ud t}<0$ for $t >T_3$.

Using \eqref{intemediate_est}, \eqref{6.38} and the fact that $\lim_{t \to \infty}F_2(t)=0$ by virtue of Lemma \ref{lem:F_p_mid},
for any $\va \in (0,\va_0)$, we can find a bigger $T_4\ge
T_3$ and a positive constant $C_2$ which depends on $n,\sigma$, and $\|f\|_{C^2(\Sn)}$, such that for all $t \ge T_4$, 
\begin{equation}\label{add3}
\| f_\phi -
f(\theta)\|_{L^2(\Sn)}\leq C_2 F_2^{\frac{1}{2}}
\end{equation}
and
$$F_2(t) \leq \frac{1}{(C_1C_3)^2} \left(\frac{\va}{2C_0}\right)^{2(N+1)}.$$
Here $C_3=\max\{1,C_2\}$ and $C_1,C_0$ are given in the inequalities \eqref{add2}  and \eqref{add4}, respectively. Then we define $\delta :=\delta(\va,u_0)= \min\{
E_f[u(t,u_0)], \mathscr{M}(T_4,u_0)\} > 0$.  Since
$ E_f[u(t,u_0)]<E_f[u_0]$ for all $t \geq 0$ in view of Proposition \ref{prop:properties} and
$\lim_{t \to \infty}F_2(t)=0$, there exists a $T_5 \ge T_4 + 1$ such that
$\mathscr{M}(T_5) < \delta$. Hence the set $\{ t; t\ge T_4 + 1 \mathrm{~~and~~}
\mathscr{M}(t) < \delta \}$ is not empty. Finally we select $T_1(u_0) :=
T_1(\va, u_0) = \inf \{ t; t\ge T_4 + 1 \mathrm{~~and~~} \mathscr{M}(t)< 
\delta \} $. We need to check the following two properties: (a)
$T_1(u_0)$ is continuously dependent on $u_0$ in
$H^{2N}(\Sn)$ and (b) $\| v(T_1,u_0) -
1\|_{C^1(\Sn)} < \va$.

To show the continuity of $T_1(u_0)$.  Let $u_0^k\to u_0$ in $H^{2N}(\Sn)$ as $k \to \infty$ and simplify $T_1^k:=T_1(\va,u_0^k)$ and $T_1=T_1(u_0)$, then it follows from Lemma \ref{lem:cont_dep_u_0} that as $k \to \infty, \mathscr{M}_k(t):=\mathscr{M}(t,u_0^k)\to g(t)$ for all $t \in [0,T_4+1]$. This implies that $\mathscr{M}(T_4,u_0^k) \to \mathscr{M}(T_4,u_0)$ and then $ \delta_k:=\delta(\va, u_0^k)\to\delta$ as $k \to \infty$. Notice that, fix every  $t>T_1$, 
\begin{align*}
	\lim_{k\to \infty} \delta_k-\mathscr{M}_k(t)= \delta-\mathscr{M}(t)>0
\end{align*}
again by Lemma \ref{lem:cont_dep_u_0}, then there exists $N_0(t) \in \mathbb{N}$, such that for all $k\geq N_0(t), \delta_k-\mathscr{M}_k(t)>0$. Hence, by definition of $T_1^k$ we have 

$$\limsup_{k\to \infty}T_1^k\leq T_1.$$ 

On the other hand, if $\liminf_{k\to \infty}T_1^k< T_1$, then it follows from the definition of $T_1^k$ that there exists a sequence of real numbers $\{t_m\}_{m=1}^{\infty}$ such that $\liminf_{k\to \infty}T_1^k<t_m< T_1, t_m \nearrow T_1$ and $ \mathscr{M}(t_m)\geq \delta$. Since $\frac{\ud \mathscr{M}(t)}{\ud t}<0$ for all $t \geq T_4$, then $\mathscr{M}(t_m)>\delta$.  However, for each $m$, we can choose all sufficiently large $k$ such that $\mathscr{M}_k(t_m)>\delta_k$. Consequently, we have
$$\liminf_{k\to \infty}T_1^k\geq t_m.$$
Finally, letting $m \to \infty$, we obtain 
$$\liminf_{k\to \infty}T_1^k\geq T_1.$$
This yields a contradiction.

For the assertion (b), it follows from the selection of $T_1(u_0)$ that $T_1>T_4$ and 
$$F_2(T_1) \le \left(\frac{\va}{2C_0}\right)^{2(N+1)}\frac{1}{(C_1C_3)^2}.$$ 
This together with Estimates \eqref{add2}, \eqref{add4} and \eqref{add3} yields
\begin{align*}
 \|v(T_1,u_0) - 1\|_{C^1(\Sn)} \le& C_0 [ C_1 ( F_2(T_1)^{1/2} + \| f_{\phi(T_1)} - f(\theta(T_1))\|_{L^2(\Sn)})]^{\frac{1}{N+1}} \\
 \leq &C_0 [  C_1C_3 F_2(T_1)^{1/2}]^{\frac{1}{N+1}} < \va.
 \end{align*}

\vskip 8pt

\emph{Step 3.} Let $T_1$ be chosen as in \emph{Step 2}, $H(s,u_0)$ is a contraction within $L_{\beta_0}$.

\vskip 8pt

Since $T_1(u_0)$ is continuous in $H^{2N}(\Sn)$ by virtue of \emph{Step 2}, $H(s,u_0)$ is continuous in $[0,1] \times H^{2N}(\Sn)$ and hence  is a contraction within $C_\ast^\infty$. Furthermore, by definition of $H(s,u_0)$, it is not hard to see that
$$E_f[H(s, u_0)] \le \beta_0\quad  \mathrm{if}\quad   s \in \left[0,
\frac{1}{3}\right]\cup \left[\frac{2}{3}, 1\right].$$
This together with \eqref{E_f_below_beta_0} indicates that 
$$E_f[H(s, u_0)] \le \beta_0\quad  \mathrm{for}\quad   s \in [0,1].$$
Notice that our homotopy $H(s,u_0)$ is the one which is homotopic to the constant function $1\in L_{\beta_0}$. Therefore,  $H(s,u_0)$ is indeed a contraction within $L_{\beta_0}$.

\vskip 8pt

In order to prove \textit{(ii)}, given an initial datum $u_0$, we rescale
the time $t$ by letting $\tau(t)$ solve
\begin{equation}\label{7.34}
\frac{\ud\tau}{\ud t}=\min\left\{\frac{1}{2},\epsilon(t,u_0)^2\right\},~~\tau(0)=0.
\end{equation}
Then, we have $\tau(t)\to\infty$ as $t\to\infty$ (see \cite[(6.27)]{Chen&Xu}).
Set $U(\tau,u_0)=u(t(\tau),u_0)$
and $\Gamma(\tau)=\Theta(t(\tau),u_0)$.
It follows from Proposition \ref{prop6.10}
that, for $\epsilon^2<1/2$, the rescaled flow satisfies
(in the stereographic coordinates)
\begin{equation*}
\left(\frac{\ud\Gamma(\tau)}{\ud\tau}\right)^\top=4n^{-1}(n+1)\alpha(\ud f(\theta)+O(\epsilon))
\end{equation*}
and
\begin{equation*}
\frac{\ud}{\ud\tau}(1-|\Gamma(\tau)|^2)=
\frac{32(n+1)}{(n-2)^2}\alpha\epsilon^3(\Delta_{\Sn}f(\Gamma(\tau))+O(1)|\nabla f(\Gamma(\tau))|_{g_{\Sn}}^2+O(\epsilon|\log\epsilon|)),
\end{equation*}
with error $O(1)$ being bounded as $\epsilon\to 0$.
By using the non-increasing energy of the flow (see Proposition \ref{prop:properties} (3)) and asymptotic behaviors of its shadow flow (see Lemma \ref{lem:lim_seq_v}), we apply a very similar argument in the proof of \cite[Proposition 7.1 (ii) on p.484]{Chen&Xu} to conclude that there exists $T>0$ such that
$$u(T, L_{\beta_i-\nu})\subset L_{\beta_{i+1}+\nu}.$$
 Then, for $u\in L_{\beta_i-\nu}\setminus L_{\beta_{i+1}+\nu}$, we define
$$T(u_0):=\inf\{t\geq 0; E_f[u(t,u_0)]\leq \beta_i+\nu\}\leq T.$$
By Lemma \ref{lem:cont_dep_u_0}, $T(u_0)$ continuously depends  on $u_0$. A mapping
$K(s,u_0)=u(sT(u_0),u_0)$ for $0\leq s\leq 1$ if $u\in L_{\beta_i-\nu}\setminus L_{\beta_{i+1}+\nu}$
and $K(s,u_0)=u_0$ if $u_0\in L_{\beta_{i+1}+\nu}$ defines the desired homotopy equivalence between
$L_{\beta_{i+1}+\nu}$ and $L_{\beta_i-\nu}$.
\end{proof}

For the rest parts of Proposition \ref{proposition7.1}, we need some technical lemmas.

\begin{lemma}\label{lem7.3}
There exists two dimensional constants $C_1,C_2>0$ such that if $\|v-1\|_{H^{\sigma}(\Sn)}$ is sufficiently small, then
$$C_1\|v-1\|_{H^{\sigma}(\Sn)}^2\leq \fint_{\Sn}vP_\sigma(v) \ud V_{g_{\Sn}}-R_\sigma\leq C_2\|v-1\|_{H^{\sigma}(\Sn)}^2$$
for all $v\in H^{2\sigma}(\Sn)\cap C_*^\infty$ satisfying (\ref{2.2}).
\end{lemma}
\begin{proof}
As a direct consequence of \eqref{6.31}, we obtain 
$$\fint_{\Sn}vP_\sigma(v) \ud V_{g_{\Sn}}-R_\sigma\leq C_1\|v-1\|_{H^{\sigma}(\Sn)}^2$$
for some constant $C_1>0$. On the other hand, if $\|v-1\|_{H^{\sigma}(\Sn)}$ is sufficiently small,
 by (\ref{6.28}), \eqref{2.4} and \eqref{eq:Lambda_n+2} we have
\begin{equation*}
\begin{split}
&\fint_{\Sn}vP_\sigma(v) \ud V_{g_{\Sn}}-R_\sigma\\
=&\fint_{\Sn}(v-1)P_\sigma(v-1)\ud V_{g_{\Sn}}+2R_\sigma\fint_{\Sn}(v-1) \ud V_{g_{\Sn}}\\
\geq& \fint_{\Sn}(v-1)P_\sigma(v-1)\ud V_{g_{\Sn}}+o(1)\|v-1\|_{H^\sigma(\Sn)}^2\\
&+2R_\sigma \left[\frac{n-2\sigma}{2n}\fint_{\Sn} (v^{\frac{2n}{n-2\sigma}}-1)\ud V_{g_{\Sn}}-\frac{1}{2}\frac{n+2\sigma}{n-2\sigma}\fint_{\Sn}(v-1)^2 \ud V_{g_{\Sn}}\right]\\
\geq& \sum_{i=n+2}^\infty \left(\Lambda_i-R_\sigma\frac{n+2\sigma}{n-2\sigma}\right)|v^i|^2+o(1)\|v-1\|_{H^{\sigma}(\Sn)}^2\\
\geq& \left(\Lambda_{n+2}-R_\sigma\frac{n+2\sigma}{n-2\sigma}+o(1)\right)\|v-1\|_{H^{\sigma}(\Sn)}^2\geq \frac{2\sigma}{2+n-2\sigma} \|v-1\|_{H^{\sigma}(\Sn)}^2.
\end{split}
\end{equation*}
This proves the assertion. 
\end{proof}

For $r_0>0$ and each critical point $p_i \in \Sn$ of $f$, we define
\begin{equation*}
\begin{split}
B_{r_0}(p_i)&=\Big\{ u\in C_*^\infty; g=u^{\frac{4}{n-2\sigma}}g_{\Sn}\mathrm{~~induces~~a~~normalized~~metric }\\
&\hspace{8mm} h=\phi^\ast g=v^{\frac{4}{n-2\sigma}}g_{\Sn}\mathrm{~~with~~} \phi=\phi_{-p,\epsilon} \mathrm{~~for~~some~~ } p\in\Sn\mathrm{~~and~~}\\
&\hspace{8mm} 0<\epsilon\leq 1\mathrm{~~such~~that~~}\|v-1\|_{H^{\sigma}(\Sn)}^2+|p-p_i|^2+\epsilon^2<r_0^2\Big\}.
\end{split}
\end{equation*}
As shown in \cite{Chen&Xu}, the new coordinates $(\epsilon, p, v)$ are corresponding to each $u\in B_{r_0}(p_i)$. 
Under the assumption on $f$, by the Morse lemma, we can introduce 
the local coordinates $p=p^++p^-$ near $p_i=0$, such that 
$$f(p)=f(p_i)+|p^+|^2-|p^-|^2.$$

\begin{lemma}\label{lem7.4}
For $r_0>0$ and $u=(\epsilon, p, v)\in B_{r_0}(p_i)$, with $o(1)\to 0$
as $r_0\to0$:\\
(a) There holds
\begin{equation}\label{7.29}
\fint_{\Sn}f\circ \phi_{-p,\epsilon}\ud V_h
=f(p)+\frac{2}{n-2}\epsilon^2\Delta_{g_{\Sn}}f(p)+O(\epsilon^3|\log\epsilon|)
+o(\epsilon)\|v-1\|_{H^\sigma(\Sn)}.
\end{equation}
(b) There holds
\begin{align}\label{7.30}
&\left|\frac{\partial}{\partial\epsilon}E_f[u]+\frac{4(n-2\sigma)}{n(n-2)}\left(\fint_{\Sn}uP_\sigma(u)\ud V_{g_{\Sn}}\right)
\epsilon\, \omega_n^{\frac{2\sigma}{n}}f(p)^{\frac{2(\sigma-n)}{n}}\Delta_{g_{\Sn}}f(p)
\right|\no\\
\leq& C\epsilon^2|\log\epsilon|+C(\epsilon+|p-p_i|)\|v-1\|_{H^\sigma(\Sn)}.
\end{align}
In particular, if $\Delta_{g_{\Sn}}f(p)>0$, then
\begin{align}\label{7.31}
\frac{\partial}{\partial\epsilon}E_f[u]
&\leq-\frac{4(n-2\sigma)}{n(n-2)}R_\sigma
\epsilon\, \omega_n^{\frac{2\sigma}{n}} f(p)^{\frac{2(\sigma-n)}{n}}\Delta_{g_{\Sn}}f(p)\no
\\
&+C\epsilon^2|\log\epsilon|+C(\epsilon+|p-p_i|)\|v-1\|_{H^1(\Sn)}.
\end{align}
(c) For any $q\in T_p(\Sn)$, there holds
\begin{align}\label{7.32}
&\left|\frac{\partial}{\partial p}E_f[u]\cdot q+\frac{n-2\sigma}{n}\left(\fint_{\Sn}vP_\sigma(v)\ud V_{g_{\Sn}}\right)
\omega_n^{\frac{2\sigma}{n}}f(p)^{\frac{2(\sigma-n)}{n}}\ud f(p)\cdot q
\right|\no\\
\leq& C\epsilon(\epsilon+\|v-1\|_{H^\sigma(\Sn)})|q|.
\end{align}
(d) There exists a uniform constant $C_0>0$ such that
\begin{equation}\label{7.33}
\left\langle\frac{\partial}{\partial v}E_f[u],v-1\right\rangle\geq
C_0\|v-1\|_{H^\sigma(\Sn)}^2+o(\epsilon) \|v-1\|_{H^\sigma(\Sn)},
\end{equation}
where $\langle\cdot, \cdot\rangle$ denotes the duality of pairing
of $H^\sigma(\Sn)$ with its dual.
\end{lemma}
\begin{proof}
For notational convenience, let
$$A=A(u)=\fint_{\Sn}f\circ \phi_{-p,\epsilon}\ud V_h.$$

\textit{(a)} Observe that
$$A-f(p)
=\fint_{\Sn}(f\circ \phi_{-p,\epsilon}-f(p))\ud V_{g_{\Sn}}+I,$$
where the error term $I$ is given by
$$I=\fint_{\Sn}(f\circ \phi_{-p,\epsilon}-f(p))(v^{\frac{2n}{n-2\sigma}}-1)\ud V_{g_{\Sn}}$$
which can be estimated as follows:
\begin{equation*}
\begin{split}
|I|&\leq \|f\circ \phi_{-p,\epsilon}-f(p)\|_{L^2(\Sn)}
\|v^{\frac{2n}{n-2\sigma}}-1\|_{L^2(\Sn)}\\
&\leq o(\epsilon)\|v-1\|_{H^\sigma(\Sn)}
\end{split}
\end{equation*}
in view of (\ref{intemediate_est})
and $|\ud f(p)|\to 0$ as $r_0\to 0$ since $p_i$ is a critical point of $f$.
Now we can follow the proof of \cite[Lemma 7.3 (a)]{Chen&Xu} to conclude that
\begin{equation*}
\fint_{\Sn}(f\circ\phi_{-p,\epsilon}-f(p))\ud V_{g_{\Sn}}
=\frac{2}{n-2}\epsilon^2\Delta_{g_{\Sn}}f(p)+O(\epsilon^3|\log\epsilon|).
\end{equation*}
Thus, \textit{(a)} follows from combining all these.

\textit{(b)} By (\ref{1.1}), (\ref{1.9}) and (\ref{2.4}), we have
$$\omega_n^{-\frac{2\sigma}{n}} E_f[u]=\frac{\fint_{\Sn}uP_\sigma(u)\ud V_{g_{\Sn}}}{(\fint_{\Sn}f u^{\frac{2n}{n-2\sigma}}\ud V_{g_{\Sn}})^{\frac{n-2\sigma}{n}}}
=\frac{\fint_{\Sn}vP_\sigma(v)\ud V_{g_{\Sn}}}{(\fint_{\Sn}f\circ\phi_{-p,\epsilon} v^{\frac{2n}{n-2\sigma}}\ud V_{g_{\Sn}})^{\frac{n-2\sigma}{n}}}.$$
Then
$$\omega_n^{-\frac{2\sigma}{n}} \frac{\partial}{\partial\epsilon}E_f[u]
=-\frac{n-2\sigma}{n}\left(\fint_{\Sn}vP_\sigma(v)\ud V_{g_{\Sn}}\right)
A^{-\frac{2(n-\sigma)}{n}}
\frac{\partial}{\partial\epsilon}\fint_{\Sn}f\circ\phi_{-p,\epsilon} v^{\frac{2n}{n-2\sigma}}\ud V_{g_{\Sn}}.$$
As in the proof of \cite[Lemma 7.3 (b)]{Chen&Xu},
we can introduce the stereographic coordinates and denote $\phi_{-p,\epsilon}$ by $\psi_\epsilon$
to get
\begin{equation*}
\begin{split}
\frac{\partial}{\partial\epsilon}\fint_{\Sn}f\circ\phi_{-p,\epsilon} v^{\frac{2n}{n-2\sigma}}\ud V_{g_{\Sn}}
&=\omega_n^{-1}\int_{\mathbb{R}^n}\frac{\partial}{\partial\epsilon}f(\psi_\epsilon(z))\left(\frac{2}{1+|z|^2}\right)^n\ud z\\
&\hspace{4mm}
+\omega_n^{-1}\int_{\mathbb{R}^n}\frac{\partial}{\partial\epsilon}f(\psi_\epsilon(z))(v^{\frac{2n}{n-2\sigma}}-1)\left(\frac{2}{1+|z|^2}\right)^n\ud z\\
&:=I+II.
\end{split}
\end{equation*}
Then, by the same argument in the proof of \cite[Lemma 7.3 (b)]{Chen&Xu}, we have
\begin{equation*}
I=\frac{4}{n-2}\epsilon \Delta_{g_{\Sn}}f(p)+O(\epsilon^2|\log\epsilon|).
\end{equation*}
On the other hand, as in the proof of  \cite[Lemma 7.3 (b)]{Chen&Xu}, we have
\begin{align*}
\left|\frac{\partial}{\partial\epsilon}f(\psi_\epsilon(z))\right|\leq C|z| \quad\mathrm{for}\quad z \in \mathbb{R}^n
\end{align*}
and
$$
\left|\frac{\partial}{\partial\epsilon}f(\psi_\epsilon(z))\right|\leq C|p-p_i||z|+C\epsilon|z|^2 \quad \mathrm{~~in~~}B_{\epsilon^{-1}}(0).
$$
Then
\begin{equation*}
\begin{split}
|II|&\leq C\left[(\epsilon+|p-p_i|)\int_{B_{\epsilon^{-1}}(0)}\frac{|v^{\frac{2n}{n-2\sigma}}-1|}{(1+|z|^2)^{n-1}}\ud z
+\int_{\mathbb{R}^n\setminus B_{\epsilon^{-1}}(0)}\frac{|v^{\frac{2n}{n-2\sigma}}-1||z|}{(1+|z|^2)^n}\ud z\right]\\
&\leq C(\epsilon+|p-p_i|)\|v-1\|_{H^\sigma(\Sn)}.
\end{split}
\end{equation*}
Thus, (\ref{7.30}) follows from combining all these and \textit{(a)}.
Moreover, (\ref{eq:coercive-1}) together with $u \in C_\ast^\infty$ gives
$$\fint_{\Sn}uP_\sigma(u)\ud V_{g_{\Sn}}
=\fint_{\Sn}vP_\sigma(v)\ud V_{g_{\Sn}}\geq R_\sigma.$$
Hence, if $\Delta_{g_{\Sn}}f(p)>0$, then (\ref{7.31}) follows from (\ref{7.30}).

\textit{(c)} For any $q\in T_p(\Sn)$, we have
\begin{equation*}
\begin{split}
&\omega_n^{-\frac{2\sigma}{n}}\frac{\partial}{\partial p}E_f[u]\cdot q+\frac{n-2\sigma}{n}\left(\fint_{\Sn}vP_\sigma(v)\ud V_{g_{\Sn}}\right)
A^{\frac{2(\sigma-n)}{n}} \ud f(p)\cdot q\\
&=\frac{2\sigma-n}{n}\left(\fint_{\Sn}vP_\sigma(v)\ud V_{g_{\Sn}}\right)A^{\frac{2(\sigma-n)}{n}}
\left[\fint_{\Sn}(\ud(f\circ\phi_{-p,\epsilon})\cdot q-\ud f(p)\cdot q )\ud V_{g_{\Sn}}\right.\\
&\hspace{8mm}\left.+\fint_{\Sn}(\ud(f\circ\phi_{-p,\epsilon})\cdot q-\ud f(p)\cdot q )(v^{\frac{2n}{n-2\sigma}}-1)\ud V_{g_{\Sn}}\right]\\
&=\frac{2\sigma-n}{n}\left(\fint_{\Sn}vP_\sigma(v)\ud V_{g_{\Sn}}\right)A^{\frac{2(\sigma-n)}{n}} (I_1+I_2).
\end{split}
\end{equation*}
Similar to the proof of \cite[Lemma 7.3 (c)]{Chen&Xu}), we have
$$|I_1|\leq C\epsilon^2 |q|\quad \mathrm{ and }\quad |I_2|\leq C\epsilon\|v-1\|_{H^\sigma(\Sn)}|q|.$$
Now (\ref{7.32}) follows from these.

\textit{(d)} A direct computation yields
\begin{equation*}
\begin{split}
& \omega_n^{-\frac{2\sigma}{n}}\left\langle\frac{\partial}{\partial v}E_f[u],v-1\right\rangle\\
=&2A^{\frac{2\sigma-n}{n}}\Bigg[\fint_{\Sn}(v-1)P_\sigma(v)\ud V_{\Sn}
-\left(\fint_{\Sn} vP_\sigma(v)\ud V_{\Sn}\right)A^{-1}\fint_{\Sn} f\circ\phi_{-p,\epsilon}v^{\frac{n+2\sigma}{n-2\sigma}}(v-1)\ud V_{\Sn}\Bigg]\\
=&2A^{\frac{2\sigma-n}{n}}\Bigg[\fint_{\Sn} (v-1)P_\sigma(v-1)\ud V_{\Sn}
-R_\sigma\fint_{\Sn} (v^{\frac{n+2\sigma}{n-2\sigma}}-1)(v-1)\ud V_{\Sn}+I\Bigg],
\end{split}
\end{equation*}
where
$$I=-\fint_{\Sn}\left[\left(\fint_{\Sn} vP_\sigma(v)\ud V_{\Sn}\right)A^{-1}f\circ\phi_{-p,\epsilon}-R_\sigma\right] v^{\frac{n+2\sigma}{n-2\sigma}}(v-1)\ud V_{\Sn}.$$
By (\ref{6.28}) and (\ref{6.32}), we compute
\begin{equation}\label{7.38}
\begin{split}
&\fint_{\Sn} (v-1)P_\sigma(v-1)\ud V_{\Sn}-R_\sigma\fint_{\Sn} (v^{\frac{n+2\sigma}{n-2\sigma}}-1)(v-1)\ud V_{\Sn}\\
=&\fint_{\Sn} (v-1)P_\sigma(v-1)\ud V_{\Sn}-\frac{n+2\sigma}{n-2\sigma}R_\sigma\fint_{\Sn} (v-1)^2\ud V_{\Sn}
+o(1)\|v-1\|^2_{H^\sigma(\Sn)}\\
=&\sum_{i=0}^\infty\left(\Lambda_i-\frac{n+2\sigma}{n-2\sigma}R_\sigma\right)|v^i|^2+o(1)\|v-1\|^2_{H^\sigma(\Sn)}\\
\geq&\frac{\Lambda_{n+2}-\frac{n+2\sigma}{n-2\sigma}R_\sigma}{\Lambda_{n+2}+1}\sum_{i\geq n+2}^\infty(\Lambda_i+1)|v^i|^2+o(1)\|v-1\|^2_{H^\sigma(\Sn)}\\
\geq& C_0\|v-1\|^2_{H^\sigma(\Sn)}.
\end{split}
\end{equation}
Moreover, we can decompose
\begin{equation*}
\begin{split}
I&=-\left(\fint_{\Sn} vP_\sigma(v)\ud V_{\Sn}-R_\sigma\right)A^{-1}\fint_{\Sn}f\circ\phi_{-p,\epsilon} v^{\frac{n+2\sigma}{n-2\sigma}}(v-1)\ud V_{\Sn}\\
&\hspace{4mm}
-R_\sigma A^{-1}\fint_{\Sn}(f\circ\phi_{-p,\epsilon}-A) v^{\frac{n+2\sigma}{n-2\sigma}}(v-1)\ud V_{\Sn}\\
&=I_1+I_2.
\end{split}
\end{equation*}
By Lemma \ref{lem7.3}, we have
$$|I_1|\leq C\|v-1\|^3_{H^\sigma(\Sn)}=o(1)\|v-1\|^2_{H^\sigma(\Sn)}.$$
On the other hand, using (\ref{intemediate_est}),
(\ref{7.29})
and the fact that $$|\ud f(p)|\to 0~~\mbox{ as }r_0\to 0,$$
we can estimate
\begin{equation*}
\begin{split}
|I_2|&\leq C\left|\fint_{\Sn}(f\circ\phi_{-p,\epsilon}-f(p)) v^{\frac{n+2\sigma}{n-2\sigma}}(v-1)\ud V_{\Sn}\right|\\
&\hspace{4mm}+C\left|\fint_{\Sn}(f(p)-A) v^{\frac{n+2\sigma}{n-2\sigma}}(v-1)\ud V_{\Sn}\right|\\
&\leq C(\|f\circ\phi_{-p,\epsilon}-f(p)\|_{L^2(\Sn)}+|f(p)-A|)\|v-1\|_{H^\sigma(\Sn)}\\
&\leq (o(\epsilon)+C\epsilon^2)\|v-1\|_{H^\sigma(\Sn)}=o(\epsilon) \|v-1\|_{H^\sigma(\Sn)}.
\end{split}
\end{equation*}
Now (\ref{7.33}) follows by collecting all these.
\end{proof}

We are in a position to complete the proof of 
parts \textit{(iii)} and \textit{(iv)} in Proposition \ref{proposition7.1}.

\begin{proof}[Proof of Proposition \ref{proposition7.1} (iii) and (iv)]
As in the proof of part \textit{(ii)}, we choose $\nu\leq r_0^3<\nu_0$ and $r_0>0$ sufficiently
small such that
$B_{r_0}(p_i)\subset L_{\beta_i+\nu}\setminus L_{\beta_i-\nu}$.
As in \textit{(ii)}, for any $1\leq i\leq N$ , we show that there exists a sufficiently large $T>0$ such that
$u(T,L_{\beta_i+\nu_0})\subset L_{\beta_i+\nu}$.
In addition, for any $u_0\in L_{\beta_i+\nu}$,
we can choose a larger $T=T(u_0)>0$ such that
either $u(T,u_0)\in L_{\beta_i-\nu}$ or
$u(t,u_0)\in B_{r_0/4}(p_i)$ for some $t\in [0,T]$.

For $u=(\epsilon, p, v)\in B_{r_0}(p_i)$, we have
\begin{equation}\label{7.37}
\begin{split}
&\omega_n^{-\frac{2\sigma}{n}}(E_f[u]-\beta_i)\\
=&\frac{\fint_{\Sn}uP_\sigma(u)\ud V_{g_{\Sn}}}{(\fint_{\Sn}f u^{\frac{2n}{n-2\sigma}}\ud V_{g_{\Sn}})^{\frac{n-2\sigma}{n}}}
-R_\sigma f(p_i)^{\frac{2\sigma-n}{n}}\\
=&\frac{\fint_{\Sn}vP_\sigma(v)\ud V_{g_{\Sn}}}{(\fint_{\Sn}f\circ\phi_{-p,\epsilon} v^{\frac{2n}{n-2\sigma}}\ud V_{g_{\Sn}})^{\frac{n-2\sigma}{n}}}
-R_\sigma f(p_i)^{\frac{2\sigma-n}{n}}\\
=&A^{\frac{2\sigma-n}{n}}\left[\left(\fint_{\Sn}vP_\sigma(v)\ud V_{g_{\Sn}}-R_\sigma\right)
-R_\sigma f(p_i)^{\frac{2\sigma-n}{n}}(A^{\frac{n-2\sigma}{n}}-f(p_i)^{\frac{n-2\sigma}{n}})\right],
\end{split}
\end{equation}
where $A=\displaystyle\fint_{\Sn}f\circ\phi_{-p,\epsilon} v^{\frac{2n}{n-2\sigma}}\ud V_{g_{\Sn}}$.
Observe that
\begin{equation*}
\begin{split}
A^{\frac{n-2\sigma}{n}}-f(p_i)^{\frac{n-2\sigma}{n}}&=
f(p_i)^{\frac{n-2\sigma}{n}}\Bigg[\Bigg(1+\frac{A-f(p_i)}{f(p_i)}\Bigg)^{\frac{n-2\sigma}{n}}-1\Bigg]\\
&=f(p_i)^{\frac{n-2\sigma}{n}}\Bigg[\frac{n-2\sigma}{n}\frac{A-f(p_i)}{f(p_i)}+O(|A-f(p_i)|^2)\Bigg]
\end{split}
\end{equation*}
and
$$A-f(p_i)=A-f(p)+f(p)-f(p_i).$$
Combining these with Lemma \ref{lem7.4} \textit{(a)}, we find
\begin{equation}\label{7.35}
\begin{split}
&f(p_i)^{\frac{2\sigma}{n}}(A^{\frac{n-2\sigma}{n}}-f(p_i)^{\frac{n-2\sigma}{n}})\\
=&\frac{2(n-2\sigma)}{n(n-2)}\epsilon^2\Delta_{g_{\Sn}}f(p)+\frac{n-2\sigma}{n}(|p^+|^2-|p^-|^2)\\
&+o(1)(\epsilon^2+|p-p_i|^2+\|v-1\|^2_{H^\sigma(\Sn)}).
\end{split}
\end{equation}
Hence, from Lemma \ref{lem7.3}, we can conclude that
$$E_f[u]-\beta_i\geq C_2\|v-1\|_{H^\sigma(\Sn)}^2-C(\epsilon^2+|p-p_i|^2).$$
Consequently, for $u\in L_{\beta_i+\nu}\cap B_{r_0}(p_i)$, we have
\begin{equation}\label{7.36}
\|v-1\|^2_{H^\sigma(\Sn)}\leq C(\epsilon^2+|p-p_i|^2+r_0^3).
\end{equation}

We still use the same normalization (\ref{7.34}) in $t$ used in the proof of part \textit{(ii)}.
With this time scale, it follows from Proposition \ref{prop:properties} (3),
Lemmas \ref{lemma5.5} and \ref{lemma6.7} that
\begin{equation*}
\begin{split}
\frac{\ud}{\ud\tau}E_f[U(\tau,u_0)]
&=\epsilon^{-2}\frac{\ud}{\ud t}E_f[u(t(\tau),u_0)]\\
&\leq -C_3(|\nabla f(p)|_{g_{\Sn}}^2+\epsilon^2|\Delta_{g_{\Sn}}f(p)|^2)\\
&\leq -C_4(\epsilon^2+|p-p_i|^2)
\end{split}
\end{equation*}
with uniform constants $C_3>0$, $C_4>0$ for all $u_0\in B_{r_0}(p_i)$.
We would like to explain how to get this estimate.
With the coordinates we chose, there holds $|\nabla f(p)|_{g_{\Sn}}^2=|p-p_i|^2$.
Also the non-degeneracy condition implies that $|\Delta_{g_{\Sn}} f(p)|>0$
if $r_0$ is sufficiently small since $p_i$ is a critical point of $f$.
From these, it is not hard to see the above estimate holds.

Thus, for $u_0\in B_{r_0}(p_i)\setminus B_{r_0/4}(p_i)$, we have
$$\frac{\ud}{\ud\tau}E_f[U(\tau,u_0)]\leq -C_5r_0^2,$$
with a uniform constant $C_5>0$ in view of (\ref{7.36}).
Hence, the transversal time of the annular region $L_{\beta_i+\nu}\cap (B_{r_0/2}(p_i)\setminus B_{r_0/4}(p_i))$
is uniformly positive. Choosing sufficiently large $T^*>0$ and sufficiently small $\nu>0$, we have
$$U(T^*,L_{\beta_i+\nu})\subset L_{\beta_i-\nu}\cup (B_{r_0/2}(p_i)\cap L_{\beta_i+\nu}).$$
Then
$$T_\nu(u_0)=\min\big\{T^*,\inf\{t; E_f[U(t,u_0)]\leq\beta_i-\nu\}\big\}$$
continuously depends on $u_0$. Thus the map
$(t,u_0)\mapsto U(\min\{t,T_\nu(u_0)\},u_0)$ gives a homotopy equivalence of $L_{\beta_i+\nu}$
with a subset of $L_{\beta_i-\nu}\cup (B_{r_0/2}(p_i)\cap L_{\beta_i+\nu})$.

For part \textit{(iii)}, with the help of Lemma \ref{lem7.4} \textit{(a)} and \textit{(b)}, a similar argument as in the proof  of \cite[Proposition 7.1 (iii) on pp.493-494]{Chen&Xu} goes through.

For part \textit{(iv)}, assume $\Delta_{g_{\Sn}}f(p_i)<0$.
By (\ref{7.37}), (\ref{7.35}) and  Lemma \ref{lem7.3}, we have
\begin{equation*}
\begin{split}
&\omega_n^{-\frac{2\sigma}{n}}(E_f[u]-\beta_i)\\
\geq& A^{\frac{2\sigma-n}{n}}\Bigg[C_2\|v-1\|^2_{H^\sigma(\Sn)}
-\frac{2(n-2\sigma)}{n(n-2)}R_\sigma\epsilon^2\Delta_{g_{\Sn}}f(p)\\
&\qquad\quad~~+\frac{n-2\sigma}{n}R_\sigma(|p^-|^2-|p^+|^2)+o(1)(\epsilon^2+|p-p_i|^2+\|v-1\|^2_{H^\sigma(\Sn)})\Bigg]
\end{split}
\end{equation*}
with $o(1)\to 0$ as $r_0\to 0$.
We can then deduce that there exists some number $\delta>0$
with $4\delta^2\leq  \frac{7}{9}\min\{1,r_0^2\}$ such that
\begin{equation}\label{est:negative_Laplace}
\epsilon^2+|p^-|^2+\|v-1\|^2_{H^\sigma(\Sn)}\leq \frac{r_0^2}{4}
\end{equation}
for any $u=(\epsilon, p, v)\in B_{r_0}(p_i)\cap L_{\beta_i+\nu}$ with $|p^+|<2\delta r_0$, provided
that $r_0>0$ is sufficiently small and $\nu\leq r_0^3$.

Let $a_+=\max\{a,0\}$ for $a\in\mathbb{R}$. Let  $\eta$ be a cut-off function defined by
$$\eta=\eta(|p^+|)=\displaystyle\left(1-\frac{(|p^+|-\delta r_0)_+}{\delta r_0}\right)_+$$
with $\delta>0$ given as above. For $0\leq r\leq 1$ and $u=(\epsilon, p, v)\in B_{r_0}(p_i)$,
we choose $\epsilon_0>0$ sufficiently small such that $\displaystyle 0<\epsilon/3<\epsilon_0<2\epsilon/3$
and define $u_r$ by
$$u_r=(\epsilon_r, p_r, v_r)
=\big(\epsilon+(\epsilon_0-\epsilon)r\eta, p-r\eta p^-,((1-r\eta)v^{\frac{2n}{n-2\sigma}}+r\eta)^{\frac{n-2\sigma}{2n}}\big).$$

\begin{claim}
If $\|v-1\|_{C^1(\Sn)}$ is sufficiently small, then $u_r\in B_{r_0}(p_i)$.
\end{claim}

To that end, we first consider $\eta=1$ and define
$$ H(r) = ( \epsilon + (\epsilon_0 - \epsilon) r)^2 + |p - r p^-|^2 + \int_{\Sn} (v_r-1)P_\sigma(v_r-1) \ud V_{g_{\Sn}}.$$
A direct computation yields
\begin{align*}
H'(r)= 2(\epsilon + (\epsilon_0 - \epsilon) r)(\epsilon_0 - \epsilon) + 2\langle p - rp^-, -p^-\rangle+2\int_{\Sn}v_r'P_\sigma(v_r-1)\ud V_{g_{\Sn}}.
\end{align*}
Notice that
$$v_r'=\frac{n-2\sigma}{2n}v_r^{-\frac{n+2\sigma}{n-2\sigma}}(1-v^{\frac{2n}{n-2\sigma}}), \quad v_r''=-\frac{n+2\sigma}{n-2\sigma}v_r^{-1}(v_r')^2.$$
Then we have
\begin{align*}
\frac{1}{2}H''(r)=&(\epsilon_0 - \epsilon)^2 + |p^-|^2+\int_{\Sn}v_r'P_\sigma(v_r')\ud V_{g_{\Sn}}+\int_{\Sn}v_r''P_\sigma(v_r-1)\ud V_{g_{\Sn}} \\
=&(\epsilon_0 - \epsilon)^2 + |p^-|^2+\int_{\Sn}v_r'P_\sigma(v_r')\ud V_{g_{\Sn}}+o(1) \|v_r'\|_{H^\sigma(\Sn)}^2\geq 0,
\end{align*}
where the second identity follows from
\begin{align*}
\int_{\Sn}v_r''P_\sigma(v_r-1)\ud V_{g_{\Sn}}=&-\frac{n+2\sigma}{n-2\sigma}\int_{\Sn}v_r^{-1}(v_r')^2P_\sigma(v_r-1)\ud V_{g_{\Sn}}\\
=&o(1)\|v_r'\|_{H^\sigma(\Sn)}^2
\end{align*}
through a similar argument to derive \eqref{est:2nd_diff}.

Thus, we conclude that
$$H(r) \leq \max\{H(0),H(1)\}\leq \max\{r_0^2,\epsilon_0^2+|p^+|^2\}.$$
Next, if $\eta r>0$, then  $|p^+|<2\delta r_0$. From this and \eqref{est:negative_Laplace}, we have $\epsilon^2\leq r_0^2/4$ and then $\epsilon_0^2 \leq r_0^2/9$. By the selection of $\delta$, there holds $g(1)\leq \epsilon_0^2+|p^+|^2 \leq (1/9+4\delta^2)<r_0^2$. Hence, $H(\eta r)\leq H(1)\leq r_0^2$.

Hence, under the condition of $\|v-1\|_{C^1(\Sn)}$ being small (which can be
guaranteed by the construction of homotopies below),
we have established that
the homotopy $H_1: \overline{B_{r_0}(p_i)}\cap L_{\beta_i+\nu}\times [0,1]\to \overline{B_{r_0}(p_i)}$
given by $H_1(u,r)=u_r$ is well defined
and $H_1(\cdot, 1)$ maps the set
$\{u\in B_{r_0}(p_i)\cap L_{\beta_i+\nu}; |p^+|<\delta r_0\}$
to a set $B_{\delta r_0}^+$, where
$$B_{\rho}^+:=\{u\in B_{r_0}(p_i); \epsilon=\epsilon_0, p^-=0, |p^+|<\rho, v=1\} \quad\mathrm{for~~} 0<\rho<r_0$$
is clearly diffeomorphic to a unit ball of dimension $n-\mbox{ind}(f,p_i)$.

Now we need to show that the energy of $u_r$ satisfies
$E_f[u_r]\leq\beta_i+\nu$ if $\nu$ is sufficiently small.
To that end, we compute
\begin{equation*}
\begin{split}
\frac{\ud}{\ud r}E_f[u_r]
=&\eta\Bigg(\frac{\partial E_f[u_r]}{\partial\epsilon_r}(\epsilon_0-\epsilon)
-\frac{\partial E_f[u_r]}{\partial p_r} p^{-}\\
&-\frac{n-2\sigma}{2n}\left\langle \frac{\partial E_f[u_r]}{\partial v_r},
v_r^{-\frac{n+2\sigma}{n-2\sigma}}(v^{\frac{2n}{n-2\sigma}}-1)\right\rangle\Bigg)\\
=&\eta(1-r\eta)^{-1}\Bigg(\frac{\partial E_f[u_r]}{\partial\epsilon_r}(\epsilon_0-\epsilon_r)
-\frac{\partial E_f[u_r]}{\partial p_r} \cdot p_r^{-}\\
&-\frac{n-2\sigma}{2n}\left\langle \frac{\partial E_f[u_r]}{\partial v_r},
v_r^{-\frac{n+2\sigma}{n-2\sigma}}(v_r^{\frac{2n}{n-2\sigma}}-1)\right\rangle\Bigg)\\
:=&\eta(1-r\eta)^{-1}D:=\eta(1-r\eta)^{-1}(I-II-III).
\end{split}
\end{equation*}
The last term can be rewritten as
\begin{align*}
III&=\frac{n-2\sigma}{2n}\left\langle \frac{\partial E_f[u_r]}{\partial v_r},
v_r^{-\frac{n+2\sigma}{n-2\sigma}}(v_r^{\frac{2n}{n-2\sigma}}-1)\right\rangle\\
&=\frac{n-2\sigma}{n}\omega_n^{\frac{2\sigma}{n}} A_r^{\frac{2\sigma-n}{n}}
\fint_{\Sn}v_r^{-\frac{n+2\sigma}{n-2\sigma}}(v_r^{\frac{2n}{n-2\sigma}}-1)P_\sigma(v_r)\ud V_{g_{\Sn}}\\
&\hspace{4mm}-\frac{n-2\sigma}{n}\omega_n^{\frac{2\sigma}{n}} A_r^{\frac{2\sigma-n}{n}-1}\left(\fint_{\Sn}v_rP_\sigma(v_r)\ud V_{g_{\Sn}}\right)
\fint_{\Sn}f\circ \phi_{-p_r,\epsilon_r}(v_r^{\frac{2n}{n-2\sigma}}-1)\ud V_{g_{\Sn}}\\
&:=\frac{n-2\sigma}{n}\omega_n^{\frac{2\sigma}{n}} A_r^{\frac{2\sigma-n}{n}}(III_1+III_2),
\end{align*}
where $A_r=\displaystyle\fint_{\Sn}f\circ \phi_{-p_r,\epsilon_r}v_r^{\frac{2n}{n-2\sigma}}\ud V_{g_{\Sn}}$.
Notice that  $\|v_r-1\|_{C^1(\Sn)}=o(1)$ since $\|v-1\|_{C^1(\Sn)}=o(1)$.
We can estimate $III_1$ as follows:
\begin{align*}
III_1=&\fint_{\Sn}v_r^{-\frac{n+2\sigma}{n-2\sigma}}(v_r^{\frac{2n}{n-2\sigma}}-1)P_\sigma(v_r)\ud V_{g_{\Sn}}\\
=&\frac{2n}{n-2\sigma}\left(\fint_{\Sn} (v_r-1)P_\sigma(v_r-1)\ud V_{\Sn}-\frac{n+2\sigma}{n-2\sigma}R_\sigma\fint_{\Sn} (v_r-1)^2\ud V_{\Sn}\right)\\
&+o(1)\|v_r-1\|_{H^\sigma(\Sn)}^2\\
\geq&(C_0+o(1))\|v_r-1\|_{H^\sigma(\Sn)}^2
\end{align*}
for some constant $C_0>0$ depending only on $n,\sigma$, where the last inequality follows from (\ref{7.38}).
We can rewrite $III_2$ as
\begin{equation*}
III_2=-\left(\fint_{\Sn}v_rP_\sigma(v_r)\ud V_{g_{\Sn}}\right)
\left(1-A_r^{-1}\fint_{\Sn}f\circ \phi_{-p_r,\epsilon_r}\ud V_{g_{\Sn}}\right).
\end{equation*}
Note that it follows from the proof of Lemma \ref{lem7.4} \textit{(a)} that
$$\fint_{\Sn}f\circ \phi_{-p_r,\epsilon_r}\ud V_{g_{\Sn}}-f(p_r)
=(A_r-f(p_r))+o(\epsilon_r)\|v_r-1\|_{H^\sigma(\Sn)}.$$
Hence,
\begin{equation*}
\begin{split}
1-A_r^{-1}\fint_{\Sn}f\circ \phi_{-p_r,\epsilon_r}\ud V_{g_{\Sn}}
&=A_r^{-1}\Bigg[A_r-f(p_r)-\left(\fint_{\Sn}f\circ \phi_{-p_r,\epsilon_r}\ud V_{g_{\Sn}}-f(p_r)\right)\Bigg]\\
&=o(\epsilon_r)\|v_r-1\|_{H^\sigma(\Sn)},
\end{split}
\end{equation*}
which implies that
$$III_2=-o(\epsilon_r)\|v_r-1\|_{H^\sigma(\Sn)}\left(\fint_{\Sn}v_rP_\sigma(v_r)\ud V_{g_{\Sn}}\right).$$
Therefore, we obtain
\begin{align*}
III\geq&\frac{n-2\sigma}{n}\omega_n^{\frac{2\sigma}{n}} A_r^{\frac{2\sigma-n}{n}}C_0\|v_r-1\|_{H^\sigma(\Sn)}^2\\
&+o(1)(\epsilon_r+\|v_r-1\|_{H^\sigma(\Sn)})\|v_r-1\|_{H^\sigma(\Sn)}.
\end{align*}
By Lemma \ref{lem7.4}, we have
\begin{align*}
I=&-\frac{4(n-2\sigma)}{n(n-2)}\left(\fint_{\Sn}v_rP_\sigma(v_r)\ud V_{g_{\Sn}}\right)
\epsilon_r\, \omega_n^{\frac{2\sigma}{n}}A_r^{\frac{2(\sigma-n)}{n}}\Delta_{g_{\Sn}}f(p_r) (\epsilon_0-\epsilon_r)\\
&+C\Big[\epsilon_r^2|\log\epsilon_r|+(\epsilon_r+|p_r-p_i|)\|v_r-1\|_{H^\sigma(\Sn)}\Big](\epsilon_r-\epsilon_0)
\end{align*}
and
\begin{align*}
II=&-\frac{n-2\sigma}{n}\left(\fint_{\Sn}v_rP_\sigma(v_r)\ud V_{g_{\Sn}}\right)
\omega_n^{\frac{2\sigma}{n}}A_r^{\frac{2(\sigma-n)}{n}}\ud f(p_r)\cdot p_r^-\\
&
+C\epsilon_r(\epsilon_r+\|v_r-1\|_{H^\sigma(\Sn)})|p_r^-|.
\end{align*}
Since $f(p_r)=f(p_i)+|p_r^+|^2-|p_r^-|^2$ in the local coordinates near $p_i$,
there holds $\ud f(p_r)\cdot p_r^-=-2|p_r^-|^2$.
Therefore, by combining $I$, $II$ and $III$, we conclude that
\begin{align*}
D\leq& \frac{n-2\sigma}{n}\omega_n^{\frac{2\sigma}{n}}A_r^{\frac{2\sigma-n}{n}}\Bigg\{
-\frac{4}{n-2}\left(\fint_{\Sn}v_rP_\sigma(v_r)\ud V_{g_{\Sn}}\right)
\epsilon_r(\epsilon_0-\epsilon_r) \frac{\Delta_{g_{\Sn}}f(p_r)}{A_r} \\
&
-2\left(\fint_{\Sn}v_rP_\sigma(v_r)\ud V_{g_{\Sn}}\right)
\frac{|p_r^-|^2}{A_r}+C\|v_r-1\|_{H^\sigma(\Sn)}|p_r^+|(\epsilon_r-\epsilon)\\
&-C_0\|v_r-1\|^2_{H^\sigma(\Sn)}+o(1)\Big(\epsilon_r(\epsilon_r-\epsilon_0)+\|v_r-1\|_{H^\sigma(\Sn)}^2+|p_r^-|^2\Big)\Bigg\}.
\end{align*}
Thus, we follow the same lines as in \cite[p. 499]{Chen&Xu} to conclude that for any $u \in B_{r_0}(p_i)\cap L_{\beta_i+\nu}$, if  $r_0$ is sufficiently small, then
$$\frac{\ud}{\ud r} E_f[u_r]<0.$$
This gives $$E_f[u_r]<E_f[u_r]|_{r=0}=E_f[u]\leq \beta_i+\nu.$$

With these preparations, we can proceed the argument in the proof of Proposition 7.1 (iv)
in \cite[pp. 499-500]{Chen&Xu} to finish the proof of part \textit{(iv)}.
\end{proof}

\vskip 48pt

\noindent X. Chen

\noindent Department of Mathematics \& IMS, Nanjing University, \\
Nanjing 210093, China\\[1mm]
Email: \textsf{xuezhangchen@nju.edu.cn}

\medskip

\noindent P. T. Ho

\noindent Department of Mathematics, Sogang University, \\
Seoul 121-742, Korea\\[1mm]
Korea Institute for Advanced Study, \\
Hoegiro 85, Seoul 02455, Korea\\[1mm]
Email: \textsf{paktungho@yahoo.com.hk, ptho@sogang.ac.kr}


\small

\begin{thebibliography}{30}

\bibitem{ACH} W. Abdelhedi, H. Chtioui and H. Hajaiej,
           A complete study of the lack of compactness and existence results of a fractional Nirenberg equation via a flatness hypothesis, I.
           \textit{Anal. PDE} 9 (2016), no. 6, 1285--1315. 
           
\bibitem{BFR} P. Baird,  A.  Fardoun and R. Regbaoui, 
            $Q$-curvature flow on 4-manifolds. 
            \textit{Calc. Var. Partial Differential Equations} \textbf{27} (2006), no. 1, 75--104. 

\bibitem{Beckner} W. Beckner,
                 Sharp Sobolev inequalities on the sphere and the Moser-Trudinger inequality.
                 \textit{Ann. of Math. (2)} \textbf{138} (1993),  213--242. 
                 
\bibitem{Bran85} T.P. Branson,  
           Differential operators canonically associated to a conformal structure. 
           \textit{Math. Scand.} \textbf{57} (1985), 293--345. 

\bibitem{Bran} T.P. Branson,             
           Sharp inequalities, the functional determinant, and the complementary series. 
           \textit{Trans. Amer. Math. Soc.} \textbf{347} (1995), 367--3742.  
           
 \bibitem{Brendle1}  S. Brendle, 
            Prescribing a higher order conformal invariant on $\Sn$. 
            \textit{Comm. Anal. Geom.} \textbf{11} (2003), 837--858.
            
 \bibitem{Brendle2}  S. Brendle,  
          Global existence and convergence for a higher order flow in conformal geometry. 
          \textit{Ann. of Math.}  \textbf{158} (2003), no. 1, 323--343.  
          
\bibitem{Brendle3} S. Brendle,  
         Convergence of the Yamabe flow for arbitrary initial energy.
         \textit{J. Differential Geom.} \textbf{69} (2005), 217--278. 

\bibitem{Brendle4} S. Brendle,  
         Convergence of the $Q$-curvature flow on $\mathbb{S}^4$. 
         \textit{Adv. Math} \textbf{205} (2006), no. 1, 1--32. 

\bibitem{CaS} L. Caffarelli and L. Silvestre,
              An extension problem related to the fractional Laplacian.
              \textit{Comm. Partial Differential Equations}  \textbf{32} (2007), no. 7-9, 1245--1260.

\bibitem{CSS} H. Chan, Y. Sire and L. Sun,
              Convergence of the fractional Yamabe flow for a class of initial data.
             \href{https://arxiv.org/abs/1809.05753}{arXiv:1809.05753.}
             
\bibitem{Chang} K. C. Chang, 
              Infinite Dimensional Morse Theory and Multiple Solution Problems. Birkh\"{a}user, Basel, 1993. 
              
\bibitem{Chang-G} S.-Y. Chang and M. Gonz\'alez, 
           Fractional Laplacian in conformal geometry. 
           \textit{Adv. Math.} \textbf{226} (2011), 1410--1432. 
              
\bibitem{CY87} S.-Y. Chang and P. Yang, 
            Prescribing Gaussian curvature on $\mathbb{S}^2$. 
            \textit{Acta Math.} \textbf{159} (1987), 215--259.

\bibitem{CY}   S.-Y. Chang and P. Yang, 
               A perturbation result in prescribing scalar curvature on $\Sn$. 
               \textit{Duke Math. J.} \textbf{64} (1991), no. 1, 27--69.

\bibitem{Chen&Xu1} X. Chen and X. Xu,
            $Q$-curvature flow on the standard sphere of even dimension. 
            \textit{J. Funct. Anal.} \textbf{261} (2011), no. 4, 934--980.

\bibitem{Chen&Xu} X. Chen and X. Xu,
                The scalar curvature flow on $\Sn$---perturbation theorem revisited.
                \textit{Invent Math.}  \textbf{187} (2012), no. 2, 395--506;  with an erratum: \textit{Invent Math.}  \textbf{187} (2012), no. 2, 507-509.
                
\bibitem{CZ}  G. Chen and Y. Zheng, 
        A perturbation result for the $Q_\gamma$ curvature problem on $\Sn$.
         \textit{Nonlinear Anal.} \textbf{97} (2014), 4--14.  

\bibitem{CLZ} Y.-H. Chen, C. Liu and Y. Zheng,
              Existence results for the fractional Nirenberg problem.
              \textit{J. Funct. Anal.} \textbf{270} (2016), no. 11, 4043--4086. 
              
\bibitem{DMR} F. Da Lio, L. Martinazzi and T.  Rivi\'ere, 
             Blow-up analysis of a nonlocal Liouville-type equation. 
             \textit{Anal. PDE} \textbf{8} (2015), no. 7, 1757--1805. 

\bibitem{DSV} P. Daskalopoulos, Y. Sire and Juan-Luis V\'azquez,
              Weak and smooth solutions for a fractional Yamabe flow: the case of
              general compact and locally conformally flat manifolds.
               \href{https://arxiv.org/abs/1702.05221}{arXiv:1702.05221.}
               
  \bibitem{DQRV}             
              A. de Pablo, F. Quir\'os, A. Rodr\'iguez and J. L. V\'azquez, A general fractional porous medium equation. \textit{Comm. Pure Appl. Math.} \textbf{65} (2012), no. 9, 1242--1284. 
              

\bibitem{EG}
J. Escobar and G. Garcia, Conformal metrics on the ball with zero scalar curvature and prescribed mean curvature on the boundary. \textit{J. Funct. Anal.} \textbf{211} (2004), no. 1, 71--152. 

\bibitem{FKS} E. B. Fabes, C. E. Kenig and R. P. Serapioni,
           The local regularity of solutions of degenerate elliptic equations.
           \textit{Comm. Partial Differential Equations} \textbf{7} (1982), 77--116. 
           
 \bibitem{FGon}
              Y. Fang and M. Gonz\'alez. Asymptotic behavior of Palais-Smale sequences associated with fractional Yamabe type equations. \textit{Pacific J. Math.} \textbf{278}, No. 2 (2015) 369-405.

\bibitem{FG85} C. Fefferman and C.R. Graham,            
           Conformal invariants. in Elie Cartan et les Mathematiques d'Aujourd'hui, Ast\'erisque (1985), Numero Hors Serie, 95--116.


\bibitem{FG} C. Fefferman and C.R. Graham, 
          The ambient metric. Annals of Mathematics Studies, 178. 
          Princeton University Press, Princeton, NJ, 2012. 
          
\bibitem{GQ} M. Gonz\'alez and J. Qing, 
            Fractional conformal Laplacians and fractional Yamabe problems. 
            \textit{Anal. PDE} \textbf{6} (2013), no. 7, 1535--1576.   
            
\bibitem{GW} M. Gonz\'alez and M. Wang, 
            Further results on the fractional Yamabe problem: the umbilic case. 
            \textit{J. Geom. Anal.} \textbf{28} (2018), no. 1, 22--60.
          
\bibitem{GP} R.A. Gover and L. Peterson,  
           Conformally invariant powers of the Laplacian, $Q$-curvature, and tractor calculus. 
           \textit{Comm. Math. Phys.} \textbf{235} (2003), no. 2, 339--378.  
           
\bibitem{GJMS} C.R. Graham, R. Jenne, L. Mason and G. Sparling, 
           Conformally invariant powers of the Laplacian, I: existence. 
            \textit{J. London Math. Soc.} \textbf{46}, 557--565.  

\bibitem{GZ} C.R. Graham and M. Zworski, 
             Scattering matrix in conformal geometry.
              \textit{Invent. Math.} \textbf{152} (2003), 89--118. 
              
\bibitem{G+} Y. Guo, J. Nie, M. Niu and Z. Tang,
           Local uniqueness and periodicity for the prescribed scalar curvature problem of 
           fractional operator in $\mathbb{R}^N$. 
           \textit{Calc. Var. Partial Differential Equations} \textbf{56} (2017), no. 4, Art. 118, 41 pp. 

\bibitem{GM} M. Gursky and A. Malchiodi, 
          A strong maximum principle for the Paneitz operator and a non-local flow for the $Q$-curvature. 
          \textit{J. Eur. Math. Soc. (JEMS)} \textbf{17} (2015), no. 9, 2137--2173.  

\bibitem{Ho} P. T. Ho,
               Prescribed $Q$-curvature flow on $\Sn$.
               \textit{J. Geom. Phys.}  \textbf{62}  (2012),  1233--1261.
               
\bibitem{Ho2} P. T. Ho, $Q$-curvature flow on $\Sn$. \textit{Comm. Anal. Geom.}  \textbf{18}  (2010), no. 4,
791--820.

\bibitem{Ho3} P. T. Ho,
Results of prescribing $Q$-curvature on $\Sn$. \textit{Arch. Math. (Basel)} \textbf{100} (2013), no. 1, 85--93.
               
\bibitem{J1} M. Ji, 
          Scalar curvature equation on $\Sn$. I. Topological conditions. 
          \textit{J. Differential Equations} \textbf{246} (2009), 749--787. 
          
\bibitem{J2} M. Ji, 
          Scalar curvature equation on $\Sn$. II. Analytic characterizations. 
          \textit{J. Differential Equations} \textbf{246} (2009), 788--818.   

\bibitem{Jin&Li&Xiong1} T. Jin, Y. Y. Li and J. Xiong,
            On a fractional Nirenberg problem, part I: blow up analysis and compactness of solutions.
            \textit{J. Eur. Math. Soc. (JEMS)}  \textbf{16}  (2014),  1111--1171.

\bibitem{Jin&Li&Xiong2} -----,
         On a fractional Nirenberg problem, Part II: existence of solutions.
         \textit{Int. Math. Res. Not. IMRN}  (2015),  1555--1589.

\bibitem{Jin&Li&Xiong3} -----,
         The Nirenberg problem and its generalizations: A unified approach.
         \textit{Math. Ann.} \textbf{369} (2017), no. 1-2, 109--151.

\bibitem{Jin&Xiong} T. Jin  and J. Xiong,
         A fractional Yamabe flow and some applications.
         \textit{J. Reine Angew. Math.}  \textbf{696}  (2014), 187--223.

\bibitem{JX}  T. Jin  and J. Xiong,
           Sharp constants in weighted trace inequalities on Riemannian manifolds.
           \textit{Calc. Var. Partial Differential Equations} \textbf{48} (2013), no. 3-4, 555--585.

\bibitem{Juhl}  A. Juhl, 
          Explicit formulas for GJMS-operators and $Q$-curvatures. 
          \textit{Geom. Funct. Anal.} \textbf{23} (2013), no. 4, 1278--1370. 

\bibitem{KMW}  S. Kim, M. Musso and J. Wei, 
          Existence theorems of the fractional Yamabe problem. 
          \textit{Anal. PDE} \textbf{11} (2018), no. 1, 75--113.  
          
\bibitem{L} Y.Y. Li, 
          Prescribing scalar curvature on $\Sn$ and related problems. I. 
          \textit{J. Differential Equations} \textbf{120} (1995), 319--410. 

\bibitem{LR} C. Liu and Q. Ren, 
          Multi-bump solutions for fractional Nirenberg problem. 
          \textit{Nonlinear Anal.} \textbf{171} (2018), 177--207. 
          
\bibitem{M} A. Malchiodi,
          The scalar curvature problem on $\Sn$: an approach via Morse theory.
          \textit{Calc. Var. Partial Differential Equations} \textbf{14} (2002), 429--445.
          
   

\bibitem{MS} A. Malchiodi and M. Struwe, 
            $Q$-curvature flow on $\mathbb{S}^4$. \textit{J. Differential Geom.}
           \textbf{73} (2006), 1--44.  
      
\bibitem{Mayer}            
           M. Mayer, A scalar curvature flow in low dimensions. \textit{Calc. Var. Partial Differential Equations} \textbf{56} (2017), no. 2, Paper No. 24, 41 pp.
           
           \bibitem{MN} M. Mayer and  C.  Ndiaye, 
     Fractional Yamabe problem on locally flat conformal infinities of Poincare-Einstein manifolds.  \href{https://arxiv.org/abs/1701.05919}{arXiv: 1701.05919}.   

\bibitem{Morpurgo} C. Morpurgo,
         Sharp inequalities for functional integrals and traces of conformally invariant operators.
         \textit{Duke Math. J.} \textbf{114} (2002),  477--553. 
         
\bibitem{Moser} J. Moser, 
            On a nonlinear problem in differential geometry. Dynamical systems (Proc. Sympos.,Univ. Bahia, 
            Salvador, 1971), 273--280. Academic Press, New York, 1973. 
            
 \bibitem{NSS}          
  C. Ndiaye, Y. Sire and L. Sun,  Uniformization Theorems: Between Yamabe and Paneitz. \href{https://arxiv.org/abs/1911.02680 }{arXiv: 1911.02680}.

\bibitem{NTW}  M. Niu, Z.  Tang and L. Wang, 
           Solutions for conformally invariant fractional Laplacian equations with multi-bumps centered in lattices. \textit{J. Differential Equations} \textbf{266} (2019), no. 4, 1756--1831. 

\bibitem{Paneitz} S. Paneitz,  
           A quartic conformally covariant differential operator for arbitrary pseudo-Riemannian manifolds. 
           SIGMA Symmetry Integrability Geom. Methods Appl. \textbf{4} (2008), Paper 036, 3 pp. 

\bibitem{Pavlov&Samko} P. Pavlov and S. Samko,
         A description of spaces $L^\alpha_p(S_{n-1})$ in terms of spherical hypersingular integrals (Russian). \textit{Dokl. Akad. Nauk SSSR} \textbf{276} (1984), 546--550; translation in \textit{Soviet Math. Dokl.} \textbf{29} (1984),  549--553.


\bibitem{Schwetlick&Struwe} H. Schwetlick and M. Struwe,
         Convergence of the Yamabe flow for ``large" energies.
         \textit{J. Reine Angew. Math.} \textbf{562} (2003), 59--100.
         
 \bibitem{Silvestre}
 L. Silvestre, Regularity of the obstacle problem for a fractional power of the Laplace operator. \textit{Comm. Pure Appl. Math.} \textbf{60} (2007), no. 1, 67--112.

\bibitem{S} M. Struwe,
         A flow approach to Nirenberg problem.
          \textit{Duke Math.J.} \textbf{128} (2005), 19--64. 

\bibitem{TX} J. Tan and J. Xiong,
         A Harnack inequality for fractional Laplace equations with lower order terms.
         \textit{Discrete Contin. Dyn. Syst.} \textbf{31} (2011), 975--983.
         
\bibitem{WX98} J. Wei and X. Xu, 
       On conformal deformations of metrics on $\Sn$. 
       \textit{J. Funct. Anal.} \textbf{157} (1998), 292--325. 
       
\bibitem{WX09} J. Wei and X. Xu,  
       Prescribing $Q$-curvature problem on $\Sn$. 
       \textit{J. Funct. Anal.} \textbf{257} (2009), 1995--2023.
         
\bibitem{XZ}       
      X. Xu and H. Zhang, Conformal metrics on the unit ball with prescribed mean curvature. \textit{Math. Ann.} \textbf{365} (2016), no. 1-2, 497--557.

\end{thebibliography}
\end{document}